\documentclass[11pt, a4paper,leqno]{amsart}
\usepackage{amsmath,amsthm,amscd,amssymb,amsfonts, amsbsy,mathrsfs}
\usepackage{latexsym}
\usepackage{exscale}
\usepackage{txfonts}
\usepackage{tikz}

\usepackage[colorlinks,citecolor=red,pagebackref,hypertexnames=false]{hyperref}

\newcommand{\cp}{{\rm cap}}

\usepackage{xcolor}

\calclayout
\allowdisplaybreaks

\theoremstyle{plain}
\newtheorem{theorem}[equation]{Theorem}
\newtheorem{lemma}[equation]{Lemma}

\newtheorem{claim}[equation]{Claim}
\newtheorem{mlemma}[equation]{Main Lemma} 

\theoremstyle{definition}
\newtheorem{definition}[equation]{Definition}

\theoremstyle{remark}
\newtheorem{remark}[equation]{Remark}

\numberwithin{equation}{section}
\usepackage{chngcntr} 
\counterwithin{figure}{section}

\newcommand{\RR}{{\mathbb{R}}}

\newcommand{\ZZ}{{\mathbb{Z}}}

\newcommand{\eps}{\varepsilon}

\newcommand{\dist}{\operatorname{dist}}

\newcommand{\re}{\mathbb{R}}
\newcommand{\rn}{\mathbb{R}^n}

\newcommand{\reu}{\mathbb{R}^{n+1}_+}
\newcommand{\ree}{\mathbb{R}^{n+1}}

\newcommand{\dd}{\mathcal{D}}

\newcommand{\sbf}{{\mathsf T}}
\newcommand{\C}{\mathcal{C}}

\newcommand{\om}{\Omega}

\newcommand{\F}{\mathcal{F}}

\newcommand{\M}{\mathcal{M}}

\newcommand{\W}{\mathcal{W}}

\newcommand{\B}{\mathcal{B}}

\newcommand{\G}{\mathcal{G}}

\newcommand{\mut}{\mathfrak{m}}

\newcommand{\pom}{\partial\Omega}

\newcommand{\hm}{\omega}

\newcommand{\sub}{{\bf U}^*}

\renewcommand{\P}{\mathcal{P}}

\renewcommand{\emptyset}{\mbox{\textup{\O}}}

\DeclareMathOperator{\supp}{supp}

\DeclareMathOperator{\diam}{diam}

\DeclareMathOperator{\interior}{int}

\newcommand{\DD}{{\mathcal D}}
\newcommand{\HH}{{H}}

\newcommand{\EE}{{\mathcal E}}
\newcommand{\WW}{{\mathcal W}}

\newcommand{\TT}{{\mathsf T}}

\newcommand{\rf}[1]{{\eqref{#1}}}

\newcommand{\vphi}{{\varphi}}

\newcommand{\vv}{{\vspace{2mm}}}

\newcommand{\wt}[1]{{\widetilde{#1}}}
\newcommand{\wh}[1]{{\widehat{#1}}}

\newcommand{\noi}{\noindent}

\newcommand{\fG}{{\mathsf G}}
\newcommand{\LD}{{\mathsf{LD}}}
\newcommand{\HD}{{\mathsf{HD}}}

\newcommand{\fS}{{\mathsf{S}}}
\newcommand{\fH}{{\mathsf{H}}}

\newcommand{\Stop}{{\mathsf{Stop}}}
\newcommand{\sstop}{{\mathsf{SubStop}}}

\newcommand{\ttt}{{\mathsf{Top}}}

\newcommand{\tree}{{\mathsf{T}}}

\newcommand{\stree}{{\mathsf{SubTree}}}

\newcommand{\eend}{{\mathsf{End}}}

\newcommand{\bdy}{{\mathsf{Bdy}}}

\newcommand{\WA}{{{\mathsf{WA}}}}
\newcommand{\WSBC}{{{\mathsf{WSBC}}}}

\newcommand{\Con}{{\mathrm{Con}}}
\newcommand{\NC}{{\mathsf{NC}}}

\def\XXint#1#2#3{{\setbox0=\hbox{$#1{#2#3}{\int}$ }
\vcenter{\hbox{$#2#3$ }}\kern-.58\wd0}}

\def\avint{\fint}

\def\zed{{x}}
\def\exe{{z}}
\def\ball{{B}}

\begin{document}
\allowdisplaybreaks

\title[Harmonic measure and quantitative connectivity]{
	Harmonic measure and quantitative connectivity: geometric characterization of the $L^p$-solvability of the Dirichlet problem}

\author[J. Azzam]{Jonas Azzam}

\address{Jonas Azzam\\
School of Mathematics \\ University of Edinburgh \\ JCMB, Kings Buildings \\
Mayfield Road, Edinburgh,
EH9 3JZ, Scotland}
\email{j.azzam@ed.ac.uk}

\author[S. Hofmann]{Steve Hofmann}

\address{Steve Hofmann
\\
Department of Mathematics
\\
University of Missouri
\\
Columbia, MO 65211, USA} \email{hofmanns@missouri.edu}

\author[J.M. Martell]{Jos\'e Mar{\'\i}a Martell}

\address{Jos\'e Mar{\'\i}a Martell\\
Instituto de Ciencias Matem\'aticas CSIC-UAM-UC3M-UCM\\
Consejo Superior de Investigaciones Cient{\'\i}ficas\\
C/ Nicol\'as Cabrera, 13-15\\
E-28049 Madrid, Spain} \email{chema.martell@icmat.es}

\author[M. Mourgoglou]{Mihalis Mourgoglou}

\address{Mihalis Mourgoglou\\
Departamento de Matem\'aticas, Universidad del Pa\' is Vasco, Barrio Sarriena s/n 48940 Leioa, Spain and\\
Ikerbasque, Basque Foundation for Science, Bilbao, Spain
}
\email{michail.mourgoglou@ehu.eus}

\author[X. Tolsa]{Xavier Tolsa}
\address{Xavier Tolsa
\\
ICREA, Passeig Llu\'{\i}­s Companys 23 08010 Barcelona, Catalonia, and\\
Departament de Matem\`atiques and BGSMath
\\
Universitat Aut\`onoma de Barcelona
\\
08193 Bellaterra (Barcelona), Catalonia
}
\email{xtolsa@mat.uab.cat}

\thanks{S.H. was 
supported by NSF grant DMS-1664047.
J.M.M. acknowledges financial support from the Spanish Ministry of Economy and Competitiveness, through the ``Severo Ochoa" Programme for Centres of Excellence in R\&D (SEV-2015- 0554). He also acknowledges that the research leading to these results has received funding from the European Research Council under the European Union's Seventh Framework Programme (FP7/2007-2013)/ ERC agreement no. 615112 HAPDEGMT. 
In addition, S.H. and J.M.M. were supported by NSF Grant DMS-1440140 while in residence at the MSRI in Berkeley, California, during Spring semester 2017. M.M. was supported  by IKERBASQUE and partially supported by the grant MTM-2017-82160-C2-2-P of the Ministerio de Econom\'ia y Competitividad (Spain), and by  IT-641-13 (Basque Government). X.T. was supported by the ERC grant 320501 of the European Research Council (FP7/2007-2013) and partially supported by MTM-2016-77635-P,  MDM-2014-044 (MICINN, Spain), 2017-SGR-395 (Catalonia), and by Marie Curie ITN MAnET (FP7-607647).
}

\date{July 4, 2019. \textit{Revised}\textup{:} \today}



\begin{abstract}
 It is well-known that
quantitative, scale invariant absolute continuity (more precisely, the weak-$A_\infty$ property) 
of harmonic measure with respect to surface measure, 
on the boundary of an open set $ \Omega\subset \ree$ with Ahlfors-David regular boundary,
 is equivalent to the 
solvability of the Dirichlet problem in $\Omega$, with data in $L^p(\pom)$ for some $p<\infty$.
In this paper, we give a geometric characterization
of the weak-$A_\infty$ property, of harmonic measure, and hence of solvability of the $L^p$
Dirichlet problem for some finite $p$.  This characterization is obtained under background
hypotheses (an interior corkscrew condition, along with Ahlfors-David regularity of the boundary) that are natural, and
in a certain sense optimal:  we provide counter-examples in the absence of either of them (or even one of the two,
upper or lower, Ahlfors-David bounds);  moreover, the examples show 
that the upper and lower Ahlfors-David  bounds are each quantitatively sharp.
\end{abstract}

\maketitle

\setcounter{tocdepth}{1}

\tableofcontents

\section{Introduction}\label{s1}
A classical criterion of Wiener characterizes the domains in which one can solve
the Dirichlet problem for Laplace's equation
with continuous boundary data, and with continuity of the solution up to the boundary.
In this paper, 
we address the analogous issue in the case of singular data.  To be more precise, the present work
provides a purely geometric characterization
of the open sets for which $L^p$ solvability holds, for some $p<\infty$, and with 
non-tangential convergence to the data a.e., thus allowing for singular boundary data.  We establish this 
characterization in the presence of background hypotheses (an interior corkscrew condition
[see Definition \ref{def1.cork} below], and Ahlfors-David regularity of the boundary [Definition \ref{defadr}]) 
that are in the nature of best possible, in 
the sense that there are counter-examples in the absence of either of them (or of even one of the 
two, upper or lower, Ahlfors-David bounds); moreover, the examples show 
that the upper and lower
Ahlfors-David  bounds are each quantitatively sharp (see the discussion following Theorem \ref{t1}, as well as
Appendix \ref{appa}, for more details.

Solvability of the $L^p$ Dirichlet problem 
 is fundamentally tied to quantitative absolute continuity of harmonic measure
with respect to surface measure on the boundary:
indeed, {\em it is equivalent} to the so-called ``weak-$A_\infty$" property of the harmonic measure 
(see Definition \ref{defAinfty}).  It is through 
this connection to quantitative absolute continuity of harmonic measure
that we shall obtain our geometric characterization of $L^p$ solvability.  

The study of the relationship between the geometry of a domain, and absolute continuity properties of its
harmonic measure, has a long history.
A classical result of F. and M. Riesz \cite{RR} states that for a simply 
connected domain $\Omega$ 
in the complex plane,
rectifiability of $\pom$ implies that harmonic measure for $\Omega$ is absolutely continuous with respect to
arclength measure on the boundary.  A quantitative version of this theorem was later
proved by Lavrentiev \cite{Lav}.
More generally, if only a portion of the boundary is rectifiable, Bishop and Jones
\cite{BiJo} have shown that harmonic measure
is absolutely continuous with respect to arclength on that portion.  They also present a counter-example
to show that the result of \cite{RR}
may fail in the absence of some connectivity hypothesis
(e.g., simple connectedness).

In dimensions greater than 2, a fundamental result of Dahlberg \cite{Dah} establishes a quantitative version
of absolute continuity, namely that 
harmonic measure belongs to the class $A_\infty$ in an appropriate local sense  
(see Definitions \ref{defAinfty} and \ref{deflocalAinfty} below),
with respect to surface measure
on the boundary of a Lipschitz domain.

The result of Dahlberg was extended to the class of Chord-arc domains (see Definition \ref{def1.cad})
by David and Jerison \cite{DJe}, and independently by Semmes \cite{Se}.    The Chord-arc hypothesis
was weakened to that of a two-sided corkscrew condition (Definition \ref{def1.cork}) 
by Bennewitz and Lewis
\cite{BL}, who then drew the conclusion that harmonic measure is weak-$A_\infty$ 
(in an appropriate local sense, see Definitions \ref{defAinfty} and \ref{deflocalAinfty})
with respect to 
surface measure on the boundary; the latter condition is similar to the $A_\infty$ condition, but without
the doubling property, and is the best conclusion that can be obtained under the weakened
geometric conditions considered in \cite{BL}.  We note that weak-$A_\infty$ is still a quantitative,
scale invariant version of absolute continuity.

More recently, one of us (Azzam)
has given in  \cite{Az} a geometric
characterization of the $A_\infty$ property of harmonic measure with respect to 
surface measure for domains with $n$-dimensional Ahlfors-David regular ($n$-ADR) boundary  (see Definition \ref{defadr}).  
Azzam's results are related to those of the present paper, so let us describe them in a bit more detail.  
Specifically, 
he shows that
for a domain $\Omega$ with $n$-ADR boundary, harmonic measure is in $A_\infty$ with respect
to surface measure, if and only if 1) $\pom$ is uniformly rectifiable ($n$-UR)\footnote{This 
is a quantitative, 
scale-invariant version of rectifiability, see Definition \ref{defur} and the 
ensuing comments.},
and 2) $\Omega$ is semi-uniform in the sense of Aikawa and Hirata \cite{AH}.  The semi-uniform condition
is a connectivity condition which states that for some uniform constant $M$,
every pair of points $x\in \Omega$ and $y\in \pom$ may 
be connected by a rectifiable curve $\gamma=\gamma(y,x)$, with
$\gamma \setminus \{y\} \subset \Omega$,
with length $\ell(\gamma) \leq M|x-y|$, 
and which satisfies the ``cigar path" condition
\begin{equation}\label{cigar}
\min\left\{\ell\big(\gamma(y,z)\big),\ell\big(\gamma(z,x)\big)\right\} \,\leq \, 
M\dist(z,\pom)\,,\quad \forall \, z\in \gamma\,.
\end{equation}
Semi-uniformity is a weak version of the well known uniform condition, whose definition is similar, except 
that it applies to all pairs of points $x,y\in \Omega$.   For example, the unit disk centered at the origin, 
with the slit 
$\{-1/2\leq x\leq 1/2, y=0\}$
removed, is semi-uniform, but not uniform.
It was shown in \cite{AH} that for a domain satisfying a John condition and the
Capacity Density Condition (in particular, for a domain with an $n$-ADR boundary), 
semi-uniformity characterizes the
doubling property of harmonic measure.    
The method of \cite{Az} is, broadly speaking, related to that of
\cite{DJe}, and of \cite{BL}.  In \cite{DJe}, the authors show that a Chord-arc domain $\Omega$ 
may be approximated in a ``Big Pieces" sense
(see \cite{DJe} or \cite{BL} for a precise statement; also cf. 
Definition \ref{def1.john} below) by Lipschitz subdomains $\Omega'\subset \Omega$; 
this fact allows one to reduce 
matters to the result of Dahlberg via
the maximum principle (a method which, to the present authors' knowledge, first appears in \cite{JK} in the 
context of $BMO_1$ domains).   
The same strategy, i.e., Big Piece approximation by  Lipschitz subdomains, is employed in \cite{BL}.
Similarly, in \cite{Az}, matters are reduced to the result of \cite{DJe}, by showing that
for a domain $\Omega$ with an $n$-ADR boundary,  $\Omega$ is 
semi-uniform with a uniformly rectifiable boundary
if and only if it has ``Very Big Pieces" of Chord-arc subdomains (see \cite{Az} for a precise statement of the latter condition).
As mentioned above, the converse direction is also treated in \cite{Az}.  In that case, given an 
interior corkscrew condition (which holds automatically in the presence of the doubling property of harmonic measure),
and provided that $\pom$ is $n$-ADR, the $A_\infty$ (or even weak-$A_\infty$) property of harmonic measure was 
already known to imply uniform rectifiability of the boundary \cite{HM-4} (although the published version appears
in \cite{HLMN}; see also \cite{MT} for an alternative proof, and a somewhat more general result); as in \cite{AH}, 
semi-uniformity
follows from the doubling property, although in \cite{Az}, the author manages to show this while dispensing with
the John domain background assumption (given a harmlessly strengthened version of the doubling property).

Thus, in \cite{Az}, the 
connectivity condition (semi-uniformity), is tied to the doubling property 
of harmonic measure, and not to absolute continuity.  On the other hand,
in light of the example of \cite{BiJo}, and on account of the aforementioned
connection to solvability of the Dirichlet problem, it has been an important 
open problem to 
determine the minimal 
connectivity assumption which, in conjunction with uniform rectifiability of the boundary, 
yields quantitative
absolute continuity of harmonic measure with respect to surface measure.
In the present work, we present a connectivity condition,
significantly milder than semi-uniformity, which we call the 
{\em weak local John condition} (see
Definition  \ref{def1.john} below), and which solves this problem.  Thus,
we obtain a geometric 
characterization of the domains for which one has
quantitative absolute continuity of harmonic measure;  equivalently, for which one has
solvability of the Dirichlet problem with singular ($L^p$) data
(see Theorem \ref{tmain} below).  
In fact, we provide two geometric characterizations of such domains, one in terms of uniform
rectifiability combined with the weak local John condition, the other in terms of 
approximation of the boundary
in a big pieces sense, by boundaries
of Chord-arc subdomains.

Let us now describe the weak local John condition, which says, roughly speaking, that from each point
$x\in \Omega$, there is local non-tangential access to an ample portion of a surface ball
at a scale on the order of $\delta_{\Omega}(x):=\dist(x,\pom)$.  Let us make this a bit more precise.
A ``carrot path" (aka non-tangential path) joining
a point $x\in \om$, and a point $y \in \pom$, is a
connected rectifiable path $\gamma=\gamma(y,x)$, with endpoints $y$ and $x$,
such that
for some $\lambda\in(0,1)$ and for all $z\in \gamma$,
\begin{equation}\label{eq1.2}
\lambda\, \ell\big(\gamma(y,z)\big)\, \leq\,  \delta_{\Omega}(z)\,,
\end{equation} 
where $\ell\big(\gamma(y,z)\big)$ denotes the arc-length of the portion of the original path
with endpoints $y$ and $z$.  For $x\in \Omega$, and $N\geq 2$, set
$$\Delta_x = \Delta_x^N := B\big(x,N\delta_{\Omega}(x)\big)\cap\pom\,.$$
We assume that  every
point $x\in \Omega$ may be joined by a carrot path to each $y$ in
a  ``Big Piece"  of $\Delta_x$, i.e., to each $y$
in a Borel subset $F \subset \Delta_x$, with $\sigma(F)\geq \theta\sigma(\Delta_x)$, where 
$\sigma$ denotes surface measure on $\pom$, and where the parameters
$N\geq 2$, $\lambda\in (0,1)$, and $\theta\in (0,1]$ are uniformly controlled.
We refer to this condition as a ``weak local John
condition", although ``weak local semi-uniformity" would be equally appropriate.
See Definitions \ref{def1.carrot}, \ref{def1.johnpoint}
and \ref{def1.john} for more details.    We remark that a strong version of the local John condition
(i.e., with $\theta =1$) has appeared in \cite{HMT}, in connection with boundary Poincar\'e inequalities 
for non-smooth domains.  

Let us observe that the weak local John condition is strictly weaker than semi-uniformity:
for example, the unit disk centered a the origin, with either the cross 
$\{-1/2\leq x\leq 1/2, y=0\}\cup \{-1/2\leq y\leq 1/2, x=0\}$ removed, or with the slit
$\{0\leq x\leq 1, y=0 \}$ removed,
satisfies the weak local John condition, although semi-uniformity fails in each case.

The main result in the present work  is 
the following geometric characterization of
quantitative absolute continuity of harmonic measure, and of the $L^p$ solvability
of the Dirichlet problem.  
The terminology used here will be defined in the sequel.

\begin{theorem}\label{tmain} 
Let $\Omega\subset \ree$, $n\ge 1$, be an open set satisfying an interior 
corkscrew condition (see Definition \ref{def1.cork} below), and 
suppose that $\pom$ is $n$-dimensional Ahlfors-David
regular ($n$-ADR; see Definition \ref{defadr} below).   Then the following are equivalent:
\begin{enumerate}
\item $\pom$ is Uniformly Rectifiable ($n$-UR; see Definition \ref{defur} below) 
and $\Omega$ satisfies the weak local John condition (see Definition \ref{def1.john} below).
\item $\Omega$ satisfies an Interior Big Pieces of 
Chord-Arc Domains (IBPCAD) condition
 (see Definition \ref{def1.ibpcad} below).
 
\item Harmonic measure $\hm$ is locally in weak-$A_\infty$ (see Definition \ref{deflocalAinfty} below)
with respect to surface measure $\sigma$ on $\pom$.

\item The $L^p$ Dirichlet problem is solvable for some $p<\infty$,
i.e., for some $p<\infty$, there is a constant $C$ such that if
$g\in L^p(\pom)$, then the solution
to the
Dirichlet problem with data $g$,  is well defined as
$u(x):=\int_{\pom} g d\hm^x$ for each $x\in \Omega$, converges to
$g$ non-tangentially, and enjoys
the estimate
\begin{equation}\label{eq1.Dirichlet}
\|N_*u\|_{L^{p}(\pom)}\, \leq \,C\, \|g\|_{L^p(\pom)}\,,
\end{equation}
where $N_*u$ is a suitable version of the non-tangential maximal function of $u$.  
\end{enumerate}
\end{theorem}

Some explanatory comments are in order.   The proof has two main new ingredients:
the implication  (1) implies (2), and the fact that the weak-$A_\infty$ property
of harmonic measure implies the weak local John condition
 (this is the new part of (3) implies (1)).  In turn, we split these main new results into two
 theorems: the first implication  is the content of Theorem \ref{t1} below, and the second is the 
 content of Theorem \ref{teo1a}.
We remark that the interior corkscrew
condition is not needed for  (1) implies (2) (nor for (2) implies (3) 
if and only if (4)).   
Rather, it is crucial for (3) implies (1) (see Appendix \ref{appa}). 

As regards the other implications,
the fact that
(2) implies (3) follows by a well-known argument using the maximum principle 
and the result of
\cite{DJe} and \cite{Se} for Chord-arc domains\footnote{See, e.g., \cite[Proposition 13]{H} 
for the details in this context, but the proof originates in \cite{JK}.}, along with the criterion for
weak-$A_\infty$ obtained in \cite{BL}; the equivalence of (3) and (4) is well known, and
we refer the reader to, e.g., 
\cite[Section 4]{HLe}, and to \cite{H} for details. 
The implication (3) implies (1) has two parts.
As mentioned above,
the fact that weak-$A_\infty$ implies weak local John is new, and is the content of Theorem
\ref{teo1a}.  The remaining implication, namely
 that weak-$A_\infty$ implies
$n$-UR, is the main result of \cite{HM-4}; an alternative proof, with a more general result, appears in \cite{MT}, and see also
\cite{HLMN} for the final published version of the results of \cite{HM-4}, along with an extension to the $p$-harmonic setting.

We note that our background hypotheses
(upper and lower $n$-ADR, and interior corkscrew) are in the nature of best possible:  
one may construct a counter-example in
the absence of any one of them, for at least one direction of this chain of implications,
as we shall discuss in Appendix \ref{appa}.  In addition, in the case of the $n$-ADR
condition, given any $\eps>0$, the counter-examples for the upper (respectively, lower) $n$-ADR property
can be constructed in such a way 
as to show that no weaker condition of the form $H^n(B(x,r)\cap \pom) \lesssim
r^{n -\eps}$ (resp., $H^n(B(x,r)\cap \pom) \gtrsim
r^{n+ \eps}$), with $r<1$, may be substituted for a true $n$-ADR upper or lower bound. 
Moreover, the first example shows that one cannot substitute 
the Capacity Density Condition (CDC)\footnote{The CDC is a
scale invariant potential theoretic
``thickness" condition, i.e., a quantitative version of Weiner regularity; see, e.g., \cite{AH}.} in place of 
the $n$-ADR condition:
indeed, the example is an NTA domain, in particular, it satisfies an exterior corkscrew condition, and thus also the CDC.

As regards our assumption of the interior corkscrew condition, 
we point out that, as is well known, 
the $n$-ADR condition implies that the open set 
$\ree\setminus \pom$ satisfies a corkscrew condition, with constants depending only on $n$ and ADR, i.e.,
 at every scale $r$, and for every point $x\in\pom$, 
 there is at least one component of $\ree\setminus \pom$ containing a corkscrew point relative to the ball $B(x,r)$.
Our last example shows that such a component should lie inside $\Omega$ itself, for 
 each $x$ and $r$; i.e., that $\Omega$ should enjoy an interior corkscrew condition.

As explained above, the main new contributions of the present work
are contained in the following pair of theorems,

\begin{theorem}\label{t1}  
	Let $\Omega \subset \ree$,  $n\ge 1$, be an open set, not necessarily connected, with an 
	$n$-dimensional
Ahlfors-David regular ($n$-ADR) boundary.
Then the following are equivalent:
\begin{itemize}
\item[(i)]
$\pom$ is uniformly rectifiable ($n$-UR), and 
$\Omega$ satisfies the weak local John condition.
\smallskip
\item[(ii)]  $\Omega$
	satisfies an Interior Big Pieces of Chord-Arc Domains (IBPCAD) 
	condition.
\end{itemize}	
\end{theorem}

Only the direction (i) implies (ii) is new.
For the converse,  the fact that IBPCAD implies the  
weak local John condition is immediate from the definitions.
Moreover, the boundary of a Chord-arc domain is 
$n$-UR, and an $n$-ADR set with big pieces of $n$-UR is also $n$-UR (see \cite{DS2}). 
As noted above, that (ii) implies the weak-$A_\infty$ property follows by 
well known arguments.

\begin{theorem}\label{teo1a}
Let $\Omega\subset \ree$, $n\geq1$, be an open set satisfying an interior 
corkscrew condition and 
suppose that $\pom$ is $n$-dimensional Ahlfors-David
regular ($n$-ADR). If the harmonic measure for $\Omega$ satisfies the
weak-$A_\infty$ condition, then $\Omega$ satisfies the weak local John condition.
\end{theorem}

Let us mention that the present paper is a combination of unpublished 
work of two different subsets of the present authors:  Theorem \ref{t1} is due to the
second and third authors, and was first posted in the draft manuscript
\cite{HM3}\footnote{An earlier version of this work \cite{HM4} gave a direct proof
of the fact that (1) implies (3) in Theorem \ref{tmain}, without passing through condition (2).};  Theorem \ref{teo1a}
is due to the first, fourth and fifth authors, and appeared first in the draft
manuscript \cite{AMT}.

The paper is organized as follows.  In  the next section, we set notation and give some 
definitions.    In Part \ref{part-1} of the paper (Sections \ref{s2}-\ref{s-step3}), we give the proof of Theorem \ref{t1}.
In Part \ref{part-2} of the paper (Sections \ref{s5}-\ref{sec8}) we give the proof of Theorem \ref{teo1a}.
Finally, in Appendix \ref{appa}, we discuss some counter-examples which show 
that our background hypotheses are in the nature of best possible.
\vspace{1mm}

We thank the referee for a careful reading of the paper, and for several helpful suggestions that have led us 
to clarify certain matters, and to make improvements in the presentation.

\section{Notation and definitions}

\begin{list}{$\bullet$}{\leftmargin=0.4cm  \itemsep=0.2cm}

\item Unless otherwise stated, we use the letters $c,C$ to denote harmless positive constants, not necessarily
the same at each occurrence, which depend only on dimension and the
constants appearing in the hypotheses of the theorems (which we refer to as the
``allowable parameters'').  We shall also
sometimes write $a\lesssim b$, $a \gtrsim b$, and $a \approx b$ to mean, respectively,
that $a \leq C b$, $a \ge c b$, and $0< c \leq a/b\leq C$, where the constants $c$ and $C$ are as above, unless
explicitly noted to the contrary.  In some occasions we will employ the notation $a\lesssim_\lambda b$, $a \gtrsim_\lambda b$ and $a \approx_\lambda b$ to emphasize that the previous implicit constants $c$ and/or $C$ may depend on some relevant parameter $\lambda$. At times, we shall designate by $M$ a particular constant whose value will remain unchanged throughout the proof of a given lemma or proposition, but
which may have a different value during the proof of a different lemma or proposition.

\item Our ambient space is $\ree$, $n\ge 1$.

\item $\Omega$ will always denote an open set in $\ree$,
not necessarily connected unless otherwise specified.

\item  We use the notation
$\gamma(x,y)$ to denote a rectifiable path with endpoints $x$ and $y$,
and its arc-length will be denoted $\ell(\gamma(x,y))$.  Given such a path,
if $z\in \gamma(x,y)$, we use the notation $\gamma(z,y)$ to denote the portion of the original path
with endpoints $z$ and $y$.  


\item We let $e_j$,  $j=1,2,\dots,n+1,$ denote the standard unit basis vectors in $\ree$.

\item The open $(n+1)$-dimensional Euclidean ball of radius $r$ will be denoted
$B(x,r)$.  For $x\in \pom$, a {\it surface ball} is denoted
$\Delta(x,r):= B(x,r) \cap\partial\Omega.$

\item Given a Euclidean ball $B$ or surface ball $\Delta$, its radius will be denoted
$r_B$ or $r_\Delta$, respectively.

\item Given a Euclidean or surface ball $B= B(x,r)$ or $\Delta = \Delta(x,r)$, its concentric
dilate by a factor of $\kappa >0$ will be denoted
$\kappa B := B(x,\kappa r)$ or $\kappa \Delta := \Delta(x,\kappa r).$

\item Given an open set $\om \subset \ree$, for $x \in \om$, we set $\delta_{\Omega}(x):= \dist(x,\pom)$.

\item We let $H^n$ denote $n$-dimensional Hausdorff measure, and let
$\sigma := H^n\big\lfloor_{\,\pom}$ denote the surface measure on $\pom$.

\item For a Borel set $A\subset \ree$, we let $\chi_A$ denote the usual
indicator function of $A$, i.e. $\chi_A(x) = 1$ if $x\in A$, and $\chi_A(x)= 0$ if $x\notin A$.

\item For a Borel set $A\subset \ree$,  we let $\interior(A)$ denote the interior of $A$.


\item Given a Borel measure $\mu$, and a Borel set $A$, with positive and finite $\mu$ measure, we
set $\fint_A f d\mu := \mu(A)^{-1} \int_A f d\mu$.

\item We shall use the letter $I$ (and sometimes $J$)
to denote a closed $(n+1)$-dimensional Euclidean dyadic cube with sides
parallel to the co-ordinate axes, and we let $\ell(I)$ denote the side length of $I$.
If $\ell(I) =2^{-k}$, then we set $k_I:= k$.
Given an $n$-ADR set $E\subset \ree$, we use $Q$ (or sometimes $P$ or $R$)
to denote a dyadic ``cube''
on $E$.  The
latter exist (see \cite{DS1}, \cite{Ch}, \cite{HK}), and enjoy certain properties
which we enumerate in Lemma \ref{lemmaCh} below.

\end{list}

\begin{definition}\label{defadr} ({\bf  $n$-ADR})  (aka {\it $n$-Ahlfors-David regular}).
We say that a  set $E \subset \ree$, of Hausdorff dimension $n$, is $n$-ADR
if it is closed, and if there is some uniform constant $C$ such that
\begin{equation} \label{eq1.$n$-ADR}
\frac1C\, r^n \leq \sigma\big(\Delta(x,r)\big)
\leq C\, r^n,\quad\forall r\in(0,\diam (E)),\ x \in E,
\end{equation}
where $\diam(E)$ may be infinite.
Here, $\Delta(x,r):= E\cap B(x,r)$ is the {\it surface ball} of radius $r$,
and as above, $\sigma:= H^n\lfloor_{\,E}$ 
is the ``surface measure" on $E$.
\end{definition}

\begin{definition}\label{defur} ({\bf $n$-UR}) (aka {\it $n$-uniformly rectifiable}).
An $n$-ADR (hence closed) set $E\subset \ree$
is $n$-UR if and only if it contains ``Big Pieces of
Lipschitz Images" of $\rn$ (``BPLI").   This means that there are positive constants $c_1$ and
$C_1$, such that for each
$x\in E$ and each $r\in (0,\diam (E))$, there is a
Lipschitz mapping $\rho= \rho_{x,r}: \rn\to \ree$, with Lipschitz constant
no larger than $C_1$,
such that 
$$
H^n\Big(E\cap B(x,r)\cap  \rho\left(\{z\in\rn:|z|<r\}\right)\Big)\,\geq\,c_1 r^n\,.
$$
\end{definition}

We recall that $n$-dimensional rectifiable sets are characterized by the
property that they can be
covered, up to a set of
$H^n$ measure 0, by a countable union of Lipschitz images of $\rn$;
we observe that BPLI  is a quantitative version
of this fact.

We remark
that, at least among the class of $n$-ADR sets, the $n$-UR sets
are precisely those for which all ``sufficiently nice" singular integrals
are $L^2$-bounded  \cite{DS1}.    In fact, for $n$-ADR sets
in $\ree$, the $L^2$ boundedness of certain special singular integral operators
(the ``Riesz Transforms"), suffices to characterize uniform rectifiability (see \cite{MMV} for the case $n=1$, and 
\cite{NToV} in general). 
We further remark that
there exist sets that are $n$-ADR (and that even form the boundary of a domain satisfying 
interior corkscrew and Harnack Chain conditions),
but that are totally non-rectifiable (e.g., see the construction of Garnett's ``4-corners Cantor set"
in \cite[Chapter 1]{DS2}).  Finally, we mention that there are numerous other characterizations of $n$-UR sets
(many of which remain valid in higher co-dimensions); cf. \cite{DS1,DS2}.

\begin{definition}\label{defurchar} ({\bf ``UR character"}).   Given an $n$-UR set $E\subset \ree$, its ``UR character"
is just the pair of constants $(c_1,C_1)$ involved in the definition of uniform rectifiability,
along with the ADR constant; or equivalently,
the quantitative bounds involved in any particular characterization of uniform rectifiability.
\end{definition}

\begin{definition} ({\bf Corkscrew condition}).  \label{def1.cork}
Following
\cite{JK}, we say that an open set $\Omega\subset \ree$
satisfies the {\it corkscrew condition} if for some uniform constant $c>0$ and
for every surface ball $\Delta:=\Delta(x,r),$ with $x\in \partial\Omega$ and
$0<r<\diam(\partial\Omega)$, there is a ball
$B(x_\Delta,cr)\subset B(x,r)\cap\Omega$.  The point $x_\Delta\subset \Omega$ is called
a {\it corkscrew point} relative to $\Delta.$  We note that  we may allow
$r<C\diam(\pom)$ for any fixed $C$, simply by adjusting the constant $c$.
In order to emphasize that $B(x_\Delta,cr) \subset \om$, we shall sometimes refer to this property as
the  {\it interior
corkscrew condition}.
\end{definition}

\begin{definition}({\bf Harnack Chains, and the Harnack Chain condition}  \cite{JK}).  \label{def1.hc} 
Given two points $x,x' \in \Omega$, and a pair of numbers $M,N\geq1$, 
an $(M,N)$-{\it Harnack Chain connecting $x$ to $x'$},  is a chain of
open balls
$B_1,\dots,B_N \subset \Omega$, 
with $x\in B_1,\, x'\in B_N,$ $B_k\cap B_{k+1}\neq \emptyset$
and $M^{-1}\diam (B_k) \leq \dist (B_k,\partial\Omega)\leq M\diam (B_k).$
We say that $\Omega$ satisfies the {\it Harnack Chain condition}
if there is a uniform constant $M$ such that for any two points $x,x'\in\om$,
there is an $(M,N)$-Harnack Chain connecting them, with $N$ depending only on $M$ and the ratio
$|x-x'|/\left(\min\big(\delta_{\Omega}(x),\delta_{\Omega}(x')\big)\right)$.  
\end{definition}

\begin{definition}({\bf NTA}). \label{def1.nta} Again following \cite{JK}, we say that a
domain $\Omega\subset \ree$ is NTA ({\it Non-tangentially accessible}) if it satisfies the
Harnack Chain condition, and if both $\Omega$ and
$\Omega_{\rm ext}:= \ree\setminus \overline{\Omega}$ satisfy the corkscrew condition.
\end{definition}

\begin{definition}({\bf CAD}). \label{def1.cad}  We say that a connected open set $\om \subset \ree$
is a CAD ({\it Chord-arc domain}), if it is NTA, and if $\pom$ is $n$-ADR.
\end{definition}

\begin{definition}({\bf Carrot path}).  \label{def1.carrot} Let $\Omega\subset \ree$ be an open set.
Given a point $x\in \om$, and a point $y \in \pom$, we say that 
a connected rectifiable path $\gamma=\gamma(y,x)$, with endpoints $y$ and $x$,
is a {\it carrot path} (more precisely, a {\it $\lambda$-carrot path}) connecting
$y$ to $x$, if $\gamma\setminus\{y\}\subset \om$, and if 
for some $\lambda\in(0,1)$ and for all $z\in \gamma$,
\begin{equation}\label{eq1.carrot}
\lambda\, \ell\big(\gamma(y,z)\big)\, \leq\,  \delta_{\Omega}(z)\,.
\end{equation}
With a slight abuse of terminology, we shall sometimes 
refer to such a path as a {\it $\lambda$-carrot path in}
$\om$, although of course the endpoint $y$ lies on $\pom$.
\end{definition}

A carrot path is sometimes referred to as a non-tangential path.

\begin{definition}({\bf $(\theta,\lambda,N)$-weak local John point}).  \label{def1.johnpoint} Let $x\in\om$, 
and for  constants $\theta\in (0,1]$, $\lambda\in (0,1)$, and $N\geq 2$, set
$$\Delta_x=\Delta_x^N:= B\big(x, N\delta_{\Omega}(x)\big)\cap\pom\,.$$
We say that a point $x\in\Omega$ is a {\it $(\theta,\lambda,N)$-weak local John point}
if there is a Borel set $F\subset \Delta^N_x$, with
$\sigma(F)\geq \theta\sigma(\Delta^N_x)$, such that for every $y\in F$, there is a 
$\lambda$-carrot path connecting $y$ to $x$.
\end{definition}

Thus, a weak local John point is non-tangentially connected to an ample portion of the boundary,
locally.  We observe that one can always choose $N$ smaller, for possibly different values of
$\theta$ and $\lambda$, by moving from $x$ to a point $x'$ on a line segment joining $x$ to the boundary.

\begin{remark}\label{remark2}
We observe that it is a slight abuse of notation to write
$\Delta_x$,
since the latter is not centered on $\pom$, and thus it is not a true surface ball; on the other hand,
there are true surface balls,  $\Delta'_x:=\Delta(\hat{x},(N-1)\delta_{\Omega}(x))$ and 
$\Delta''_x:=\Delta(\hat{x},(N+1)\delta_{\Omega}(x))$, centered at a ``touching point"
$\hat{x}\in\pom$ with $\delta_{\Omega}(x)=|x-\hat{x}|$,
which, respectively, are contained in, and contain, $\Delta_x$.   
\end{remark}

\begin{definition}({\bf Weak local John condition}).  \label{def1.john}
We say that $\om$ satisfies
a {\it weak local  John condition} if there are constants $\lambda\in (0,1)$,
$\theta\in(0,1]$, and $N\geq 2$, such that every $x\in\om$ is a $(\theta,\lambda,N)$-weak local John point.
\end{definition}

\begin{definition}({\bf IBPCAD}). \label{def1.ibpcad} We say that a connected open set $\om \subset \ree$
has {\it Interior Big Pieces of Chord-Arc Domains} (IBPCAD) if there exist positive constants $\eta$ and $C$, 
and $N\geq 2$, such that for every $x\in \Omega$, with $\delta_{\Omega}(x)<\diam(\pom)$,
there is a  Chord-arc domain $\Omega_x\subset \Omega$ satisfying
\begin{itemize}
\item $x\in \Omega_x$.
\item $\dist(x,\pom_x) \geq \eta \delta_{\Omega}(x)$.
\item $\diam(\Omega_x) \leq C\delta_{\Omega}(x)$.
\item $\sigma(\pom_x\cap \Delta^N_x) \geq \,\eta\, \sigma(\Delta^N_x) \,\approx_N\, \eta\,\delta_{\Omega}(x)^n$.
\item The Chord-arc constants of the domains $\Omega_x$ are uniform in $x$.
\end{itemize}
\end{definition}

\begin{remark}\label{remarkBPCAD}
In the presence of an interior corkscrew condition, Definition \ref{def1.ibpcad}
is easily seen to be essentially equivalent to the following more standard ``Big Pieces" condition: there are  positive constants $\eta$ and $C$
(perhaps slightly different to that in Definition \ref{def1.ibpcad}), such that
for each surface ball $\Delta:=\Delta(x,r) = B(x,r) \cap \pom$, $x\in \pom$ and
$r <\diam(\pom)$, and for any corkscrew point $x_\Delta$ relative to $\Delta$ there is a Chord-arc domain $\Omega_\Delta$ satisfying
\begin{itemize}
\item 	$x_\Delta\in \Omega_\Delta$

\item $\dist(x_\Delta,\pom_\Delta) \geq \eta r$.

\item $\Omega_\Delta\subset B(x,Cr) \cap\Omega$.

\item $\sigma(\pom_\Delta\cap \Delta(x,Cr)) \geq \,\eta\, \sigma(\Delta(x,Cr))\approx \eta r^n$.

\item The Chord-arc constants of the domains $\Omega_\Delta$ are uniform in $\Delta$.
\end{itemize}
\end{remark}

\begin{definition}\label{defAinfty}
({\bf $A_\infty$}, weak-$A_\infty$, and weak-$RH_q$). 
Given an $n$-ADR set $E\subset\ree$, 
and a surface ball
$\Delta_0:= B_0 \cap E$ centered on $E$,
we say that a Borel measure $\mu$ defined on $E$ belongs to
$A_\infty(\Delta_0)$ if there are positive constants $C$ and $s$
such that for each surface ball $\Delta = B\cap E$ centered on $E$, with $B\subseteq B_0$,
we have
\begin{equation}\label{eq1.ainfty}
\mu (A) \leq C \left(\frac{\sigma(A)}{\sigma(\Delta)}\right)^s\,\mu (\Delta)\,,
\qquad \mbox{for every Borel set } A\subset \Delta\,.
\end{equation}
Similarly, we say that $\mu \in$ weak-$A_\infty(\Delta_0)$ if 
for each surface ball $\Delta = B\cap E$ centered on $E$, with $2B\subseteq B_0$,
\begin{equation}\label{eq1.wainfty}
\mu (A) \leq C \left(\frac{\sigma(A)}{\sigma(\Delta)}\right)^s\,\mu (2\Delta)\,,
\qquad \mbox{for every Borel set } A\subset \Delta\,.
\end{equation}
We recall that, as is well known, the condition $\mu \in$ weak-$A_\infty(\Delta_0)$
is equivalent to the property that $\mu \ll \sigma$ in $\Delta_0$, and that for some $q>1$, the
Radon-Nikodym derivative $k:= d\mu/d\sigma$ satisfies
the weak reverse H\"older estimate
\begin{equation}\label{eq1.wRH}
\left(\fint_\Delta k^q d\sigma \right)^{1/q} \,\lesssim\, \fint_{2\Delta} k \,d\sigma\,
\approx\,  \frac{\mu(2\Delta)}{\sigma(\Delta)}\,,
\quad \forall\, \Delta = B\cap E,\,\, {\rm with} \,\, 2B\subseteq B_0\,,
\end{equation}
with $B$ centered on $E$.
We shall refer to the inequality in \eqref{eq1.wRH} as
a  ``weak-$RH_q$" estimate, and we shall say that $k\in$ weak-$RH_q(\Delta_0)$ if $k$ satisfies \eqref{eq1.wRH}.
\end{definition}

\begin{definition}\label{deflocalAinfty} ({\bf Local $A_\infty$ and local weak-$A_\infty$}).
We say that harmonic measure $\hm$  is locally in $A_\infty$ (resp., locally in
weak-$A_\infty$) on $\pom$,
if there are  uniform positive constants $C$ and $s$
such that for every ball $B=B(x,r)$ centered on $\pom$, 
with radius  $r<\diam(\pom)/4$, and associated surface ball $\Delta=B\cap\pom$,
\begin{equation}\label{eq1.localainfty}
\hm^p (A) \leq C \left(\frac{\sigma(A)}{\sigma(\Delta)}\right)^s\,\hm^p (\Delta)\,,
\qquad \forall \, p\in\om\setminus 4B\,, \,\,\forall \mbox{ Borel } A\subset \Delta\,,
\end{equation}
or, respectively, that
\begin{equation}\label{eq1.localwainfty}
\hm^p (A) \leq C \left(\frac{\sigma(A)}{\sigma(\Delta)}\right)^s\,\hm^p (2\Delta)\,,
\qquad \forall \, p\in\om\setminus 4B\,, \,\,\forall \mbox{ Borel } A\subset \Delta\,;
\end{equation}
equivalently,
if for every ball $B$ and surface ball $\Delta=B\cap\pom$ as above,
and for each  
point $p\in\om\setminus 4B$, $\hm^p\in$ $A_\infty(\Delta)$ (resp.,
$\hm^p\in$ weak-$A_\infty(\Delta)$) with uniformly controlled $A_\infty$ (resp., weak-$A_\infty$) constants. 
\end{definition}

\begin{lemma}\label{lemmaCh}\textup{({\bf Existence and properties of the ``dyadic grid''})
\cite{DS1,DS2, Ch}.}
Suppose that $E\subset \ree$ is an $n$-ADR set.  Then there exist
constants $ a_0>0,\, s>0$ and $C_1<\infty$, depending only on $n$ and the
ADR constant, such that for each $k \in \mathbb{Z},$
there is a collection of Borel sets (``cubes'')
$$
\dd_k:=\{Q_{j}^k\subset E: j\in \mathfrak{I}_k\},$$ where
$\mathfrak{I}_k$ denotes some (possibly finite) index set depending on $k$, satisfying

\begin{list}{$(\theenumi)$}{\usecounter{enumi}\leftmargin=.8cm
\labelwidth=.8cm\itemsep=0.2cm\topsep=.1cm
\renewcommand{\theenumi}{\roman{enumi}}}

\item $E=\cup_{j}Q_{j}^k\,\,$ for each
$k\in{\mathbb Z}$.

\item If $m\geq k$ then either $Q_{i}^{m}\subset Q_{j}^{k}$ or
$Q_{i}^{m}\cap Q_{j}^{k}=\emptyset$.

\item For each $(j,k)$ and each $m<k$, there is a unique
$i$ such that $Q_{j}^k\subset Q_{i}^m$.

\item $\diam\big(Q_{j}^k\big)\leq C_1 2^{-k}$.

\item Each $Q_{j}^k$ contains some ``surface ball'' $\Delta \big(x^k_{j},a_02^{-k}\big):=
B\big(x^k_{j},a_02^{-k}\big)\cap E$.

\item $H^n\big(\big\{x\in Q^k_j:{\rm dist}(x,E\setminus Q^k_j)\leq \vartheta \,2^{-k}\big\}\big)\leq
C_1\,\vartheta^s\,H^n\big(Q^k_j\big),$ for all $k,j$ and for all $\vartheta\in (0,a_0)$.
\end{list}
\end{lemma}

A few remarks are in order concerning this lemma.

\begin{list}{$\bullet$}{\leftmargin=0.4cm  \itemsep=0.2cm}

\item In the setting of a general space of homogeneous type, this lemma has been proved by Christ
\cite{Ch} (see also \cite{HK}), with the
dyadic parameter $1/2$ replaced by some constant $\delta \in (0,1)$.
In fact, one may always take $\delta = 1/2$ (see  \cite[Proof of Proposition 2.12]{HMMM}).
In the presence of the Ahlfors-David
property \eqref{eq1.$n$-ADR}, the result already appears in \cite{DS1,DS2}. Some predecessors of this construction have appeared in \cite{David88} and \cite{David91}.

\item  For our purposes, we may ignore those
$k\in \mathbb{Z}$ such that $2^{-k} \gtrsim {\rm diam}(E)$, in the case that the latter is finite.

\item  We shall denote by  $\dd=\dd(E)$ the collection of all relevant
$Q^k_j$, i.e., $$\dd := \cup_{k} \dd_k,$$
where, if $\diam (E)$ is finite, the union runs
over those $k$ such that $2^{-k} \lesssim  {\rm diam}(E)$.

\item Properties $(iv)$ and $(v)$ imply that for each cube $Q\in\dd_k$,
there is a point $x_Q\in E$, a Euclidean ball $B(x_Q,r_Q)$ and a surface ball
$\Delta(x_Q,r_Q):= B(x_Q,r_Q)\cap E$ such that
$r_Q\approx 2^{-k} \approx {\rm diam}(Q)$
and \begin{equation}\label{cube-ball}
\Delta(x_Q,r_Q)\subset Q \subset \Delta(x_Q,Cr_Q),\end{equation}
for some uniform constant $C$. We shall refer to the point $x_Q$ as the ``center'' of $Q$.

\item For a dyadic cube $Q\in \dd_k$, we shall
set $\ell(Q) = 2^{-k}$, and we shall refer to this quantity as the ``length''
of $Q$. 
Evidently, by adjusting if necessary some parameters, we can assume that
$$\diam(Q)\leq \ell(Q)\lesssim \diam(Q).$$
We shall denote 
\begin{equation}\label{cube-ball2}
B_Q:= B(x_Q,4\ell(Q)) \,,\qquad\Delta_Q:= \Delta(x_Q,4\ell(Q)).
\end{equation}
Notice that $Q\subset \Delta_Q\subset B_Q$. 

\item For a dyadic cube $Q \in \dd$, we let $k(Q)$ denote the dyadic generation
to which $Q$ belongs, i.e., we set  $k = k(Q)$ if
$Q\in \dd_k$; thus, $\ell(Q) =2^{-k(Q)}$.

\item 
Given $R\in\DD$, we set
\begin{equation}\label{eq3.4a}
\DD(R):=\{Q\in\DD:Q\subset R\}\,.
\end{equation}
For $j\ge 1$, we also let
\begin{equation}\label{eq3.4aaa}
\dd_j(R):=\left\{Q\in\dd(R): \, \ell(Q)=2^{-j}\,\ell(R)\right\}\,.
\end{equation}

\item For a pair of cubes $Q',Q \in \dd$, if $Q'$ is a dyadic child of $Q$,
i.e., if $Q'\subset Q$, and $\ell(Q) =2\ell(Q')$, then we write $Q'\lhd Q$.

\item For $\lambda>1$, we write
$$\lambda Q = \bigl\{x\in E:\, \dist(x,Q)\leq (\lambda-1)\,\ell(Q)\bigr\}.$$

\end{list}

\vv

With the dyadic cubes in hand, we may now define the notion of a corkscrew point relative to a cube $Q$.

\begin{definition}({\bf Corkscrew point relative to $Q$}).  \label{def1.CScube}
Let $\om$ satisfy the corkscrew condition (Definition \ref{def1.cork}), suppose that $\pom$ is $n$-ADR,
and let $Q\in \dd(\pom)$.
A {\it corkscrew point relative to $Q$} is simply a corkscrew point relative to the surface ball
$\Delta(x_Q,r_Q)$ defined in \eqref{cube-ball}.
\end{definition}

\begin{definition}({\bf Coherency and Semi-coherency}). \cite{DS2}.
\label{d3.11}   
Let $E\subset \ree$ be an $n$-ADR set.
Let $\sbf\subset \dd(E)$. We say that $\sbf$ is
{\it coherent} if the following conditions hold:
\begin{itemize}\itemsep=0.1cm

\item[$(a)$] $\sbf$ contains a unique maximal element $Q(\sbf)$ which contains all other elements of $\sbf$ as subsets.

\item[$(b)$] If $Q$  belongs to $\sbf$, and if $Q\subset \widetilde{Q}\subset Q(\sbf)$, then 
$\widetilde{Q}\in {\sbf}$.

\item[$(c)$] Given a cube $Q\in \sbf$, either all of its children belong to $\sbf$, or none of them do.

\end{itemize}
We say that $\sbf$ is {\it semi-coherent} if conditions $(a)$ and $(b)$ hold.  We shall refer to a coherent or
semi-coherent collection $\sbf$ as a {\em tree}.
\end{definition}



\newcounter{parte}

\bigskip
\begin{center} 
\Large Part \refstepcounter{parte}\theparte\label{part-1}: Proof of Theorem \ref{t1}
\end{center}
\smallskip

\addcontentsline{toc}{section}{Part 1: Proof of Theorem \ref{t1}}

\section{Preliminaries for the Proof of Theorem \ref{t1}}\label{s2}

We begin by recalling a bilateral version of the
David-Semmes ``Corona decomposition" of an $n$-UR set.   We refer the reader to \cite{HMM} for the proof.

\begin{lemma}\label{lemma2.1}\textup{(\cite[Lemma 2.2]{HMM})}  
Let $E\subset \ree$ be an $n$-UR set.  Then given any positive constants
$\eta\ll 1$
and $K\gg 1$, there is a disjoint decomposition
$\dd(E) = \G\cup\B$, satisfying the following properties.
\begin{enumerate}
\item  The ``Good" collection $\G$ is further subdivided into
disjoint  trees, such that each such tree $\sbf$ is coherent (Definition \ref{d3.11}).

\smallskip
\item The ``Bad" cubes, as well as the maximal cubes $Q(\sbf)$, $\sbf\subset\G$, satisfy a Carleson
packing condition:
$$\sum_{Q'\subset Q, \,Q'\in\B} \sigma(Q')
\,\,+\,\sum_{\sbf\subset\G: Q(\sbf)\subset Q}\sigma\big(Q(\sbf)\big)\,\leq\, C_{\eta,K}\, \sigma(Q)\,,
\quad \forall Q\in \dd(E)\,.$$

\smallskip
\item For each $\sbf\subset\G$, there is a Lipschitz graph $\Gamma_{\sbf}$, with Lipschitz constant
at most $\eta$, such that, for every $Q\in \sbf$,
\begin{equation}\label{eq2.2a}
\sup_{x\in \Delta_Q^*} \dist(x,\Gamma_{\sbf} )\,
+\,\sup_{y\in B_Q^*\cap\Gamma_{\sbf}}\dist(y,E) < \eta\,\ell(Q)\,,
\end{equation}
where $B_Q^*:= B(x_Q,K\ell(Q))$ and $\Delta_Q^*:= B_Q^*\cap E$, and $x_Q$ is the ``center"
of $Q$ as in \eqref{cube-ball}-\eqref{cube-ball2}.
\end{enumerate}
\end{lemma}

We remark that in \cite{HMM}, the trees $\sbf$  were denoted by ${\bf S}$, and were called
``stopping time regimes" rather than trees.

We mention that David and Semmes, in \cite{DS1}, had
previously proved a unilateral version of Lemma \ref{lemma2.1}, 
in which the bilateral estimate \eqref{eq2.2a} 
is replaced by the unilateral bound
\begin{equation}\label{eq2.5}
\sup_{x\in \Delta_Q^*} \dist(x,\Gamma_{\sbf} )\,
<\, \eta\,\ell(Q)\,,\qquad \forall\,Q\in \sbf\,.
\end{equation} 

Next, we make a standard Whitney decomposition of $\Omega_E:=\ree\setminus E$, for a given $n$-UR set $E$
(in particular, 
$\om_E$ is open, since $n$-UR sets are closed by definition).
Let $\mathcal{W}=\W(\om_E)$ denote a collection
of (closed) dyadic Whitney cubes of $\om_E$, so that the cubes in $\mathcal{W}$
form a pairwise non-overlapping covering of $\om_E$, which satisfy
\begin{equation}\label{Whintey-4I}
4 \diam(I)\leq
\dist(4I,E)\leq \dist(I,E) \leq 40\diam(I)\,,\qquad \forall\, I\in \mathcal{W}\,\end{equation}
(just dyadically divide the standard Whitney cubes, as constructed in  \cite[Chapter VI]{St},
into cubes with side length 1/8 as large)
and also
$$\tfrac14\diam(I_1)\leq\diam(I_2)\leq 4\diam(I_1)\,,$$
whenever $I_1$ and $I_2$ touch.

We fix a small parameter $\tau_0>0$, so that
for any $I\in \W$, and any $\tau \in (0,\tau_0]$,
the concentric dilate
\begin{equation}\label{whitney1}
I^*(\tau):= (1+\tau) I
\end{equation} 
still satisfies the Whitney property
\begin{equation}\label{whitney}
\diam I\approx \diam I^*(\tau) \approx \dist\left(I^*(\tau), E\right) \approx \dist(I,E)\,, \quad 0<\tau\leq \tau_0\,.
\end{equation}
Moreover,
for $\tau\leq\tau_0$ small enough, and for any $I,J\in \W$,
we have that $I^*(\tau)$ meets $J^*(\tau)$ if and only if
$I$ and $J$ have a boundary point in common, and that, if $I\neq J$,
then $I^*(\tau)$ misses $(3/4)J$.


Pick two parameters $\eta\ll 1$ and $K\gg 1$ (eventually, we shall take
$K=\eta^{-3/4}$).   For $Q\in \dd(E)$, define
\begin{equation}\label{eq3.1}
\W^0_Q:= \left\{I\in \W:\,\eta^{1/4} \ell(Q)\leq \ell(I)
 \leq K^{1/2}\ell(Q),\ \dist(I,Q)\leq K^{1/2} \ell(Q)\right\}.
 \end{equation}
 
\begin{remark}\label{remark:E-cks} 
We note that $\W^0_Q$ is non-empty,
 provided that we choose $\eta$ small enough, and $K$ large enough, depending only on dimension and ADR,
since the $n$-ADR condition implies that $\om_E$ satisfies a corkscrew condition.  In the sequel, we shall always
assume that $\eta$ and $K$ have been so chosen.
 \end{remark}
 
Next, we recall a construction in  \cite[Section 3]{HMM}, leading up to and including in particular
\cite[Lemma 3.24]{HMM}.   We summarize this construction as follows. 
\begin{lemma}\label{lemma2.7}
Let $E\subset \ree$ be 
$n$-UR, and set $\om_E:= \ree\setminus E$.  Given positive constants
$\eta\ll 1$
and $K\gg 1$, as in \eqref{eq3.1} and Remark \ref{remark:E-cks},  
let $\dd(E) = \G\cup\B$, be the corresponding 
bilateral Corona decomposition of Lemma \ref{lemma2.1}. 
Then for each $\sbf\subset \G$, and for each $Q\in \sbf$, the collection 
$\W^0_Q$ in \eqref{eq3.1} has an augmentation $\W^*_Q\subset \W$ satisfying the following properties.
\begin{enumerate}
\item $\W^0_Q\subset \W^*_Q = \W_Q^{*,+}\cup\W_Q^{*,-}$,
where (after a suitable rotation of coordinates)
each $I \in \W_Q^{*,+}$ lies above the Lipschitz graph $\Gamma_{\sbf}$
of Lemma \ref{lemma2.1},  each $I \in \W_Q^{*,-}$ lies below $\Gamma_{\sbf}$.
Moreover, if $Q'$ is a child of $Q$, also belonging to
$\sbf$, then $\W_Q^{*,+}$ (resp. $\W_Q^{*,-}$) belongs to the same connected
component of
$\om_E$ as does $\W_{Q'}^{*,+}$ (resp. $\W_{Q'}^{*,-}$)
and  $\W_{Q'}^{*,+}\cap \W_{Q}^{*,+}\neq \emptyset$ (resp.,
$\W_{Q'}^{*,-}\cap\W_{Q}^{*,-}\neq \emptyset$). 

\smallskip
\item There are uniform constants $c$ and $C$ such that
\begin{equation}\label{eq2.whitney2}
\begin{array}{c}
c\eta^{1/2} \ell(Q)\leq \ell(I) \leq CK^{1/2}\ell(Q)\,, \quad \forall I\in \mathcal{W}^*_Q,
\\[5pt]
\dist(I,Q)\leq CK^{1/2} \ell(Q)\,,\quad\forall I\in \mathcal{W}^*_Q,
\\[5pt]
c\eta^{1/2} \ell(Q)\leq\dist(I^*(\tau),\Gamma_{\sbf})\,,\quad \forall I\in \mathcal{W}^*_Q\,,\quad \forall 
\tau\in (0,\tau_0]\,.
\end{array}
\end{equation}
\end{enumerate}

Moreover, given $\tau\in(0,\tau_0]$, set
\begin{equation}\label{eq3.3aa}
U^\pm_Q=U^\pm_{Q,\tau}:= \bigcup_{I\in \W^{*,\pm}_Q} {\rm int}\left(I^*(\tau)\right)\,,\qquad U_Q:= U_Q^+\cup U_Q^-\,,
\end{equation}
and given $\sbf'$, a semi-coherent subtree of $\sbf$, define 
\begin{equation}\label{eq3.2}
\Omega_{\sbf'}^\pm = \Omega_{\sbf'}^\pm(\tau) := \bigcup_{Q\in\sbf'} U_Q^{\pm}\,.
\end{equation}
Then 
each of $\Omega^\pm_{\sbf'}$ is a CAD, with Chord-arc constants
depending only on $n,\tau,\eta, K$, and the
ADR/UR constants for $\pom$ (see Figure \ref{figure:domains}).  
\end{lemma}

\begin{figure}[!ht]
	\begin{tikzpicture}[scale=.45]

\newcommand{\mysetE}{
	{ plot [smooth]  coordinates { (-12, -2) (-10, 1.2) (-8,- 1.1) (-5.8, 0.1)  (-5, 0) (-4.4, -0.05)  (-4, 0.1) (-3, 0.35) (-2, 0.3) (-1, 0.1) (0, -0.25) (.4, 0.85) 
			(.8, 0.2) (1.7, 0.5)
			(2,0.54) (3, 0.7) (4,0.6)  (4.3,0.7) (4.5,0.4) (5,0.8) (5.9,-0.1) (7.4, 1.5)
			(7.8, -1.4)
			(9, 0.3)
			(9.5, -0.2) 
			(10, 0.3)  
			(11, -0.7)
			(12, 1) 
			(13,-2)
			(14.5,1.5)}}}

\begin{scope}[yshift=-7, xshift=2]

\path[clip,yshift=7, xshift=-2] (-13,12)--(-13,-5)-- (-12,-2) \mysetE -- (15,3)--(15,12)--(-13,12);

\draw [clip] 
(-14,12)--
(-14,-5)--
(-1.53125,-5)--
(-1.53125,0.46875)
--++(.03125,0) --++(0,.03125)
--++(.0625,0)--++(0,.0625)--++(.0625,0)--++(0,.0625)
--++(.125,0)--++(0,.125)
--++(.25,0)--++(0,.25)
--++(.5,0)--++(0,.5)--++(.5,0)--++(0,.5) --++(.5,0)--++(0,-.5)--++(.5,0)--++(0,-.25)--++(.25,0)--++(0,-.25)--++(.25,0)--++(0,-.125)--++(.125,0)--++(0,-.0625)--++(.0625,0)--++(0,-.03125)--++(.03125,0)--++(0,-.03125)
--++(0,-6)--++(14,0)--++(0,17);

\draw[blue, fill=blue!10] 
(3.65625,-5)--(3.65625,.875)--++(0,+.03125)
--++(+.03125,0)--++(0,.03125)
--++(.0625,0)--++(0,.0625)--++(.0625,0)--++(0,.0625)--++(.0625,0)--++(0,.0625)
--++(.125,0)--++(0,.125)
--++(.25,0)--++(0,.25)
--++(.5,0)--++(0,.5)--++(.5,0)--++(0,.5)--++(.5,0)--++(0,.5)--++(1,0)--++(0,1)
--++(1,0)--++(0,-1)--++(0.5,0)--++(0,-0.5)	--++(0.5,0)--++(0,-0.5)	--++(0.25,0)--++(0,-0.25) --++(0.25,0)--++(0,-0.25)
--++(0.25,0)--++(0,0.25) --++(0.25,0)--++(0,0.25)--++(0.25,0)--++(0,0.25)--++(0.5,0)--++(0,0.5)--++(0.5,0)--++(0,0.5)--++(0.5,0)--++(0,0.5)				 				
--++(1,0)--++(0,1)
--++(2,0)--++(0,4)--++(-8,0)--++(0,2)--++(-8,0)--++(0,-2)--++(-2,0)--
(-12.75,8.75)--++(0,-2)--++(2,0)--++(0,-2)--++(2,0)--++(0,-1)--++(1,0)--++(0,-1)--++(1,0)--++(0,-1)--++(1,0)--++(0,-.5)--++(.5,0)--++(0,-.5)--++(.5,0)
--++(0,-.25)--++(.25,0)--++(0,-.125)--++(.125,0)--++
(0,-.0625)--++(.0625,0)--++(0,-.03125)--++(.03125,0)--++(0,-5);

\draw [blue] 
(-14,12)--
(-14,-5)--
(-1.53125,-5)--
(-1.53125,0.46875)
--++(.03125,0) --++(0,.03125)
--++(.0625,0)--++(0,.0625)--++(.0625,0)--++(0,.0625)
--++(.125,0)--++(0,.125)
--++(.25,0)--++(0,.25)
--++(.5,0)--++(0,.5)--++(.5,0)--++(0,.5) --++(.5,0)--++(0,-.5)--++(.5,0)--++(0,-.25)--++(.25,0)--++(0,-.25)--++(.25,0)--++(0,-.125)--++(.125,0)--++(0,-.0625)--++(.0625,0)--++(0,-.03125)--++(.03125,0)--++(0,-.03125)
--++(0,-6)--++(14,0)--++(0,19);

\node[blue] at (2,5) {$\Omega^+_{{\mathsf T}'}$};

\end{scope}

\begin{scope}[yshift=-7, xshift=2]

\path[clip, yshift=7, xshift=-2] (-13,-12)--(-13,-5)-- (-12,-2)  \mysetE -- (15.5,-12)--(-13,-12);

\draw[clip] 
(-14,4)--(-1.53125,4)--
(-1.53125,0.46875)--++(0,-.03125)
--++(.03125,0)--++(0,-.03125)--++(.03125,0)--++(0,-.03125)--++(.03125,0)--++(0,-.03125)
--++(.0625,0)--++(0,-.0625)--++(.0625,0)--++(0,-.0625)
--++(.125,0)--++(0,-.125)
--++(.25,0)--++(0,-.25)--++(.25,0)--++(0,-.25)--++(.25,0)--++(0,-.25)--++(.25,0)--++(0,-.25)--++(0.25,0)
--++(0,0.5)--++(0.5,0)--++(0,0.25)--++(0.25,0)--++(0,0.25)--++(0.25,0)--++(0,0.125)--++(0.125,0)--++(0,0.125)--++(0.125,0)--++(0,0.125)--++(0.125,0)--++(0,0.0625)--++(0.0625,0)--++(0,0.0625)--++(0.0625,0)--++(0,0.0625)--++(0.0625,0)--++(0,0.0625)--++(0.0625,0)--++(0,0.03125)--++(0.03125,0)--++(-.03125,0)--++(0,-.03125)--++(0,4)--++(15,0)
--++(0,-20)--++(-30,0);

\draw[red, fill=red!10] 
(-4.28125,2)--
(-4.28125,0.28125)--++(0,-.03125)	
--++(-.03125,0)--++(0,-.03125)--++(-.03125,0)--++(0,-.03125)--++(-.03125,0)--++(0,-.03125)
--++(-.0625,0)--++(0,-.0625)--++(-.0625,0)--++(0,-.0625)
--++(-.125,0)--++(0,-.125)--++(-.125,0)--++(0,-.125)--++(-.125,0)--++(0,-.125)
--++(-.25,0)--++(0,-.25)--++(-.25,0)--++(0,-.25)--++(-.25,0)--++(0,-.25)
--++(-.5,0)--++(0,-.5)--++(-.5,0)--++(0,-.5)
--++(-1,0)--++(0,-1)--++(-1,0)--++(0,-1)
--++(-2,0)--++(0,-2)--++(-2,0)--++(0,-2)
--++(5.24625,0)
--++(0,-2)
--++(14,0)
--++(0,2)
--++(8,0)
--++(0,4)--++(-4,0)
--++(0,1)--++(-2,0)--++(0,1)--++(-2,0)
--++(0,1)--++(-1,0)
--++(0,.5)--++(-.5,0)--++(0,.5)--++(-.5,0)
--++(0,.25)--++(-.25,0)--++(0,.25)--++(-.25,0)
--++(0,.125)--++(-.125,0)--++(0,.125)--++(-.125,0)
--++(0,.0625)--++(-.0625,0)--++(0,.0625)--++(-.0625,0)
--++(0,.03125)--++(-.03125,0)--++(0,.03125)--++(-.03125,0)--++(0,.03125)--++(-.03125,0)
--++(0,2)
;

\draw[red] 
(-1.53125,0.46875)--++(0,-.03125)
--++(.03125,0)--++(0,-.03125)--++(.03125,0)--++(0,-.03125)--++(.03125,0)--++(0,-.03125)
--++(.0625,0)--++(0,-.0625)--++(.0625,0)--++(0,-.0625)
--++(.125,0)--++(0,-.125)
--++(.25,0)--++(0,-.25)--++(.25,0)--++(0,-.25)--++(.25,0)--++(0,-.25)--++(.25,0)--++(0,-.25)--++(0.25,0)
--++(0,0.5)--++(0.5,0)--++(0,0.25)--++(0.25,0)--++(0,0.25)--++(0.25,0)--++(0,0.125)--++(0.125,0)--++(0,0.125)--++(0.125,0)--++(0,0.125)--++(0.125,0)--++(0,0.0625)--++(0.0625,0)--++(0,0.0625)--++(0.0625,0)--++(0,0.0625)--++(0.0625,0)--++(0,0.0625)--++(0.0625,0)--++(0,0.03125)--++(0.03125,0)--++(-.03125,0)--++(0,-.03125);

\node[red] at (-1,-5) {$\Omega^-_{{\mathsf T}'}$};

\end{scope}

\begin{scope}

\draw[violet!70, dashed, line width=0.3mm] plot [smooth]  coordinates {(-12, -0.2) (-8, -0.3) (-6,-0.2)  (-4.5,-0.08) (-4,0.1)};
\draw[violet!70, dashed, line width=0.3mm] plot [smooth]  coordinates {(-4,0.1) (-3, 0.35) (-2, 0.3)};
\draw[violet!70, dashed, line width=0.3mm] plot [smooth]  coordinates {(-2,0.3) (-1, 0.15) (0, 0.3)  (0.5, 0.4) (1, 0.5)};
\draw[violet!70, dashed, line width=0.3mm] plot [smooth]  coordinates {(1, 0.5) (2,0.54) (3, 0.7) (4,0.6)};
\draw[violet!70, dashed, line width=0.3mm] plot [smooth]  coordinates {(4,0.6) (4.9,0.4) (5.4,0.3) (5.9, 0.2) (7, 0.1) (12, 0)  (15,-0.1)};
\node[below, violet!70] at (15,-.3) {$\Gamma_{{\mathsf T}}$};

\node at (-11.2,-1.5) {$E$}; 
\draw \mysetE;

\end{scope}

\end{tikzpicture}	
\caption{The domains $\Omega_{\sbf'}^\pm$.}\label{figure:domains}
\end{figure}


\begin{remark}\label{remark2.12} In particular, for each $\sbf\subset \G$,
if $Q'$ and $Q$ belong to $\sbf$, and 
if $Q'$ is a dyadic child of $Q$, then $U_{Q'}^+\cup U_{Q}^+$ is Harnack Chain connected,
and every pair of points $x,y\in U_{Q'}^+\cup U_Q^+$ may be connected by a Harnack Chain 
 in $\Omega_E$ 
of length at most $C= C(n,\tau,\eta,K,\textup{ADR/UR})$.  The same is true for  
$U_{Q'}^-\cup U_{Q}^-$.
\end{remark}

\begin{remark}\label{remark2.13} Let $0<\tau\leq \tau_0/2$.
Given any $\sbf\subset \G$, and any semi-coherent subtree
$\sbf'\subset \sbf$, define $\om_{\sbf'}^\pm=\om_{\sbf'}^\pm(\tau)$ as in \eqref{eq3.2},
and similarly set $\widehat{\om}_{\sbf'}^\pm=\om_{\sbf'}^\pm(2\tau)$.  Then by construction, for
any $x\in \overline{\om_{\sbf'}^\pm}$, 
$$ 
\dist(x,E) \approx \dist(x, \partial \widehat{\om}_{\sbf'}^\pm)\,,$$
where of course the implicit constants depend on $\tau$.
\end{remark}

As in \cite{HMM}, it will be useful for us to extend the definition of the Whitney region $U_Q$ to the case that
$Q\in\B$, the ``bad" collection of Lemma \ref{lemma2.1}.   Let $\W_Q^*$ be the augmentation of $\W_Q^0$
as constructed in Lemma \ref{lemma2.7}, and set
\begin{equation}\label{Wdef}
\W_Q:=\left\{
\begin{array}{l}
\W_Q^*\,,
\,\,Q\in\G,
\\[6pt]
\W_Q^0\,,
\,\,Q\in\B
\end{array}
\right.\,.
\end{equation}
For $Q \in\G$ we shall henceforth simply write $\W_Q^\pm$ in place of $\W_Q^{*,\pm}$.
For arbitrary $Q\in \dd(E)$, good or bad, we may then define
\begin{equation}\label{eq3.3bb}
U_Q=U_{Q,\tau}:= \bigcup_{I\in \W_Q} {\rm int}\left(I^*(\tau)\right)\,.
\end{equation}
Let us note that for $Q\in\G$, the latter definition agrees with that in \eqref{eq3.3aa}. Note that by construction
\begin{equation}\label{dist:UQ-pom}
U_Q\subset\{y\in\Omega_E:\ \dist(y,E)> c\eta^{1/2}\ell(Q)\}\cap B(x_Q, CK^{1/2}\ell(Q)),
\end{equation}
for some uniform constants $C\ge 1$ and $0<c<1$ (see \eqref{Whintey-4I}, \eqref{eq3.1}, and \eqref{eq2.whitney2}).
In particular, for every $Q\in\dd$ if follows that
\begin{equation}\label{def:BQ*}
\bigcup_{Q'\in\dd(Q)} U_{Q'}
\subset B(x_Q,K\ell(Q))=:B_Q^*\,.
\end{equation}
where we recall that $\dd(Q)$ is defined in \eqref{eq3.4a}. 


For future reference, we introduce dyadic sawtooth regions as follows. 
Given a family $\mathcal{F}$ of disjoint cubes $\{Q_j\}\subset \dd$, we define
the {\it global discretized sawtooth} relative to $\F$ by
\begin{equation}\label{eq2.discretesawtooth1}
\dd_{\F}:=\dd\setminus \bigcup_{Q_j\in \F} \dd(Q_j)\,,
\end{equation}
i.e., $\dd_{\F}$ is the collection of all $Q\in\dd$ that are not contained in any $Q_j\in\F$.
 We may allow $\F$ to be empty, in which case $\dd_\F=\dd$. 
Given some fixed cube $Q$, we also define 
the {\it local discretized sawtooth} relative to $\F$ by
\begin{equation}\label{eq2.discretesawtooth2}
\dd_{\F}(Q):=\dd(Q)\setminus \bigcup_{Q_j\in \F} \dd(Q_j)=\dd_\F\cap\dd(Q).
\end{equation}
Note that with this convention, $\dd(Q)=\dd_{\textup{\O}}(Q)$ (i.e., if one takes $\F=\emptyset$
in \eqref{eq2.discretesawtooth2}).

\section{Step 1: the set-up}\label{s3}

In the proof of  Theorem \ref{t1}, we shall employ a two-parameter
induction argument, which is a refinement of the
method of ``extrapolation" of Carleson measures.   The latter is a bootstrapping scheme for
lifting the Carleson measure constant, developed by J. L. Lewis \cite{LM}, and based on
the corona construction of Carleson \cite{Car} and Carleson and Garnett \cite{CG}
(see also \cite{HL}, \cite{AHLT}, \cite{AHMTT}, \cite{HM-TAMS}, \cite{HM-I},\cite{HMM}).

\subsection{Reduction to a dyadic setting}\label{ss3.1}
To set the stage for the induction procedure, let us begin by making a preliminary reduction.
It will be convenient to work with a certain dyadic version of  Definition \ref{def1.ibpcad}.
To this end, 
 let $x\in \om$, with $\delta_{\Omega}(x) < \diam(\pom)$, and set
$\Delta_x=\Delta_x^N=B(x,N\delta_{\Omega}(x))\cap\pom$, for some fixed $N\geq 2$ as in Definition
\ref{def1.johnpoint}.
Let $\hat{x}\in\pom$ be a touching point for $x$, i.e., $|x-\hat{x}|=\delta_{\Omega}(x)$. 
Choose $x_1$ on the line segment joining $x$ to $\hat{x}$, with
$\delta_{\Omega}(x_1) = \delta_{\Omega}(x)/2$, and set 
$\Delta_{x_1}=B(x_1,N\delta_{\Omega}(x)/2)\cap \pom$.  Note that $B(x_1,N\delta_{\Omega}(x)/2)\subset
B(x,N\delta_{\Omega}(x))$, and furthermore,
$$\dist\Big(B(x_1,N\delta_{\Omega}(x)/2), \partial B(x,N\delta_{\Omega}(x)\Big) > \frac{N-1}{2}\delta_{\Omega}(x) \geq \frac12\delta_{\Omega}(x).$$
We may therefore 
cover $\Delta_{x_1}$ by a disjoint collection
$\{Q_i\}_{i=1}^M\subset\dd(\pom)$, of equal length $\ell(Q_i)\approx \delta_{\Omega}(x)$, 
such that each $Q_i\subset \Delta_x$, and such that
the implicit constants depend only on $n$ and ADR, and thus 
the cardinality $M$ 
of the collection depends on $n$, ADR, and $N$. 
With $E=\pom$, we make the Whitney decomposition of the set
$\om_E =\ree\setminus E$ as in Section \ref{s2} (thus, $\om\subset\om_E$). Moreover, for
sufficiently small $\eta$ and sufficiently large $K$ in \eqref{eq3.1},
we then have that $x\in U_{Q_i}$ for each $i=1,2,\dots,M$.  By hypothesis, there are constants
$\theta_0\in (0,1],\lambda_0\in (0,1)$, and $N\geq 2$ as above, such that every $z\in\om$ is a
$(\theta_0,\lambda_0,N)$-weak local John point (Definition \ref{def1.johnpoint}).  
In particular, this is true for
$x_1$, hence there is a Borel set $F\subset \Delta_{x_1}$, with
$\sigma(F) \geq \theta_0 \sigma(\Delta_{x_1})$, such that every
$y\in F$ may be connected to $x_1$ via a $\lambda_0$-carrot path.
By $n$-ADR, $\sigma(\Delta_{x_1})\approx \sum_{i=1}^M\sigma(Q_i)$
and thus by pigeon-holing, there is at least one $Q_i=:Q$ such that
$\sigma(F\cap Q) \geq \theta_1\sigma(Q)$, with $\theta_1$ depending only on
$\theta_0$, $n$ and ADR.  Moreover, the $\lambda_0$-carrot path connecting
each $y\in F$ to $x_1$ may be extended to a $\lambda_1$-carrot path
connecting $y$ to $x$, where $\lambda_1$ depends only on $\lambda_0$.

We have thus reduced matters to the following dyadic scenario:
let $Q\in\dd(\pom)$,  let $U_Q=U_{Q,\tau}$ be the associated Whitney region as in  \eqref{eq3.3bb},
with $\tau \leq \tau_0/2$ fixed, and suppose that $U_Q$ meets $\om$
(recall that by construction $U_Q\subset \om_E=\ree\setminus E$, with $E=\pom$).
For $x\in U_Q\cap\om$, and for a constant $\lambda \in (0,1)$, let
\begin{equation}\label{eqdefFcarQ}
F_{car}(x,Q)= F_{car}(x,Q,\lambda) 
\end{equation}
denote the set of $y\in Q$ which may be joined to $x$ by a $\lambda$-carrot path
$\gamma(y,x)$, and for $\theta \in (0,1]$, set
\begin{equation}\label{eqdefTQ}
T_Q=T_Q(\theta,\lambda):= \left\{ x\in U_Q\cap\om: \, \sigma\big(F_{car}(x,Q,\lambda)\big)\geq\theta\sigma(Q)\right\}.
\end{equation}
\begin{remark}\label{remark3.5}
Our goal is to prove that, given $\lambda \in (0,1)$ and $\theta\in (0,1]$,
there are
positive constants $\eta$ and $C$, depending on $\theta,\lambda$, and the 
allowable parameters, such that
for each  $Q\in\dd(\pom)$, and for each $ x\in T_Q(\theta,\lambda)$, 
there is a Chord-arc domain 
$\Omega_x$, with uniformly controlled Chord-arc constants,
constructed as a union $\cup_kI_k^*$ of fattened Whitney boxes $I_k^*$, 
 such that 
$$U^i_Q\subset\Omega_x\subset \Omega \cap B\big(x,C\delta_{\Omega}(x)\big)\,,$$ 
where $U_Q^i$ is the particular connected component of $U_Q$ containing $x$, and
\begin{equation}\label{eq3.1a}
\sigma(\pom_x\cap Q) \geq \eta \sigma(Q)\,.
\end{equation}   
For some $Q\in \dd(\pom)$, 
it may be that $T_Q$ is empty.  On the other hand,
by the preceding discussion, each $x\in \om$ belongs to $T_Q(\theta_1,\lambda_1)$ for suitable
$Q,\theta_1$ and $\lambda_1$, so that 
 \eqref{eq3.1a} (with $\theta=\theta_1, \lambda=\lambda_1$) implies 
\begin{equation*}
\sigma(\pom_x\cap \Delta_x) \geq \eta_1 \sigma(\Delta_x)\,,
\end{equation*}
with $\eta_1\approx \eta$, where $Q$ is the particular $Q_i$ selected 
in the previous paragraph.  Moreover, since $x\in T_Q\subset U_Q$, we can modify $\Omega_x$
if necessary, by adjoining to it one or more fattened Whitney boxes $I^*$ with $\ell(I) \approx \ell(Q)$, to ensure that
for the modified $\om_x$, it holds in addition that
$\dist(x,\pom_x) \gtrsim \ell(Q) \approx \delta_{\Omega}(x)$, and therefore
$\om_x$ verifies all the conditions in Definition \ref{def1.ibpcad}.
\end{remark}

The rest of this section is therefore devoted to proving that there exists, for a given $Q$ and for each
$x\in T_Q(\theta,\lambda)$, a Chord-arc domain $\Omega_x$ satisfying the stated properties
(when the set $T_Q(\theta,\lambda)$ is not vacuous).   To this end,
we let $\lambda \in(0,1)$
(by Remark \ref{remark3.5}, any fixed $\lambda \leq \lambda_1$ will suffice).
We also fix  positive numbers 
$K\gg \lambda^{-4}$, and $\eta \leq K^{-4/3}\ll \lambda^4$,
and for these values of $\eta$ and $K$, 
we make
the bilateral Corona decomposition of Lemma \ref{lemma2.1}, so that  $\dd(\pom)=\G\cup\B$. 
We also construct the Whitney collections $\W^0_Q$ in \eqref{eq3.1}, and $\W_Q^*$ of Lemma \ref{lemma2.7}
for this same choice of $\eta$ and $K$.


Given a cube $Q\in \dd(\pom)$, we set
\begin{equation}\label{eq5.8aa}
\dd_*(Q):=\left\{Q'\subset Q:\, \ell(Q)/4\leq \ell(Q')\leq \ell(Q)
\right\}\,.
\end{equation}
Thus, $\dd_*(Q)$ consists of the cube $Q$ itself, along with
its dyadic children and grandchildren.
Let
$$\M:=\{Q(\sbf)\}_{\sbf\subset\mathcal{G}}$$ denote the collection of   
cubes which are the maximal elements of the  trees $\sbf$ in $\G$.
We define
\begin{equation}\label{eq4.0}
\alpha_Q:= 
\begin{cases} \sigma(Q)\,,&{\rm if}\,\,  (\M\cup\B)\cap \dd_*(Q)\neq \emptyset, \\
0\,,& {\rm otherwise}.\end{cases}
\end{equation}
Given  any collection $\dd'\subset\dd(\pom)$, we set
\begin{equation}\label{eq4.1}
\mut(\dd'):= \sum_{Q\in\dd'}\alpha_{Q}.
\end{equation}
Then $\mut$ is a discrete Carleson measure, i.e.,
recalling that $\dd(R)$ is the discrete Carleson region
relative to $R$
defined in \eqref{eq3.4a}, we claim that there is a uniform constant $C$ such that
\begin{equation}\label{eq6.2}
\mut(\dd(R))\, =\sum_{Q\subset R}\alpha_{Q} \leq\, C\sigma(R)\,,\qquad \forall\,R\in \dd(\pom)\,.
\end{equation}
Indeed,  note that for any 
$Q'\in\dd$, there are at most 3 cubes $Q$ such that 
$Q'\in \dd_*(Q)$ (namely, $Q'$ itself, its dyadic parent, and its dyadic grandparent),
and that by $n$-ADR,
$\sigma(Q)\approx \sigma(Q')$, if $Q'\in\dd_*(Q)$.   Thus,
given any
$R\in\dd(\pom)$,
\begin{multline*}
\mut(\dd(R))\, =\sum_{Q\subset R}\alpha_Q\,
\leq \sum_{Q'\in \M\cup\B}\,\sum_{Q\subset R:\, Q'\in \dd_*(Q)}\sigma(Q) \\[4pt]
\lesssim\sum_{Q'\in \M\cup\B:\, Q' \subset R}\sigma(Q')\,\leq\,C \sigma(R)\,,
\end{multline*}
by Lemma \ref{lemma2.1} part (2).  Here, and throughout the remainder of this section, a
generic constant $C$, and implicit constants, are allowed to depend upon the choice of the parameters
$\eta$ and $K$ that we have fixed, along with the usual allowable parameters.

With \eqref{eq6.2} in hand, we therefore have
\begin{equation}\label{eq4.7a}
M_0:= 
\sup_{Q\in\dd(\Omega)}\frac{\mut(\dd(Q))}{\sigma(Q)}\leq C<\infty\,.
\end{equation}

\subsection{Induction Hypothesis and Outline of Proof} \label{ss-indhyp}
As mentioned above, our proof will be based on a 
two parameter induction scheme.  Given $\lambda\in (0,\lambda_1]$ fixed as above, we 
recall that the set $F_{car}(x,Q,\lambda)$
is defined in \eqref{eqdefFcarQ}.   The induction hypothesis, which we formulate for any $a\geq 0$, and any $\theta\in (0,1]$ is as follows:
\\[.3cm]
\null\hskip.1cm \fbox{\rule[8pt]{0pt}{0pt}$H[a,\theta]$}\hskip4pt \fbox{\ \parbox[c]{.75\textwidth}{%
%
\rule[10pt]{0pt}{0pt}\it  
There is a positive constant $c_{a}=c_a(\theta)<1$ 
such that for any given $Q\in\dd(\pom)$,  if
\begin{equation}\label{eq3.12}
\!\!\!\!\!\!\!\!\!\!\!\!\!\!\!\!\!\!\!\!\!\!\!\!\!\!\!\!\!\!\!\!\!\!\!\!\!\!\mut(\dd(Q))\le \, a\sigma(Q),
\end{equation} 
and if there is a subset $V_Q\subset U_Q\cap\om$ for which 
\begin{equation}\label{eq3.10}
\!\!\!\!\!\!\!\!\!\!\!\!\!\!\!\!\!\!\!\!\!\!\!\!\!\!\!\!\!\!\!\!\!\!\!\!\!
\sigma\left(\bigcup_{x\in V_Q}F_{car}(x,Q,\lambda) \right) \, \geq\, \theta \sigma(Q)\,,
\end{equation}
then there is a subset $V^*_Q\subset V_Q$, 
such that for each connected component $U_Q^i$ of $U_Q$ 
which meets $V^*_Q$, there is a  Chord-arc domain $\Omega_Q^i$ 
which is the interior of the union of a collection of fattened Whitney cubes $I^*$, and whose Chord-arc 
constants depend only on dimension, $\lambda$, $a$, $\theta$, and the ADR constants for $\Omega$.
Moreover, 
$U_Q^i\subset \Omega^i_Q\subset B_Q^*\cap\Omega=B(x_Q,K\ell(Q))\cap\Omega$,
and
$\sum_i \sigma (\pom^i_{Q}\cap Q)\geq  c_a\sigma(Q)$,
where the sum runs over those $i$ such that $U_Q^i$ meets $V^*_Q$.
}\ }

\bigskip

Let us briefly sketch the strategy of the proof.  We first fix $\theta=1$, and by induction on
$a$, establish $H[M_0,1]$.   We then show that there is a fixed $\zeta\in (0,1)$ such that
$H[M_0,\theta]$ implies $H[M_0,\zeta\theta]$, for every $\theta\in(0,1]$.   Iterating,
we then obtain $H[M_0,\theta_1]$ for any $\theta_1\in (0,1]$.  Now, by \eqref{eq4.7a},
we have \eqref{eq3.12} with $a=M_0$, for every $Q\in\dd(\pom)$.  Thus, $H[M_0,\theta_1]$ may be applied in
every cube $Q$ such that $T_Q(\theta_1,\lambda)$ (see \eqref{eqdefTQ}) is non-empty, with $V_Q=\{x\}$,
for any $x\in T_Q(\theta_1,\lambda)$.  
For $\lambda\leq \lambda_1$, and an appropriate choice
of $\theta_1$, by Remark \ref{remark3.5}, we obtain the existence of a Chord-arc domain $\Omega_x$ 
verifying the conditions of Definition \ref{def1.ibpcad}, and thus that
Theorem \ref{t1} holds, as desired.

\section{Some geometric observations}\label{s-observe}

We begin with some preliminary observations.  In what follows we have fixed 
$\lambda \in(0,\lambda_1]$ and two positive numbers 
$K\gg \lambda^{-4}$, and $\eta \leq K^{-4/3}\ll \lambda^4$, for which the bilateral Corona decomposition of $\dd(\pom)$ in Lemma \ref{lemma2.1} is applied. We now fix $k_0\in\mathbb{N}$, $k_0\ge 4$,  such that
\begin{equation}\label{eq3.15}
2^{-k_0} \,\leq\, \frac{\eta}{K}\, < \,2^{-k_0+1}\,.
\end{equation}

\begin{lemma}\label{lemma3.15}  Let $Q\in\dd(\pom)$, and suppose that $Q'\subset Q$, with
$\ell(Q')\leq 2^{-k_0}\ell(Q)$.  Suppose that there are points $x\in U_Q\cap\om$ and $y\in Q'$, that are connected by a $\lambda$-carrot
path $\gamma=\gamma(y,x)$ in $\om$.   Then $\gamma$ meets $U_{Q'}\cap\Omega$. 
\end{lemma}

\begin{proof}  
By construction (see \eqref{eq3.1}, Lemma \ref{lemma2.7}, \eqref{Wdef} and \eqref{eq3.3bb}), 
$x\in U_Q$ implies that
$$\eta^{1/2}\ell(Q) \lesssim \delta_{\Omega}(x) \lesssim K^{1/2}\ell(Q)\,.$$  Since $2^{-k_0}\ll\eta$, 
and $\ell(Q')\leq 2^{-k_0}\ell(Q)$, we then have that
$x\in \om\setminus B\big(y,2\ell(Q')\big)$, so $\gamma(y,x)$ meets
$B\big(y,2\ell(Q')\big)\setminus B\big(y,\ell(Q')\big)$, say at a point $z$.  Since
$\gamma(y,x)$ is a $\lambda$-carrot path, and since we have previously
specified that $\eta \ll\lambda^4$,
$$\delta_{\Omega}(z) \geq \lambda \ell\big(\gamma(y,z)\big) \geq \lambda|y-z|\geq\lambda \ell(Q')\gg 
\eta^{1/4}\ell(Q')\,.$$  On the other hand
$$\delta_{\Omega}(z)\leq \dist(z,Q') \leq |z-y| \leq 2\ell(Q') \ll K^{1/2} \ell(Q')\,.$$
In particular then, the Whitney box $I$ containing $z$ must belong to $\W^0_{Q'}$
(see \eqref{eq3.1}), so $z\in U_{Q'}$.  Note that $z\in\Omega$ since $\gamma\subset\Omega$.
\end{proof}

We shall also require the following.   We recall 
that by Lemma \ref{lemma2.7}, for $Q\in \sbf\subset \G$, the Whitney region $U_Q$
has the splitting $U_Q = U_Q^+\cup U_Q^-$, with $U_Q^+$ (resp. $U_Q^-$) lying
above (resp., below) the Lipschitz graph $\Gamma_\sbf$ of Lemma \ref{lemma2.1}.

\begin{lemma}\label{lemma3.37}  Let $Q'\subset Q$, 
and suppose that $Q'$ and $Q$ both belong to $\G$, and moreover
that both $Q'$ and $Q$ belong to the same 
 tree $\sbf\subset \G$.
Suppose that $y\in Q'$ and $x\in U_Q\cap\om$ are connected via a $\lambda$-carrot path
$\gamma(y,x)$ in $\om$,
and assume that there is a point $z\in \gamma(y,x) \cap U_{Q'}\cap\Omega$ 
(by Lemma \ref{lemma3.15} we know that such a $z$ exists provided $\ell(Q')\leq 2^{-k_0} \ell(Q)$).
Then $x\in U_Q^+$ if and only if $z\in U_{Q'}^+$ 
(thus, $x\in U_Q^-$  if and only if $z\in U_{Q'}^-$).
\end{lemma}

\begin{proof}
We suppose for the sake of contradiction that, e.g., $x\in U_Q^+$, and that $z\in U_{Q'}^-$.
Thus, in traveling from $y$ to $z$ and then to $x$ along the path
$\gamma(y,x)$, one must cross the Lipschitz graph $\Gamma_\sbf$ at least once between
$z$ and $x$.
Let $y_1$ be the first point on $\gamma(y,x)\cap\Gamma_\sbf$ that one 
encounters {\it after} $z$, when traveling toward $x$.
By Lemma \ref{lemma2.7},
\begin{equation*}
K^{1/2}\ell(Q) \gtrsim \delta_{\Omega}(x) \geq \lambda \ell\big(\gamma(y,x)\big)
\gg K^{-1/4} \ell\big(\gamma(y,x)\big)\,,
\end{equation*}
where we recall that we have fixed $K\gg \lambda^{-4}$.
Consequently, $\ell\big(\gamma(y,x)\big) \ll K^{3/4} \ell(Q)$, so in particular,
$\gamma(y,x)\subset B_Q^*:= B\big(x_Q,K\ell(Q)\big)$, as in Lemma \ref{lemma2.1}.
On the other hand, $y_1\notin B^*_{Q'}$.  Indeed,
$y_1\in\Gamma_\sbf$, so if $y_1 \in B^*_{Q'}$, then by
\eqref{eq2.2a},  $\delta_{\Omega}(y_1) \leq \eta \ell(Q')$.  However,
$$\delta_{\Omega}(y_1) \geq \lambda \ell\big(\gamma(y,y_1)\big)
\geq \lambda \ell\big(\gamma(y,z)\big) \geq \lambda |y-z| 
\geq \lambda \dist(z,Q')  \gtrsim \lambda \eta^{1/2} \ell(Q')\,,$$
where in the last step we have used Lemma \ref{lemma2.7}.  This contradicts
our choice of $\eta \ll \lambda^4$.

We now form a chain of consecutive dyadic cubes $\{P_i\}\subset \dd(Q)$, 
connecting $Q'$ to $Q$, i.e.,
$$ 
Q'
=P_0 \lhd P_1 \lhd P_2\lhd \dots \lhd P_M \lhd P_{M+1}=Q\,,$$
where the introduced notation $P_i\lhd P_{i+1}$ means that $P_i$ is the 
dyadic child of $P_{i+1}$, that is,   $P_i\subset P_{i+1}$ and $\ell(P_{i+1})=2\ell(P_i)$.  Let $P:=P_{i_0}$, $1\le i_0\le M+1$, be the smallest of the cubes $P_i$ such that 
$y_1 \in B_{P_i}^*$.    Setting $P':= P_{i_0-1}$, we then have that
$y_1 \in B_P^*$, and $y_1\notin B_{P'}^*$.   By the coherency of $\sbf$, it follows that $P\in \sbf$, so
by \eqref{eq2.2a},
\begin{equation}\label{eq3.38}
\delta_{\Omega}(y_1) \leq \eta\ell(P)\,.
\end{equation}  
On the other hand, 
$$\dist(y_1,P') \gtrsim K\ell(P') \approx K \ell(P)\,,$$ and therefore, since
$y\in Q' \subset P'$,
\begin{equation}\label{eq3.39}
\delta_{\Omega}(y_1)\geq \lambda \ell\big(\gamma(y,y_1)\big)
\geq \lambda|y-y_1| \geq \lambda\dist(y_1,P') \gtrsim
\lambda K\ell(P)\,.
\end{equation}
Combining \eqref{eq3.38} and \eqref{eq3.39}, we see that
$\lambda  \lesssim \eta/K$, which contradicts that we have fixed
$\eta \ll \lambda^4$, and $K\gg \lambda^{-4}$.
\end{proof}

\begin{lemma}\label{lemma:VQ}
Fix $\lambda\in(0,1)$. Given $Q\in\dd(\pom)$ and a non-empty set 
$V_Q\subset U_Q\cap \Omega$, such that each $ x \in V_Q$ may be 
connected by a $\lambda$-carrot path to some $y\in Q$, set
\begin{equation}\label{defi-FQ}
F_Q:= \bigcup_{x\in V_Q} F_{car}(x,Q,\lambda)\,,
\end{equation}
where  we recall that $F_{car}(x,Q,\lambda)$ is the 
set of $y\in Q$ that are connected via a $\lambda$-carrot path to $x$ (see \eqref{eqdefFcarQ}).  
Let $Q'\subset Q$ be such that $\ell(Q')\leq 2^{-k_0} \ell(Q)$ and $F_Q\cap Q'\neq\emptyset$. Then, there exists a non-empty set $V_{Q'}\subset U_{Q'}\cap \Omega$ such that if we 
define $F_{Q'}$ as in \eqref{defi-FQ} with $Q'$ 
replacing $Q$, then $ F_Q\cap Q'\subset  F_{Q'}$. Moreover, for 
every $y\in V_{Q'}$, there exist $x\in V_Q$, $y\in Q'$  (indeed $y\in F_Q\cap Q'$)  and 
a $\lambda$-carrot path $\gamma=\gamma(y,x)$ such that $y\in\gamma$.

\end{lemma}

\begin{proof}
For every $y\in F_Q\cap Q'$, by definition of $F_Q$, there 
exist $x\in V_Q$ and a $\lambda$-carrot path $\gamma=\gamma(y,x)$. 
By Lemma \ref{lemma3.15},
there is a point $y'=y'(y)\in\gamma\cap U_{Q'}\cap\Omega$ (there can be more than one $y'$, but we just pick one). 
Note that the sub-path $\gamma(y,y')\subset \gamma(y,x)$ is also a $\lambda$-carrot path, for the 
same constant $\lambda$. All the conclusions in the lemma follow easily from the construction  
by letting $V_{Q'}=\bigcup_{y\in F_Q\cap Q'} y'(y)$. 
\end{proof}

\begin{remark}\label{remark-VQ-children}
It follows easily from the previous proof that under the same assumptions, if one further assumes that $\ell(Q')<2^{-k_0}\,\ell(Q)$, we can then repeat the argument with both $Q'$ and $(Q')^*$ (the dyadic parent of $Q'$) to obtain respectively $V_{Q'}$ and $V_{(Q')^*}$. Moreover, this can be done in such a way that every point in $V_{Q'}$ (resp. $V_{(Q')^*}$) belongs to a $\lambda$-carrot path which also meets $V_{(Q')^*}$ (resp. $V_{Q'}$), connecting $U_Q$ and $Q'$.
\end{remark}

Given a family $\F:=\{Q_j\}\subset \dd(\pom)$ of
pairwise disjoint cubes, we recall that the ``discrete  sawtooth" $\dd_\F$ is 
the collection of all cubes in $\dd(\pom)$ that are not contained in any $Q_j\in\F$ (see \eqref{eq2.discretesawtooth1}),
and we define the restriction of $\mut$ (cf. \eqref{eq4.0}, \eqref{eq4.1}) to the sawtooth $\dd_\F$ by
\begin{equation}\label{eq4.4x}
\mut_\F(\dd'):=\mut(\dd'\cap\dd_\F)= \sum_{Q\in\dd'\setminus \big(\cup_{Q_j\in \F} \,\dd(Q_j)\big)}\alpha_{Q}.
\end{equation}
We then set 
$$\|\mut_\F\|_{\C(Q)}:= \sup_{Q'\subset Q}\frac{\mut_\F\big(\dd(Q')\big)}{\sigma(Q')}.$$
Let us note that we may allow $\F$ to be empty, in which case $\dd_\F=\dd$ and $\mut_\F$ 
is simply $\mut$. We note that  the following claim, and others in the sequel,  remain true when $\F$ is empty; sometimes trivially so, and sometimes
with some straightforward changes that are left to the interested reader. 


\begin{claim}\label{claim3.16} Given $Q\in\dd(\pom)$, and 
a family $\F=\F_Q:=\{Q_j\}\subset \dd(Q)\setminus \{Q\}$ of
pairwise disjoint sub-cubes of $Q$, 
if $\|\mut_\F\|_{\C(Q)} \leq 1/2$, then each $Q' \in \dd_{\F}\cap \dd(Q)$, each $Q_j\in\F$, 
and every dyadic child $Q_j'$ of any $Q_j\in \F$, belong to the 
good collection $\G$, and moreover, every such cube belongs to the {\bf same}  tree
$\sbf\subset\G$.  In particular, $\sbf' :=\dd_\F\cap\dd(Q)$ is a semi-coherent subtree of $\sbf$, and so is
$\sbf'':= (\dd_\F\cup\F\cup\F')\cap\dd(Q)$, where $\F'$ denotes the collection of all
dyadic children of cubes in $\F$.
\end{claim}
Indeed, if any $Q' \in \dd_{\F}\cap \dd(Q)$ were in $\M\cup\B$ (recall that
$\M:=\{Q(\sbf)\}_{\sbf\subset\G}$ is the collection of 
cubes which are the maximal elements of the  trees $\sbf$ in $\G$), then
by construction $\alpha_{Q'}=\sigma(Q')$ for that cube (see \eqref{eq4.0}), so by definition of
$\mut$ and $\mut_\F$, we would have 
$$1=\frac{\sigma(Q')}{\sigma(Q')} \leq \frac{\mut_\F\big(\dd(Q')\big)}{\sigma(Q')} \leq 
\|\mut_\F\|_{\C(Q)}\leq \frac12\,,$$
a contradiction.  Similarly, if some $Q_j\in\F$ (respectively, $Q_j'\in \F'$)
were in $\M\cup\B$, then its dyadic parent (respectively, dyadic grandparent)
$Q_j^*$ would belong to $\dd_{\F}\cap \dd(Q)$, and by definition $\alpha_{Q_j^*}=\sigma(Q_j^*)$, so again 
we reach a contradiction.  Consequently, $\F \cup \F'\cup
(\dd_\F\cap\dd(Q))$ does not meet $\M\cup\B$, and the claim follows.

\section{Construction of chord-arc subdomains}\label{s-cad}

For future reference, we now prove the following.  Recall that for $Q\in \G$, $U_Q$ has precisely two
connected components $U_Q^\pm$ in $\ree\setminus \pom$.

\begin{lemma}\label{lemmaCAD} Let  $Q\in\dd(\pom)$, let $k_1$ be such that
$2^{k_1}> 2^{k_0}\gg 100 K$, 
see \eqref{eq3.15}, 
and suppose that there is 
a family $\F=\F_Q:=\{Q_j\}\subset \dd(Q)\setminus \{Q\}$ of
pairwise disjoint sub-cubes of $Q$, 
with $\|\mut_\F\|_{\C(Q)} \leq 1/2$ (hence by Claim \ref{claim3.16},  there is some $\sbf\subset\G $ with
$\sbf \supset (\dd_\F\cup\F\cup\F')\cap\dd(Q)$), and  
 a non-empty subcollection
$\F^*\subset \F$, 
such that: 
\begin{itemize}
\item[(i)] $\ell(Q_j) \leq 2^{-k_1}\ell(Q)$, for each cube $Q_j\in \F^*$;
 \smallskip
\item[(ii)] the collection of balls $\big\{\kappa B^*_{Q_j}:= B\big(x_{Q_j},\kappa 
K \ell(Q_j)\big):\, Q_j \in \F^* \big\}$
is pairwise disjoint, where
$\kappa\gg K^4 $ is a sufficiently large positive constant; and
\smallskip
\item[(iii)] $\F^*$ has a disjoint decomposition
$\F^*=\F^*_+ \cup \F^*_-$, where for each $Q_j \in \F_\pm^*$, there is a
Chord-arc subdomain $\Omega_{Q_j}^\pm\subset \Omega$, consisting 
of a union of fattened Whitney cubes $I^*$, with
$U_{Q_j}^\pm\subset \Omega_{Q_j}^\pm\subset
B^*_{Q_j}:= B(x_{Q_j},K \ell(Q_j))$, and with uniform control of the Chord-arc constants. 
\end{itemize}
Define a semi-coherent subtree $\sbf^*\subset \sbf$ by
\[
	\sbf^*=\left\{Q'\in \dd(Q): Q_j\subset Q' \text{ for some }Q_j\in\F^*\right\}\,,
	\] 
and for each choice of $\pm$ for which $\F^*_\pm$ is non-empty, set
\begin{equation}\label{eqomdef}  \Omega_Q^\pm:= \Omega_{\sbf^*}^\pm \bigcup
 \left(\bigcup_{Q_j \,\in\, \F^*_\pm}\Omega^\pm_{Q_j}\right)
 \end{equation}	
	Then for $\kappa$ large enough, depending only on allowable parameters, 
$\om_Q^\pm$ is a Chord-arc domain, with chord arc constants depending only on the 
 uniformly controlled Chord-arc constants of $\Omega_{Q_j}^\pm$ and on the other allowable parameters.
 Moreover, $ \Omega_Q^\pm\subset  B_Q^*\cap\Omega=
 B(x_Q,K\ell(Q))\cap\Omega$, 
 and $\Omega_Q^\pm$ is a union of fattened Whitney cubes. 
\end{lemma}

\begin{remark} Note that we define $\om_Q^\pm$ if and only if $\F^*_\pm$ is non-empty.
It may be that one  of $\F^*_+, \F^*_-$ is empty, 
but $\F^*_+$ and $\F^*_-$ cannot both be empty, since $\F^*$ is non-empty by assumption.
\end{remark}

\begin{proof}[Proof of Lemma \ref{lemmaCAD}] 
Without loss of generality we may assume that $\Omega_{Q_j}\pm$ is not contained in  $\om_{\sbf^*}^\pm$ for 	
all $Q_j\in\F^*$ (otherwise we can drop those cubes from $\F^*$). On the other hand, we notice that $\Omega_Q^\pm$ 
is a union of (open) fattened Whitney cubes (assuming that it is non-empty):  each $\om_{Q_j}^\pm$ has this 
property by assumption, as does $\om_{\sbf^*}^\pm$  by construction.

We next observe that if $\Omega_Q^+$ (resp. $\Omega_Q^-$) is non-empty, then it is contained in
$\Omega$.   Indeed, by construction, $\Omega_Q^+$ is non-empty if and only if $\F^*_+$ is non-empty.
In turn, $\F^*_+$ is non-empty if and only if there is some $Q_j\in \F^*$ such that
$U_{Q_j}^+\subset \Omega_{Q_j}^+ \subset \Omega$, and moreover, the latter is true for every
$Q_j \in \F^*_+$, by definition.  But each such $Q_j$ belongs to $\sbf^*$, hence
$U_{Q_j}^+\subset  \om^+_{\sbf^*}$, again by construction (see \eqref{eq3.2}).  Thus, 
$\om_{\sbf^*}^+$ meets $\om$, and since $\om_{\sbf^*}^+\subset \ree\setminus \pom$, therefore
$\om_{\sbf^*}^+\subset \om$. Combining these observations, we see that $\Omega_Q^+\subset \om$.
Of course, the same reasoning applies to $\Omega_Q^-$, provided it is non-empty.  

In addition, 
since $\sbf^*\subset \sbf$, 
and since $K \gg K^{1/2}$, 
by Lemma \ref{lemma2.7} we have $\Omega_{\sbf^*}^\pm \subset 
B_Q^*= B(x_Q,K\ell(Q))$.  Furthermore,
$\Omega_{Q_j}^\pm\subset B^*_{Q_j}:= B(x_{Q_j},K \ell(Q_j))$, and since 
$\ell(Q_j) \leq 2^{-k_1}\ell(Q)\leq (100K)^{-1}\ell(Q)$, we obtain
\[\dist(\Omega_{Q_j}^\pm,Q) +\diam (\Omega_{Q_j}^\pm) \leq 3 K \ell(Q_j) \leq 3K 2^{-k_1} \ell(Q) \ll \ell(Q)\,.\]
Thus, in particular, $\Omega_{Q_j}^\pm\subset B_Q^*$, and therefore also 
$\om_Q^\pm\subset B_Q^*$.

It therefore remains to establish the Chord-arc properties.
It is straightforward to prove the interior corkscrew condition and the upper $n$-ADR bound, and we omit the
details.  Thus, we must verify the Harnack Chain condition,
 the lower $n$-ADR bound, and the exterior corkscrew condition.

\subsection{Harnack Chains}
Suppose, without loss of generality, 
that $\Omega_Q^+$ is non-empty, and let $x,y\in \Omega_Q^+$, with $|x-y|= r$.
If $x$ and $y$ both lie in $\Omega_{\sbf^*}^+$, or in the same $\om_{Q_j}^+$, then we can
connect $x$ and $y$ by a suitable Harnack path, since each of these domains is Chord-arc.
Thus, we may suppose either that 1) $x\in \Omega_{\sbf^*}^+$ and $y$ lies in some $\om_{Q_j}^+$,
or that 2) $x$ and  $y$  lie in two distinct $\om_{Q_{j_1}}^+$ and $\om_{Q_{j_2}}^+$.
We may reduce the latter case to the former case: 
by the separation property (ii) in Lemma \ref{lemmaCAD}, we must have
$r\gtrsim \kappa \max\big(\diam(\om_{Q_{j_1}}^+),\diam(\om_{Q_{j_2}}^+)\big)$, so given case 1), we
can connect $x\in \om_{Q_{j_1}}^+$ to the  center $z_1$ of some $I_1^*\subset U^+_{Q_1}$,
and
$y\in \om_{Q_{j_2}}^+$ to the center $z_2$ of some $I_2\subset U^+_{Q_2}$, 
where $Q_1,Q_2 \in \sbf^*$,
with $Q_{j_i}\subset Q_i\subset Q$, and
$\ell(Q_i)\approx r$, $i=1,2$.  Finally, we can connect $z_1$ and $z_2$ using that $\Omega_{\sbf^*}^+$ is Chord-arc. 

Hence, we need only construct a suitable Harnack Chain in Case 1).
We note that by assumption and construction,
$U_{Q_j}^+ \subset \Omega_{\sbf^*}^+ \cap \om_{Q_j}^+$.  

Suppose first that 
\begin{equation}\label{eqxyclose}
|x-y|=r \leq c' \ell(Q_j)\,,
\end{equation}
where $c'\leq 1$ is a sufficiently small positive
constant to be chosen.  Since $y \in \om^+_{Q_j} \subset B_{Q_j}^*$, we then have that
$x \in 2 B_{Q_j}^*$, so by the construction of $\Omega_{\sbf^*}^+$ and the separation
property (ii), it follows that 
$\delta_{\Omega}(x) \geq c \ell(Q_j)$, where $c$ is a uniform constant depending only on the allowable parameters
(in particular, this fact is true for all $x \in \om_{\sbf^*}^+ \cap 2 B_{Q_j}^*$, so it does not depend on the
choice of $c'<1$).  Now choosing $c'\leq c/2$ (eventually, it may be even smaller), 
we find that $\delta_{\Omega}(y) \geq (c/2)\ell(Q_j)$.  
Moreover, $y \in \om_{Q_j}^+ \subset B_{Q_j}^*$ implies that $\delta_{\Omega}(y) \le K\ell(Q_j)$. Also, since $x\in 2B_{Q_j}^*$ we have that $\delta_{\Omega}(x) \le 2K\ell(Q_j)$.
Since
$\om_{Q_j}^+$ and $\om_{\sbf^*}^+$ are each the interior of a union of fattened Whitney cubes,
it follows that there are Whitney cubes $I$ and $J$, with $x\in I^*$, $y\in J^*$, and
\[
\ell(I)\approx \ell(J) \approx \ell(Q_j)\,,\]
where the implicit constants depend on $K$.   
For  $c'$ small enough in \eqref{eqxyclose}, depending on the implicit
constants in the last display, and on the parameter $\tau$ in \eqref{whitney1},
this can happen only if $I^*$ and $J^*$ overlap (recall that we have fixed $\tau$ small enough that
$I^*$ and $J^*$ overlap if and only if $I$ and $J$ have a boundary point in common), in which case we may trivially connect
$x$ and $y$ by a suitable Harnack Chain.

On the other hand, suppose that 
\[
|x-y|=r \geq c' \ell(Q_j)\,.
\]
Let $z \in  U_{Q_j}^+ \subset \Omega_{\sbf^*}^+ \cap \om_{Q_j}^+$, with
$\dist(z, \pom_Q^+) \gtrsim \ell(Q_j)$ (we may find such a $z$, since $U_{Q_j}^+$ is a union of fattened
Whitney cubes, all of length $\ell(I^*) \approx \ell(Q_j)$; just take $z$ to be the center of
such an $I^*$).
We may then construct an appropriate Harnack Chain from $y$ to $x$ by connecting
$y$ to $z$ via a Harnack Chain in the Chord-arc domain $\Omega_{Q_j}^+$, and 
$z$ to $x$ via a Harnack Chain in the Chord-arc domain $\Omega_{\sbf^*}^+$.

\subsection{Lower $n$-ADR and exterior corkscrews}
We will 
establish these two properties essentially simultaneously.  
Again suppose that, e.g.,  $\Omega_Q^+$ is non-empty. 
Let $x \in \pom_Q^+$, and consider $B(x,r)$, with 
$r < \diam \om_Q^+ \approx_K\ell(Q)$. 
Our main goal at this stage is to prove the following:
\begin{equation}\label{eqvolumebound}
\big| B(x,r) \setminus \overline{\om_Q^+}\big| \geq c r^{n+1}\,,
\end{equation}
with $c$ a uniform positive constant depending only upon allowable parameters (including $\kappa$).
Indeed, momentarily taking this estimate for granted, we may combine \eqref{eqvolumebound}
with the interior corkscrew condition to deduce the lower $n$-ADR bound via the
relative isoperimetric inequality \cite[p. 190]{EG}.  In turn, with both the lower and upper 
$n$-ADR bounds in hand, \eqref{eqvolumebound} implies the existence of exterior corkscrews 
(see, e.g., \cite[Lemma 5.7]{HM-I}).  

Thus, it is enough to prove \eqref{eqvolumebound}.
We consider the following cases.

\smallskip

\noindent{\bf Case 1}:  $B(x,r/2)$ does not meet $\pom_{Q_j}^+$ for any $Q_j \in \F^*_+$.
In this case, the exterior corkscrew for $\om_{\sbf^*}^+$ associated with $B(x,r/2)$ easily implies 
\eqref{eqvolumebound}.

\smallskip

\noindent{\bf Case 2}: $B(x,r/2)$  meets $\pom_{Q_j}^+$ for at least one $Q_j \in \F^*_+$,
and $r \leq \kappa^{1/2} \ell(Q_{j_0})$, where $Q_{j_0}$ is chosen to have the largest
length $\ell(Q_{j_0})$ among those $Q_j$ such that 
$\pom_{Q_j}^+$  meets $B(x,r/2)$.  We now further split the present case into subcases.

\smallskip

\noindent{\bf Subcase 2a}:  $B(x,r/2)$  meets $\pom^+_{Q_{j_0}}$ at a point $z$ with 
$\delta_{\Omega}(z) \leq (M\kappa^{1/2})^{-1} \ell(Q_{j_0})$, where $M$ is a large number to be chosen.
Then $B(z, (M\kappa^{1/2})^{-1} r) \subset B(x,r)$, for $M$ large enough.  In addition, we claim that
$B(z, (M\kappa^{1/2})^{-1} r)$ misses $\om_{\sbf^*}^+\cup\big(\cup_{j\neq j_0} \om_{Q_j}^+\big)$.
The fact that $B(z, (M\kappa^{1/2})^{-1} r)$ misses every other $\om_{Q_j}^+, j\neq j_0$,
 follows immediately from the restriction  $r \leq \kappa^{1/2} \ell(Q_{j_0})$, and 
the separation property (ii).
To see that $B(z, (M\kappa^{1/2})^{-1} r)$ misses $\om_{\sbf^*}^+$,
note that if $|z-y|< (M\kappa^{1/2})^{-1} r$, then 
\[\delta_{\Omega}(y) \leq \delta_{\Omega}(z) + (M\kappa^{1/2})^{-1} r \leq \left((M\kappa^{1/2})^{-1}
+M^{-1}\right) \ell(Q_{j_0}) \ll  \ell(Q_{j_0})\,,\]
for $M$ large.  On the other hand, 
\[\delta_{\Omega}(y) \gtrsim \ell(Q_{j_0})\,, \qquad \forall \,y \in \om_{\sbf^*}^+ \cap
B\big(z, \kappa^{1/2} \ell(Q_{j_0})\big)
\,,\]  by the construction of
$\om_{\sbf^*}^+$ and the separation property (ii).  Thus, the claim follows, for a
sufficiently large (fixed) choice of $M$.  Since $B(z, (M\kappa^{1/2})^{-1} r)$ misses
$\om_{\sbf^*}^+$ and all other $\om_{Q_j}^+$, we inherit an exterior corkscrew point in the present 
case (depending on $M$ and $\kappa$) from the Chord-arc domain $\om_{Q_{j_0}}^+$.    Again \eqref{eqvolumebound} follows.

\smallskip

\noindent{\bf Subcase 2b}: $\delta_{\Omega}(z) \geq  (M\kappa^{1/2})^{-1} \ell(Q_{j_0})$,
for every $z\in B(x,r/2)\cap \pom_{Q_{j_0}}^+$ (hence $\delta_{\Omega}(z) \approx_{\kappa,K}
\ell (Q_{j_0})$, since $\Omega^+_{Q_{j_0}} \subset B^*_{Q_{j_0}}$).  We claim that
consequently,
$x\in \partial I^*$, for some $I$ with $\ell(I) \approx \ell(Q_{j_0}) \gtrsim r$,
such that $\text{int}\, I^* \subset \om_Q^+$.  To see this, observe that it is clear 
if $x\in \pom_{Q_{j_0}}^+$ (just take $z=x$).  Otherwise, by the separation property (ii),
the remaining 
possibility in the present scenario 
is that $x\in\partial U_{Q'}^+\cap \partial\om_{\sbf^*}^+$, for some $Q'\in \sbf^*$ with
$Q_{j_0}\subset Q'$,  in which case $\delta_{\Omega}(x)\approx\ell(Q')\ge \ell(Q_{j_0})$.  Since also
$\delta_{\Omega}(x)\le |x-z|+\delta_{\Omega}(z)\lesssim_{\kappa, K}\ell(Q_{j_0})$, for any
$z\in B(x,r/2)\cap\pom^+_{Q_{j_0}}$, the claim follows.

On the other hand, since 
$x\in \pom_Q^+$, there is a $J\in\W$ with $\ell(J) \approx \ell(Q_{j_0})$, such that
$J^*$ is not contained in $\om_Q^+$.  We then have an exterior corkscrew point in
$J^* \cap B(x,r)$, and \eqref{eqvolumebound} follows in this case.

\smallskip

\noindent{\bf Case 3}: $B(x,r/2)$  meets $\pom_{Q_j}^+$ for at least one $Q_j \in \F^*_+$,
and $r > \kappa^{1/2} \ell(Q_{j_0})$, where as above $Q_{j_0}$ has the largest
length $\ell(Q_{j_0})$ among those $Q_j$ such that 
$\pom_{Q_j}^+$  meets $B(x,r/2)$.   In particular then,
$r \gg 2K \ell(Q_{j_0})=\diam(B^*_{Q_{j_0}})\geq \diam(\om^+_{Q_{j_0}})$, since we assume
$\kappa \gg K^4$.

We next claim that $B(x,r/4)$ contains some $x_1 \in \pom_{\sbf^*}^+ \cap
\pom_Q^+$.  This is clear if $x\in \partial\om_{\sbf^*}^+$ by taking $x_1=x$. Otherwise,  $x\in\partial\Omega_{Q_j}^+$ for some $Q_j\in\F^*$. Note that $U_{Q_j}^{\pm}\subset B(x_{Q_j}, K\ell(Q_j))\subset B(x, 2 K\ell(Q_{j}))$. Also,  
$U_{Q_j}^{\pm}\subset \om_{\sbf^*}^\pm$, by construction. 
On the other hand we note that if 
$z\in U_Q^{\pm}$ we have by \eqref{dist:UQ-pom}  
\[
|z-x_{Q_j}|\ge \delta_{\Omega}(z)\gtrsim \eta^{1/2}\ell(Q) 
\ge \eta^{1/2}2^{k_1}\ell(Q_{j})
\gg K\ell(Q_{j})
\]  
by our choice of $k_1$. 
By this fact, and
the definition of $\om_{\sbf^*}$, we have
$$U_Q^{\pm}
\subset  \om_{\sbf^*}^\pm\setminus B(x, 3 K\ell(Q_{j}))\,.$$ 
Using then that $\om_{\sbf^*}^\pm$ is connected, 
we see that  a path within $\om_{\sbf^*}^\pm$ joining  
$U_{Q_j}^{\pm}$ with $U_{Q}^{\pm}$ must meet $\partial B(x,3K\ell(Q_{j}))$.
Hence we can find  $y^\pm\in \om_{\sbf^*}^\pm\cap \partial B(x,3K\ell(Q_{j}))$. 
By Lemma \ref{lemma2.7},  $\om_{\sbf^*}^+$ and $\om_{\sbf^*}^-$ are disjoint 
(they live respectively above and below the graph $\Gamma_{\sbf}$), 
so a path joining 
$y^+$ and $y^-$  within   $\partial B(x,3K\ell(Q_{j}))$ meets some 
$x_1\in \partial \om_{\sbf^*}^+\cap \partial B(x,3K\ell(Q_{j}))$. 
On the other hand, 
$x_1\notin \overline{\Omega_{Q_j}^+}$, since 
$\overline{\Omega_{Q_j}^+}\subset \overline{B_{Q_j}^*}\subset 
B(x,3K\ell(Q_{j}))$. Furthermore,
$x_1\in \partial B(x,3K\ell(Q_{j}))\subset \kappa B_{Q_j}^*$, 
so by assumption (ii), we necessarily have that $x_1\notin \overline{\Omega_{Q_k}^+}$ for $k\neq j$. Thus, 
$x_1\in\partial\Omega_Q^+$, and moreover,
since  $B(x,r/2)$ meets $\partial \Omega_{Q_j}^+$ (at $x$) we have
$\ell(Q_j)\le\ell(Q_{j_0})$.   Therefore, 
$x_1$ is the claimed point, since in the current case 
$3K\ell(Q_{j})\le 3K\ell(Q_{j_0}) \ll r$.

With the point $x_1$ in hand, we note that 
\begin{equation}\label{eqballcontain}
B(x_1,r/4) \subset B(x,r/2)\quad \text{and } \quad B(x_1,r/2) \subset B(x,r)\,.
\end{equation}
By the exterior corkscrew condition for $\om_{\sbf^*}^+$, 
\begin{equation}\label{eqcs1}
\big|B(x_1,r/4) \setminus \overline{\om_{\sbf^*}^+} \,\big| \geq \,c_1 r^{n+1}\,,
\end{equation}
for some constant $c_1$ depending only on $n$ and the ADR/UR constants for $\pom$, by 
Lemma \ref{lemma2.7}.
Also, for each $\om^+_{Q_j}$ whose boundary meets 
$B(x_1,r/4) \setminus \overline{\om_{\sbf^*}^+}$ (and thus meets $B(x,r/2)$),
\begin{equation}\label{eq3.26}
\kappa^{1/4}\diam(B^*_{Q_j}) \leq 
\kappa^{1/4}\diam(B^*_{Q_{j_0}})\leq 2K\kappa^{1/4}\ell(Q_{j_0})
 \leq \frac{2K r}{\kappa^{1/4}} \ll r\,, 
 \end{equation}
in the present scenario. 
Consequently, $\kappa^{1/4}B_{Q_j}^* \subset B(x_1,r/2)$, for all such $Q_j$.

We now make the following claim.

\begin{claim}\label{claim1-ch}
On has
\begin{equation}\label{eqcs2}
\big|B(x_1,r/2) \setminus \overline{\om^+_Q}
\,\big| \geq \,c_2 r^{n+1}\,,
\end{equation}
for some $c_2>0$ depending only on allowable parameters.   
\end{claim}

Observe that by the second
containment in \eqref{eqballcontain},  we obtain \eqref{eqvolumebound} as an  immediate
consequence of \eqref{eqcs2}, and thus the proof will be complete once we have
established Claim \ref{claim1-ch}.

\begin{proof}[Proof of Claim \ref{claim1-ch}]

 To prove the claim, we
suppose first that
\begin{equation}\label{eqcs3}
\sum \big| B_{Q_j}^*  \setminus \overline{\om_{\sbf^*}^+} \,\big| \leq\, \frac{c_1}{2}\, r^{n+1}\,,
\end{equation}
where the sum runs over those $j$ such that
$\overline{B_{Q_j}^*}$ meets $B(x_1,r/4)\setminus   \overline{\om_{\sbf^*}^+}$, 
and $c_1$ is the constant in \eqref{eqcs1}.  In that case, \eqref{eqcs2} holds with
$c_2 = c_1/2$  (and even with
$B(x_1,r/4)$), by definition of $\om^+_Q$  (see \eqref{eqomdef}), and the fact that
$\om_{Q_j} \subset B_{Q_j}^*$.  
On the other hand, if
\eqref{eqcs3} fails, then summing over the same subset of indices $j$, we have
\begin{equation}\label{eqcs4}
C K \sum \ell(Q_j)^{n+1}   \geq \sum
\big| B_{Q_j}^*  \setminus \overline{\om_{\sbf^*}^+} \,\big| \geq\, \frac{c_1}{2}\, r^{n+1}
\end{equation}
We now make a second claim:

\begin{claim}\label{claim2-ch}
For $j$ appearing in the previous sum, we have 
\begin{equation}\label{eqcs5}
 \big| \left(\kappa^{1/4}B_{Q_j}^*\setminus B_{Q_j}^*\right)  
 \setminus \overline{\om_{\sbf^*}^+} \,\big| \geq\, c\, \ell(Q_j)^{n+1}\,,
\end{equation}
for some uniform $c>0$.  
\end{claim}

Taking the latter claim for granted momentarily, we insert 
estimate \eqref{eqcs5} into \eqref{eqcs4} and sum, to obtain
\begin{equation}\label{eqcs6}
\sum  \big| \left(\kappa^{1/4}B_{Q_j}^*\setminus B_{Q_j}^*\right)  
 \setminus \overline{\om_{\sbf^*}^+} \,\big|  \gtrsim r^{n+1}\,.
\end{equation}
By the separation property (ii), the balls $\kappa^{1/4} B^*_{Q_{j}}$ are pairwise disjoint, and
by assumption $\om_{Q_j}^+\subset B_{Q_j}^*$.  Thus,
for any given $j_1$,  $\kappa^{1/4} B^*_{Q_{j_1}} \setminus \overline{B^*_{Q_{j_1}}}$ misses 
$\cup_j \overline{\Omega^+_{Q_j}}$.   Moreover, as noted above (see \eqref{eq3.26}
and the ensuing comment), $\kappa^{1/4}B_{Q_j}^* \subset B(x_1,r/2)$ for each $j$
under consideration in \eqref{eqcs3}-\eqref{eqcs6}.  Claim \ref{claim1-ch} now follows.
\end{proof}

\begin{proof}[Proof of Claim \ref{claim2-ch}]
There are two cases:  if $\frac12 \kappa^{1/4} B_{Q_j}^* \subset
\ree\setminus \overline{\om_{\sbf^*}^+}$, then \eqref{eqcs5} is trivial, since $\kappa \gg 1$.  Otherwise,
$\frac12 \kappa^{1/4} B_{Q_j}^*$ contains a point $z\in \pom_{\sbf^*}^+$.  In the latter case,
by the exterior corkscrew condition for $\om_{\sbf^*}^+$, 
\[\big| B\big(z, 2^{-1}\kappa^{1/4}K \ell(Q_j)\big) \setminus \overline{\om_{\sbf^*}^+}\,\big|
\,\gtrsim \,\kappa^{(n+1)/4}\big(K\ell(Q_j)\big)^{n+1}  \gg \, |B_{Q_j}^*|\,,
\]
since $\kappa \gg 1$.   On the other hand, $ B\big(z, 2^{-1}\kappa^{1/4}K \ell(Q_j)\big)
\subset \kappa^{1/4} B_{Q_j}^*$, and \eqref{eqcs5} follows, finishing the proof of Claim  \ref{claim2-ch}.
\end{proof}

Next, \eqref{eqballcontain} and \eqref{eqcs2} yield \eqref{eqvolumebound} in the present case and hence
the proof of Lemma \ref{lemmaCAD} is complete.
\end{proof}

\section{Step 2:  Proof of $H[M_0,1]$}\label{ss3.2}  We shall deduce $H[M_0,1]$ (see Section \ref{ss-indhyp})
from the following pair of claims.

\begin{claim}\label{claim3.17} $H[0,\theta]$ holds for every $\theta\in(0,1]$.
\end{claim}
\begin{proof}[Proof of Claim \ref{claim3.17}] 
	If $a=0$ in \eqref{eq3.12}, then $\|\mut\|_{\C(Q)}=0$, whence it follows by Claim \ref{claim3.16}, with
	$\F=\emptyset$, that there is a tree $\sbf\subset \G$, with $\dd(Q)\subset \sbf$.  
	Hence $\sbf':=\dd(Q)$ is a coherent subtree of $\sbf$,  so by 
	Lemma \ref{lemma2.7}, 
	each of $\om_{\sbf'}^\pm$ 
	is a CAD, containing $U_Q^\pm$, respectively, with $\om_{\sbf'}^\pm\subset B_Q^*$ by \eqref{def:BQ*}.  
	 Moreover,  by \cite[Proposition A.14]{HMM}
$$
Q \subset
\pom_{\sbf'}^\pm \cap\pom\,,
$$ 
so that $\sigma(Q)\leq \sigma(\pom_{\sbf'}^\pm \cap\pom).$ 
	Thus, $H[0,\theta]$ holds trivially.
\end{proof}

\begin{claim}\label{claim3.18}
	There is a uniform constant $b>0$ such that
	$H[a,1] \implies H[a+b,1]$, for all $a\in [0,M_0)$.
\end{claim}
Combining Claims \ref{claim3.17} and \ref{claim3.18}, we find that $H[M_0,1]$ holds.

To prove Claim \ref{claim3.18}, we shall require the following.
\begin{lemma}[{\cite[Lemma 7.2]{HM-I}}]\label{lemma:Corona}
	Suppose that $E$ is an $n$-ADR set,
	and let $\mut$ be a discrete Carleson measure, as in \eqref{eq4.1}-\eqref{eq4.7a}
	above.   Fix $Q\in \dd(E)$.
	Let $a\geq 0$ and $b>0$, and suppose that
	$\mut\big(\dd(Q)\big)\leq (a+b)\,\sigma(Q).$
	Then there is a family $\F=\{Q_j\}\subset\dd(Q)$
	of pairwise disjoint cubes, and a constant $C$ depending only on $n$
	and the ADR constant such that
	\begin{equation} \label{Corona-sawtooth}
	\|\mut_\F\|_{\C(Q)}
	\leq C b,
	\end{equation}
	\begin{equation}
	\label{Corona-bad-cubes}
	\sigma \bigg(\bigcup_{\F_{bad}}Q_j\bigg)
	\leq \frac{a+b}{a+2b}\, \sigma(Q)\,,
	\end{equation}
	where $\F_{bad}:=
	\{Q_j\in\F:\,\mut\big(\dd(Q_j)\setminus \{Q_j\}\big)>\,a\sigma(Q_j)\}$.
\end{lemma}

We refer the reader to \cite[Lemma 7.2]{HM-I} for the proof.   We remark that the lemma is stated
in \cite{HM-I} in the case that $E$ is the boundary of a connected domain, but the proof actually requires only that
$E$ have a dyadic cube structure, and that $\sigma$ be a
non-negative, dyadically doubling Borel measure on $E$.  In our case, we shall of course apply the lemma with
$E=\pom$, where $\om$ is open, but not necessarily connected.

\begin{proof}[Proof of Claim \ref{claim3.18}]  We assume that $H[a,1]$ holds, for some  $a\in[0,M_0)$.
	Let us set $b= 1/(2C)$, where $C$ is the constant in
	\eqref{Corona-sawtooth}.
	Consider a cube $Q\in \dd(\pom)$ with
	$\mut(\dd(Q)) \leq (a+b) \sigma(Q)$.   Suppose that there is a set $V_Q\subset U_Q\cap\om$ such that
	\eqref{eq3.10} holds with $\theta=1$.  We fix $k_1>k_0$ (see \eqref{eq3.15}) large 
	enough so that $2^{k_1}> 100 K$. 

	\noindent{\bf Case 1}: There exists $Q'\in\dd_{k_1}(Q)$ (see \eqref{eq3.4aaa})
	with  $\mut\big(\dd(Q')\big) \le a\sigma(Q')$.  
	
	In the present scenario $\theta=1$, that is, $\sigma(F_Q)=\sigma(Q)$ (see 
	\eqref{eq3.10} and \eqref{defi-FQ}), which implies $\sigma(F_Q\cap Q')=\sigma(Q')$. We apply Lemma \ref{lemma:VQ} to obtain $V_{Q'}\subset U_{Q'}\cap\Omega$ and the corresponding $F_{Q'}$ which satisfies $\sigma(F_{Q'})=\sigma(Q')$. That is, \eqref{eq3.10} holds for $Q'$, with $\theta=1$.  Consequently, we may 
	apply the induction hypothesis
	to $Q'$, to find $V^*_{Q'}\subset V_{Q'}$, such 
	that for each $U_{Q'}^i$ meeting $V^*_{Q'}$, there is a Chord-arc domain 
	$\Omega_{Q'}^i \supset U_{Q'}^i$ formed by a union of fattened Whitney cubes with $\Omega^i_{Q'}\subset B(x_Q',K\ell(Q'))\cap\Omega$, and 
\begin{equation}\label{eq3.34}
\sum_{i:  U_{Q'}^i \text{ meets } V^*_{Q'}}\sigma (\pom^i_{Q'}\cap Q')\geq  c_a\sigma(Q')\,.
\end{equation}
	By  Lemma \ref{lemma:VQ}, and since $k_1>k_0$, each $y\in V^*_{Q'}$ lies 
	on a $\lambda$-carrot path connecting some $y\in Q'$ to some $x\in V_Q$; let $V^{**}_Q$ 
	denote the 
	set of all such $x$, and let ${\bf U}^{**}_Q$ (respectively, $\sub_{Q'}$) denote the collection of
	connected components of $U_Q$ (resp., of
	$U_{Q'}$) which meet $V^{**}_Q$ (resp., $V^*_{Q'}$).  
	By construction, each component $U_{Q'}^i \in \sub_{Q'}$
	may be joined to some corresponding component in ${\bf U}^{**}_Q$,
	via one of the 
	carrot paths.
	After possible
	renumbering, we designate this component as $U_Q^i$,  we let
	$x_i, y_i$ denote the points in $V_Q^{**}\cap U_Q^i$
	and in $V_Q^*\cap U_{Q'}^i$, respectively, that are joined by this carrot path,
	and we let $\gamma_i$ be the portion of the carrot path joining $x_i$ to $y_i$
	(if there is more than one such path or component, we just pick one).  We also let
	$V_Q^*=\{x_i\}_i$ be the collection of all of the selected points $x_i$.  We let 
	$\W_i$ be the collection of Whitney cubes meeting
	$\gamma_i$, and we then define 
$$\Omega_Q^i:= \Omega^i_{Q'} \bigcup \interior\bigg(\bigcup_{I\in\W_i}\,I^*\bigg) \bigcup
U_{Q}^i\,.$$
By the definition of a $\lambda$-carrot path, since 
$\ell(Q')\approx_{k_1}\ell(Q)$, and since  $ \Omega^i_{Q'}$ is a CAD,
one may readily verify that $\Omega^i_Q$ is also a CAD consisting of a union $\cup_k I_k^*$ 
of fattened Whitney cubes $I_k^*$.
We omit the details.  Moreover, by construction, 
$$\pom_Q^i\cap Q \supset \pom_{Q'}^i\cap Q',$$
so that the analogue of \eqref{eq3.34} holds with $Q'$ replaced by $Q$, and with $c_a$ replaced by 
$c_{k_1} c_a$.

It remains to verify that $\Omega_Q^i \subset B_Q^* = B(x_Q, K \ell(Q))$.
	By the induction hypothesis, and our choice of $k_1$,  since $\ell(Q')=2^{-k_1}\ell(Q)$ we have
	\[
	\Omega^i_{Q'}\subset B_{Q'}^*\cap\Omega = B(x_{Q'},K\ell(Q'))\cap\Omega
\subset B_Q^*\cap\Omega.
	\]
Moreover, $U_Q\subset B_Q^*$, by \eqref{def:BQ*}.   We therefore need only to consider $I^*$ with $I\in\W_i$.
For such an $I$, by definition there is a point $z_i\in I\cap\gamma_i$ and $y_i\in Q'$, so that
	$z_i\in\gamma(y_i,x_i)$ and thus, 
	\[
	\delta_{\Omega}(z_i)
	\le
	|z_i-y_i|
	\le
	\ell(y_i,z_i)
	\le
	\ell(y_i,x_i)
	\le
	\lambda^{-1}\delta_{\Omega}(x_i)
	\le
	\lambda^{-1}|x_i-x_Q|
	\le
	\lambda^{-1} CK^{1/2}\ell(Q)\,,
	\]
where in the last inequality we have used \eqref{dist:UQ-pom} and the fact that $x_i\in U_Q$.	
Hence, for every $z\in I^*$ by \eqref{Whintey-4I}
	\[
	|z-x_Q|
	\le
	\diam(2I)+|z_i-y_i|+|y_i-x_Q|
	\le
	C|z_i-y_i|+\diam(Q)
	<K\ell(Q),
	\]
	by our choice of the parameters $K$ and $\lambda$. 

We then obtain the conclusion of $H[a+b,1]$ in the present case.
	
	\smallskip

	\noindent{\bf Case 2}:  $\mut\big(\dd(Q')\big) > a\sigma(Q')$ for every $Q'\in\dd_{k_1}(Q)$.

	In this case, we apply Lemma \ref{lemma:Corona}
	to obtain a pairwise disjoint family $\F=\{Q_j\}\subset \dd(Q)$ such 
	that \eqref{Corona-sawtooth} and \eqref{Corona-bad-cubes} hold.	
	In particular, by our choice of $b=1/(2C)$,
	\begin{equation}\label{eq3.27bb}
	\|\mut_\F\|_{\C(Q)}\leq 1/2\,,
	\end{equation}
	so that the conclusions of Claim \ref{claim3.16} hold. 
	
	We set 
	\begin{equation}\label{eq3.27aa}
	F_0:= Q\setminus \bigg(\bigcup_\F Q_j\bigg)\,,
	\end{equation}
	define
	\begin{equation}\label{eq3.27a}\F_{good}:= \F\setminus \F_{bad}=
	\left\{Q_j\in\F:\,\mut\left(\dd(Q_j)\setminus \{Q_j\}\right)\leq\,a\sigma(Q_j)\right\}\,,
	\end{equation}
	and let
	$$G_0:= \bigcup_{\F_{good}}Q_j\,.$$
	Then by \eqref{Corona-bad-cubes}
	\begin{equation}\label{eq3.25}
	\sigma(F_0 \cup G_0 )\, \geq \,\rho\sigma(Q)\,,
	\end{equation}
	where $\rho\in (0,1)$ is defined by
	\begin{equation}\label{eq3.42aa} \frac{a+b}{a+2b} \leq \frac{M_0+b}{M_0+2b}=:1-\rho \in (0,1)\,.
	\end{equation}
	We claim that 
	\begin{equation}\label{eq:erwgte}
	\ell(Q_j)\le 2^{-k_1}\,\ell(Q),
	\qquad
	\forall\,Q_j\in\F_{good}.
	\end{equation}
	Indeed,  were this not true for some $Q_j$, 
	then by definition of $\F_{good}$ and pigeon-holing there 
	will be $Q_j'\in\dd(Q_j)$ with $\ell(Q_j')=2^{-k_1}\,\ell(Q)$ such 
	that $\mut\big(\dd(Q_j')\big)\le a\,\sigma(Q_j')$. This contradicts the assumptions of the current case. 
	
	Note also that $Q\notin\F_{good}$ by \eqref{eq:erwgte} and $Q\notin\F_{bad}$ by \eqref{Corona-bad-cubes}, hence  $\F\subset\dd(Q)\setminus\{Q\}$. By \eqref{eq3.27bb} and Claim \ref{claim3.16}, 
	there is some  tree $\sbf\subset\G$ so that 
	$\sbf''=(\dd_\F\cup\F\cup\F')\cap\dd(Q)$  
	is a semi-coherent subtree of $\sbf$,  where $\F'$ denotes the collection of all
	dyadic children of cubes in $\F$.
	
	\smallskip
	
	\noindent{\bf Case 2a}:   $\sigma(F_0) \geq \frac12\,  \rho \sigma(Q)$.
	
	In this case,
	$Q$ has an ample overlap with the boundary 
	of a Chord-arc domain with controlled Chord-arc constants. Indeed, 
	let $\sbf'=\dd_\F\cap\dd(Q)$ which, by \eqref{eq3.27bb} and Claim \ref{claim3.16}, is a 
	semi-coherent subtree of some $\sbf\subset\G$. Hence, by Lemma \ref{lemma2.7},
	each of $\om_{\sbf'}^\pm$ is a CAD
with constants depending on the allowable parameters, formed by the union of fattened Whitney boxes, 
which satisfies
$\om_{\sbf'}^\pm\subset B_Q^*\cap\Omega$  (see \eqref{eq3.3aa}, \eqref{eq3.2}, and \eqref{def:BQ*}). Moreover, by 
\cite[Proposition A.14]{HMM} and \cite[Proposition 6.3]{HM-I} and our current assumptions,
	\[ 
	\sigma(Q\cap \pom_{\sbf'}^\pm)= \sigma(F_0)\geq \frac{\rho}{2} \sigma(Q)\,.
	\]
	Recall that in establishing $H[a+b,1]$, we assume that 
	there is a set $V_Q\subset U_Q\cap\om$ for which
	\eqref{eq3.10} holds with $\theta=1$. Pick then $x\in V_Q$ and set $V_Q^*:=\{x\}\subset V_Q$. Note that since $U_Q=U_Q^+\cup U_Q^-$ it follows that $x$ belongs to either $U_Q^+\cap\Omega$ or $U_Q^-\cap\Omega$. For the sake of specificity assume that $x\in U_Q^+\cap\Omega$ hence, in particular, 
	$U_Q^+\subset\om_{\sbf'}^+\subset \om$.  Note also that $U_Q^+$ is the only component of $U_Q$ meeting $V_Q^*$. All these together give at once that the conclusion of $H[a+b,1]$ holds in the present case.
	
	\smallskip
	
	\noindent{\bf Case 2b}:   $\sigma(F_0) < \frac12 \,\rho \sigma(Q)$.
	
	In this case  by \eqref{eq3.25} 
	\begin{equation}\label{eq3.33}
	\sigma(G_0)\,\geq \,\frac{\rho}{2}\, 
	\sigma(Q)\,.
	\end{equation}
	In addition, by the definition of $\F_{good}$ \eqref{eq3.27a}, and pigeon-holing, every $Q_j\in\F_{good}$ has a dyadic
	child $Q_j'$ (there could be more children satisfying this, but we just pick one) so that 
	\begin{equation}\label{eq3.27}
	\mut\big(\dd(Q'_j)\big) \leq a\sigma(Q'_j)\,.
	\end{equation}
	Under the present assumptions $\theta=1$, that is, $\sigma(F_Q)=\sigma(Q)$ (see
	\eqref{eq3.10} and \eqref{defi-FQ}), 
	hence $\sigma(F_Q\cap Q_j')=\sigma(Q_j')$. 
	We apply Lemma \ref{lemma:VQ} (recall \eqref{eq:erwgte}) to obtain  $V_{Q_j'}\subset U_{Q_j'}\cap \Omega$ and 
	$F_{Q_j'}$ which satisfies $\sigma(F_{Q_j'})=\sigma(Q_j')$. That is, \eqref{eq3.10} holds 
	for $Q_j'$, with $\theta=1$. 
	Consequently, recalling that 
	$Q_j'\in\sbf\subset\G$ (see Claim \ref{claim3.16}),
	and applying the induction hypothesis
	to $Q_j'$, we find $V^*_{Q_j'}\subset V_{Q_j'}$, such 
	that for each $U_{Q_j'}^\pm$ meeting $V^*_{Q_j'}$, there is a Chord-arc domain 
	$\Omega_{Q_j'}^\pm \supset U_{Q_j'}^\pm$ formed by a union of fattened Whitney cubes with $\Omega^\pm_{Q_j'}\subset B_{Q_j'}^*\cap\Omega$. 
	Moreover, since in particular, the cubes in $\F$ along with all of their children
	belong to the same tree $\sbf$ (see Claim \ref{claim3.16}),
	 the connected component $U_{Q_j}^\pm$ overlaps with
the corresponding component $U^\pm_{Q_j'}$ for its child,  so we may augment
$\Omega_{Q_j'}^\pm$ by adjoining to it the appropriate component $U_{Q_j}^\pm$, to form
a chord arc domain 
\begin{equation}\label{cadbuild}
\Omega^\pm_{Q_j} := \Omega_{Q_j'}^\pm \cup U_{Q_j}^\pm\,.
\end{equation}
Moreover, since  $K\gg 1$, and since $Q_j'\subset Q_j$, we have that
$B^*_{Q_j'} \subset B^*_{Q_j}$, hence $\Omega^\pm_{Q_j}\subset B_{Q_j}^*$ by construction.  

	
By a covering lemma argument, for a sufficiently large constant $\kappa\gg K^4$, 
	we may extract a 
	subcollection $\F_{good}^*\subset \F_{good}$ so that 
	$ \{ \kappa B_{Q_j}^*\}_{Q_j\in \F_{good}^*}$ is a pairwise disjoint family, and 
	\[
	\bigcup_{ Q_j\in \F_{good}} Q_j \subset\bigcup_{ Q_j\in \F_{good}^*} 5\kappa B_{Q_j}^*.
	\]
	In particular, by \eqref{eq3.33},
	\begin{equation}\label{G0*}
\sum_{ Q_j\in \F_{good}^*}\sigma( Q_j)
	\gtrsim_{\kappa,K}
\sum_{ Q_j\in \F_{good}}\sigma( Q_j) =\sigma(G_0)
	\gtrsim
\rho\sigma(Q),
	\end{equation}
	where the implicit constants depend  on ADR, $K$, and  
	the dilation factor $\kappa$. 



By the induction hypothesis, and by construction \eqref{cadbuild} and $n$-ADR,
\begin{equation}\label{eqx}
\sigma (Q_j\cap \pom_{Q_j})  \gtrsim \sigma(Q'_j)\gtrsim \sigma(Q_j)\,,
\end{equation}
where  $\Omega_{Q_j} $ is equal either to $\Omega_{Q_j}^+$ or to $\Omega_{Q_j}^-$
(if \eqref{eqx} holds for both choices, we arbitrarily set $\Omega_{Q_j}=\Omega_{Q_j}^+$).

Combining \eqref{eqx} with \eqref{G0*},
we obtain
\begin{equation}\label{eqxx}
\sum_{Q_j \in \F^*_{good}} \sigma (Q_j\cap \pom_{Q_j})  \gtrsim \sigma(Q)\,.
\end{equation}

We now assign each $Q_j\in \F_{good}^*$ either to
$\F^*_+$ or to $\F^*_-$, depending on whether we chose $\Omega_{Q_j}$ satisfying
\eqref{eqx} to be $\Omega_{Q_j}^+$, or $\Omega_{Q_j}^-$.  
We note that at least one of the sub-collections $\F^*_\pm$ 
is non-empty, since
for each $j$, there was at least one choice of ``+' or ``-" such that
\eqref{eqx} holds for the corresponding choice of $\Omega_{Q_j}$.  Moreover, the 
two collections are disjoint, since we have arbitrarily designated $\Omega_{Q_j} = \Omega_{Q_j}^+$
in the case that there were two choices for a particular $Q_j$.


	To proceed, as in Lemma \ref{lemmaCAD} we set 
	\[
	\sbf^*=\left\{Q'\in \dd(Q): Q_j\subset Q' \text{ for some }Q_j\in\F_{good}^*\right\}
	\] 
which is semi-coherent by construction. 
For $\F^*_\pm$ non-empty, we now define
\begin{equation}\label{defi:Omega-Q-pm}
	\Omega_{Q}^{\pm}
	=
	\Omega_{\sbf^*}^{\pm}\bigcup 
	\Big(\bigcup_{Q_j\in \F_{\pm}^{*}} \Omega_{Q_j}\Big).
	\end{equation}
Observe that by the induction hypothesis, 
 and our construction (see \eqref{cadbuild} and the ensuing comment),
 for an appropriate choice of $\pm$,
 $U_{Q_j}^\pm\subset \Omega_{Q_j} \subset B_{Q_j}^*$, and since $\ell(Q_j) \leq 2^{-k_1} \ell(Q)$,
by \eqref{eqxx} and Lemma  \ref{lemmaCAD},  with $\F^*=\F^*_{good}$,
each (non-empty) choice of $\Omega_Q^\pm$
defines a Chord-arc domain with the requisite properties.

	Thus, we have proved Claim \ref{claim3.18}
	and therefore, as noted above, it follows that $H[M_0,1]$ holds.
\end{proof}

\section{Step 3:  bootstrapping $\theta$} \label{s-step3}
In this last step, we shall prove 
that there is a uniform constant $\zeta\in (0,1)$ such that for each $\theta\in (0,1]$,
$H[M_0,\theta]\implies H[M_0,\zeta\theta]$.  Since we have already established
$H[M_0,1]$, we then conclude that $H[M_0,\theta_1]$ holds for any given
$\theta_1\in (0,1]$.  As noted above, it then follows that Theorem
\ref{t1} holds, as desired.

In turn, it will be enough to verify the following.  
\begin{claim}\label{claim3.52}   There is a uniform constant $\beta \in(0,1)$ such that
for every $a\in [0,M_0)$, $\theta\in (0,1]$, $\vartheta\in (0,1)$, and $b=1/(2C)$ as in Step 2/Proof of Claim \ref{claim3.18}, 
if $H[M_0,\theta]$ holds, then
$$H[a,(1-\vartheta)\theta]\implies H[a+b, (1-\vartheta\beta)\theta]\,.$$
\end{claim}
Let us momentarily take Claim \ref{claim3.52} for granted.  Recall that by Claim \ref{claim3.17},
$H[0,\theta]$ holds for all $\theta\in (0,1]$.  In particular, given $\theta \in (0,1]$ fixed, for
which $H[M_0,\theta]$ holds, we have that
$H[0,\theta/2]$ holds.    
Combining the latter fact with Claim \ref{claim3.52}, and iterating,
we obtain that $H[kb, (1-2^{-1}\beta^k)\theta]$ holds.  We eventually reach
$H[M_0,(1-2^{-1}\beta^\nu)\theta]$, with $\nu \approx M_0/b$.  The conclusion of Step 3 now follows,
with $\zeta := 1-2^{-1}\beta^\nu$.

\begin{proof}[Proof of Claim \ref{claim3.52}]
The proof will be a refinement of that of Claim \ref{claim3.18}.  We are given some $\theta\in(0,1]$ such that
$H[M_0,\theta]$ holds, and
we assume that $H[a,(1-\vartheta)\theta]$ holds, for some  $a\in[0,M_0)$ and $\vartheta\in (0,1)$.
Set $b= 1/(2C)$, where as before $C$ is the constant in
\eqref{Corona-sawtooth}.
Consider a cube $Q\in \dd(\pom)$ with
$\mut(\dd(Q)) \leq (a+b) \sigma(Q)$.   Suppose that there is a set $V_Q\subset U_Q\cap\om$ such that
\eqref{eq3.10} holds with $\theta$ replaced by
$(1-\vartheta\beta)\theta$, for some $\beta\in (0,1)$ to be determined.
Our goal is to show that for a sufficiently small, but uniform choice of $\beta$,
we may deduce the conclusion of the induction hypothesis, with $\C_{a+b}, c_{a+b}$ in place of 
$C_a,c_a$.

By assumption, and recalling the definition of $F_Q$ in \eqref{defi-FQ}, 
we have that \eqref{eq3.10} holds with constant $(1-\vartheta\beta)\theta$, i.e.,
\begin{equation}\label{eq3.56}
\sigma(F_Q) \geq (1-\vartheta\beta)\theta \sigma(Q)\,.
\end{equation}

As in the proof of Claim \ref{claim3.18}, we fix $k_1>k_0$ (see \eqref{eq3.15}) large 
enough so that $2^{k_1}> 100K$. 
There are two principal cases.  The first is as follows.

\smallskip


\noindent{\bf Case 1}: There exists $Q'\in\dd_{k_1}(Q)$ (see \eqref{eq3.4aaa}) with  
$\mut\big(\dd(Q')\big) \le a\sigma(Q')$.


We split \textbf{Case 1} into two subcases.

\smallskip

\noindent{\bf Case 1a}:  $\sigma(F_Q\cap Q') \geq (1-\vartheta)\theta \sigma(Q')$.

In this case, we follow the \textbf{Case 1} argument for $\theta=1$ in 
Section \ref{ss3.2} {\it mutatis mutandis}, so we merely sketch the proof.
By Lemma \ref{lemma:VQ}, we may construct $V_{Q'}$ and $F_{Q'}$ so that
$F_Q\cap Q' = F_{Q'}$ and hence $\sigma(F_{Q'})\ge (1-\vartheta)\theta \sigma(Q')$. 
We may then apply the induction hypothesis
$H[a,(1-\vartheta)\theta]$ in $Q'$, and then proceed exactly as in \textbf{Case 1}  in 
Section \ref{ss3.2}  to construct
a subset $V_Q^*\subset V_Q$ and a family of Chord-arc domains $\Omega_Q^i$ satisfying the
various desired properties, and such that
$$\sum_{i:  U_{Q}^i \text{ meets } V^*_{Q}}\sigma (\pom^i_{Q}\cap Q)\geq c_a\sigma(Q')
\gtrsim_{k_1}c_a \sigma(Q)\,.$$

The conclusion of $H[a+b, (1-\vartheta\beta)\theta]$ then holds in the present scenario.

\smallskip

\noindent{\bf Case 1b}:  $\sigma(F_Q\cap Q') < (1-\vartheta)\theta \sigma(Q')$.

By \eqref{eq3.56}
$$(1-\vartheta\beta)\theta \sigma(Q)\, \leq \,\sigma(F_Q) \,=\, \sigma(F_Q\cap Q') + 
\sum_{Q''\in\dd_{k_1}(Q)\setminus\{Q'\}}\sigma(F_Q\cap Q'')\,.$$
In the scenario of \textbf{Case 1b}, this leads to
\begin{multline*}
(1-\vartheta\beta)\theta \sigma(Q')+ (1-\vartheta\beta)\theta
\sum_{Q''\in\dd_{k_1}(Q)\setminus\{Q'\}}\sigma(Q'')=
(1-\vartheta\beta)\theta \sigma(Q)\\[4pt] \leq \, (1-\vartheta)\theta \sigma(Q') + 
\sum_{Q''\in\dd_{k_1}(Q)\setminus\{Q'\}}\sigma(F_Q\cap Q'')\,,
\end{multline*}
that is,
\begin{equation}\label{eq3.60}
(1-\beta)\vartheta\theta \sigma(Q')+ (1-\vartheta\beta)\theta
\sum_{Q''\in\dd_{k_1}(Q)\setminus\{Q'\}}\sigma(Q'')\,\leq\, \sum_{Q''\in\dd_{k_1}(Q)\setminus\{Q'\}}\sigma(F_Q\cap Q'')\,.
\end{equation}
Note that we have the dyadic doubling estimate
$$\sum_{Q''\in\dd_{k_1}(Q)\setminus\{Q'\}}\sigma(Q'') \leq \sigma(Q) \leq M_1 \sigma(Q')\,,$$
where $M_1=M_1(k_1,n, ADR)$.  Combining this estimate with \eqref{eq3.60}, we obtain
$$
\left[(1-\beta)\frac{\vartheta}{M_1} + (1-\vartheta\beta)\right] \theta\sum_{Q''\in\dd_{k_1}(Q)
\setminus\{Q'\}}\sigma(Q'')
\leq \sum_{Q''\in\dd_{k_1}(Q)\setminus\{Q'\}}\sigma(F_Q\cap Q'')\,.$$
We now choose $\beta\leq 1/(M_1+1)$, so that $(1-\beta)/M_1 \geq \beta$, and therefore the
expression in square brackets is at least 1.  Consequently, by pigeon-holing, there exists
a particular $Q''_0\in \dd_{k_1}(Q)\setminus\{Q'\}$ such that 
\begin{equation}\label{eq3.62}
\theta\sigma(Q''_0) \leq \sigma(F_Q \cap Q_0'')\,.
\end{equation}
By Lemma \ref{lemma:VQ}, we can find $V_{Q''_0}$ such that
$F_Q\cap Q_0'' = F_{Q_0''}$,  where the latter is defined as in \eqref{defi-FQ}, with
$Q_0''$ in place of $Q$.  By assumption, $H[M_0,\theta]$ holds, so combining
 \eqref{eq3.62} with the fact that \eqref{eq3.12} holds with
$a=M_0$ for every $Q\in\dd(\pom)$, we find that there exists a subset 
$V^*_{Q''_0}\subset V_{Q''_0}$, along with a family of Chord-arc domains 
$\{\Omega_{Q_0''}^i\}_i$ enjoying all the appropriate properties relative to $Q_0''$.
Using that $\ell(Q_0'')\approx_{k_1} \ell(Q)$, we may now proceed
exactly as in \textbf{Case 1a} above, and also \textbf{Case 1} in  
Section \ref{ss3.2},
to construct $V_Q^*$ and $\{\Omega_Q^i\}_i$ such that the conclusion of 
$H[a+b, (1-\vartheta\beta)\theta]$ holds in the present case also.

\smallskip

\noindent{\bf Case 2}:  $\mut\big(\dd(Q')\big) > a\sigma(Q')$ for every $Q'\in\dd_{k_1}(Q)$.

In this case, we apply Lemma \ref{lemma:Corona}
to obtain a pairwise disjoint family $\F=\{Q_j\}\subset \dd(Q)$ such 
that \eqref{Corona-sawtooth} and \eqref{Corona-bad-cubes} hold.
In particular, by our choice of $b=1/(2C)$,
$\|\mut_\F\|_{\C(Q)}\leq 1/2$. 

Recall that 
$F_Q$ is defined in \eqref{defi-FQ}, and satisfies \eqref{eq3.56}. 
We define $F_0  =Q\setminus (\bigcup_\F Q_j)$ as in \eqref{eq3.27aa},
and $\F_{good}:= \F\setminus \F_{bad}$ as in \eqref{eq3.27a}.
Let
$G_0:= \bigcup_{\F_{good}}Q_j$.  Then as above (see \eqref{eq3.25}),
\begin{equation}\label{eq3.63}
\sigma(F_0 \cup G_0 )\, \geq \,\rho\sigma(Q)\,,
\end{equation}
where again $\rho=\rho(M_0,b)\in (0,1)$ is defined as in \eqref{eq3.42aa}.  Just 
as in \textbf{Case 2} for $\theta=1$ in 
Section \ref{ss3.2}, we have that
\begin{equation}
\ell(Q_j)\le 2^{-k_1}\,\ell(Q),
\quad
\forall\,Q_j\in\F_{good},
\qquad
\mbox{and}
\qquad
\F\subset\dd(Q)\setminus\{Q\}
\label{eq:wefwerfr}
\end{equation}
(see \eqref{eq:erwgte}).
Hence,  the conclusions of Claim \ref{claim3.16} hold.

We first observe that if
$\sigma(F_0) \geq \eps \sigma(Q)$, for some $\eps>0$ to be chosen
(depending on allowable parameters), then 
the desired conclusion holds.  Indeed, in this case, we may proceed exactly as in
the analogous scenario in \textbf{Case 2a} in Section \ref{ss3.2}:  the promised Chord-arc domain is again
simply one of $\Omega_{\sbf'}^\pm$, since at least one of these contains a point in $V_Q$ 
and hence in particular is a subdomain of $\Omega$.
The constant $c_{a+b}$ in our conclusion will depend on $\eps$, but in the end this will be harmless,
since $\eps$ will be chosen to depend only on allowable parameters.

We may therefore suppose that 
\begin{equation}\label{eq3.65}
\sigma(F_0) < \eps \sigma(Q)\,.
\end{equation}
Next, we refine the decomposition $\F = \F_{good}\cup \F_{bad}$. 
With $\rho$ as in \eqref{eq3.42aa} and \eqref{eq3.63}, we choose $\beta< \rho/4$.
Set
$$\F^{(1)}_{good}:= \left\{Q_j\in\F_{good}:\, 
\sigma(F_Q\cap Q_j)\geq \big(1-4\vartheta\beta\rho^{-1}\big)\theta\sigma(Q_j)\right\}\,,$$
and define $\F^{(2)}_{good}:= \F_{good}\setminus \F_{good}^{(1)}$.
Let
$$\F^{(1)}_{bad}:= \left\{Q_j\in\F_{bad}:\,
\sigma(F_Q\cap Q_j)\geq \theta\sigma(Q_j)\right\}\,,$$
and define $\F^{(2)}_{bad}:= \F_{bad}\setminus \F_{bad}^{(1)}$.

We split the remaining part of \textbf{Case 2} into two subcases.  The first of these will be easy,
based on our previous arguments.

\smallskip

\noindent{\bf Case 2a}:  There is $Q_j\in \F^{(1)}_{bad}$ such that $\ell(Q_j)>2^{-k_1}\,\ell(Q)$.

By definition of $\F_{bad}^{(1)}$, one has $\sigma(F_Q\cap Q_j)\geq \theta\sigma(Q_j)$. By 
pigeon-holing, $Q_j$ has a descendant
$Q'$ with $\ell(Q') = 2^{-k_1}\ell(Q)$, such that $\sigma(F_Q\cap Q')\geq \theta\sigma(Q')$. 
We may then  apply $H[M_0,\theta]$ in $Q'$, and proceed exactly as we did in \textbf{Case 1b} above with
the cube $Q_0''$, which enjoyed precisely the same properties as does our current $Q'$.
Thus, we draw the desired conclusion in the present case.

\smallskip

The main case is the following.

\noindent{\bf Case 2b}:  Every $Q_j\in \F^{(1)}_{bad}$ satisfies $\ell(Q_j)\le 2^{-k_1}\,\ell(Q)$.

Observe that by definition,
\begin{equation}\label{eq3.67}
\sigma(F_Q \cap Q_j) \leq \big(1-4\vartheta\beta\rho^{-1}\big)\theta \sigma(Q_j)\,,\qquad \forall \, Q_j \in \F^{(2)}_{good}\,,
\end{equation}
and also
\begin{equation}\label{eq3.68a}
\sigma(F_Q \cap Q_j) \leq \theta \sigma(Q_j)\,,\qquad \forall \, Q_j \in \F^{(2)}_{bad}\,,
\end{equation}

Set $\F_*:=\F\setminus \F_{good}^{(2)}$.  For future reference, we shall derive a certain ampleness estimate for
the cubes in $\F_*$. By \eqref{eq3.56},
\begin{multline}\label{eq3.68}
(1-\vartheta\beta)\theta \sigma(Q)
\leq \, 
\sigma(F_Q)\, \leq \, \sigma(F_0) + \sum_{\F_*}\sigma(Q_j) + \sum_{\F_{good}^{(2)}}\sigma(F_Q\cap Q_j)\\[4pt]
\leq\, \eps\sigma(Q) + \sum_{\F_*}\sigma(Q_j) +
\left(1-4\vartheta\beta\rho^{-1}\right)\theta \sigma(Q)
\,,
\end{multline}
where in the last step have used \eqref{eq3.65} and \eqref{eq3.67}. Observe that
\begin{equation}\label{eq3.70a}
(1-\vartheta\beta)\theta = \left(4\rho^{-1}-1\right)\vartheta\beta\theta +\left(1-4\vartheta\beta\rho^{-1}\right)\theta \,.
\end{equation}
Using \eqref{eq3.68} and \eqref{eq3.70a}, for $\eps \ll\big(4\rho^{-1}-1\big)\vartheta\beta\theta$,
we obtain
$$
2^{-1}\left(4\rho^{-1}-1\right)\vartheta\beta\theta\sigma(Q)\le \sum_{\F_*}\sigma(Q_j)
$$
and thus
\begin{equation}\label{eq3.69}
\sigma(Q) \leq \,C(\vartheta,\rho,\beta,\theta)  \sum_{\F_*}\sigma(Q_j)\,.
\end{equation}


We now make the following claim.
\begin{claim}\label{claim3.70} For $\eps$ chosen sufficiently small,
$$\max\left(\sum_{\F_{good}^{(1)}}\sigma(Q_j)\,,\sum_{\F_{bad}^{(1)}}\sigma(Q_j)\right)\geq \eps\sigma(Q)\,.$$
\end{claim}
\begin{proof}[Proof of Claim \ref{claim3.70}]
If $\sum_{\F_{good}^{(1)}}\sigma(Q_j)\geq \eps\sigma(Q)$, then we are done.
Therefore, suppose that
\begin{equation}\label{eq3.71}
\sum_{\F_{good}^{(1)}}\sigma(Q_j)< \eps\sigma(Q)\,.
\end{equation}
We have made the decomposition 
\begin{equation}\label{eq3.74aa}
\F = \F_{good}^{(1)}\cup \F_{good}^{(2)} \cup \F^{(1)}_{bad} \cup\F_{bad}^{(2)}. 
\end{equation} 
Consequently
$$
\sigma(F_Q) \,\le\, 
  \sum_{\F_{good}^{(2)}}\sigma(F_Q\cap Q_j)
+  \sum_{\F_{bad}}\sigma(F_Q\cap Q_j)
+
O\left(\eps\sigma(Q)\right)
\,, 
$$
where we have used \eqref{eq3.65}, and 
\eqref{eq3.71} to estimate the contributions of $F_0$,  and of
$\F_{good}^{(1)}$, respectively.  This,  \eqref{eq3.56}, \eqref{eq3.67}, and \eqref{eq3.68a} yield
\begin{multline*}
(1-\vartheta\beta)\theta \left( \sum_{\F_{good}^{(2)}}\sigma(Q_j)
+  \sum_{\F_{bad}^{(2)}}\sigma(Q_j) \right)
\leq \, (1-\vartheta\beta)\theta \sigma(Q)\leq\,\sigma(F_Q) 
\\[4pt]
\leq\,
\left(1-4\vartheta\beta\rho^{-1}\right)\theta  \sum_{\F_{good}^{(2)}}\sigma(Q_j)
+ \sum_{\F_{bad}^{(1)}}\sigma(Q_j)
+\theta\sum_{\F_{bad}^{(2)}}\sigma(Q_j) 
+ O\left(\eps\sigma(Q)\right)
\,.
\end{multline*}
In turn, applying  \eqref{eq3.70a} in the latter estimate, and rearranging terms,
we obtain
\begin{equation}\label{eq3.74}
(4\rho^{-1} -1)\vartheta\beta\theta  \sum_{\F_{good}^{(2)}}\sigma(Q_j)
- \vartheta\beta\theta \sum_{\F_{bad}^{(2)}}\sigma(Q_j) 
\leq\,
 \sum_{\F_{bad}^{(1)}}\sigma(Q_j)
+ O\left(\eps\sigma(Q)\right)
\,.
\end{equation}
Recalling that $G_0=\cup_{\F_{good}}Q_j$, and that
$\F_{good}= \F^{(1)}_{good} \cup \F^{(2)}_{good}$,
we further note that by \eqref{eq3.63}, choosing $\eps \ll \rho$, and using
\eqref{eq3.65} and 
\eqref{eq3.71}, we find in particular that
\begin{equation}\label{eq3.75} 
\sum_{\F_{good}^{(2)}}\sigma(Q_j) \,\geq\, \frac{\rho}{2} \sigma(Q). 
\end{equation}
Applying  
\eqref{eq3.75} and the trivial estimate
$ \sum_{\F_{bad}^{(2)}}\sigma(Q_j)\leq\sigma(Q)$ in \eqref{eq3.74}, we then have 
\begin{multline*}
\vartheta\beta\theta\left[1-\frac{\rho}{2}\right]\,\sigma(Q)
=
\left[ \big(4\rho^{-1}-1\big)\vartheta\beta\theta\frac{\rho}{2} -\vartheta\beta\theta \right]\,\sigma(Q)
\\
\leq\, 
\big(4\rho^{-1}-1\big)\vartheta\beta\theta\sum_{\F_{good}^{(2)}}\sigma(Q_j) 
-\vartheta\beta\theta \sum_{\F_{nbad}^{(2)}}\sigma(Q_j) 
\le
\sum_{\F_{bad}^{(1)}}\sigma(Q_j) 
+ O\left(\eps\sigma(Q)\right)\,.
\end{multline*}
Since $\rho<1$, we conclude, 
for $\eps\le (4C)^{-1}\vartheta\beta\theta$, that 
$$
\frac14\vartheta\beta\theta\, \sigma(Q)
\,\le\,
\sum_{\F_{bad}^{(1)}}\sigma(Q_j)\,,
$$
and
Claim \ref{claim3.70}  follows.
\end{proof}

With Claim \ref{claim3.70} in hand, let us return to the proof of \textbf{Case 2b} of Claim \ref{claim3.52}. We begin by noting that by definition of $\F_{bad}^{(1)}$, and Lemma \ref{lemma:VQ}, we can apply 
$H[M_0,\theta]$ to any $Q_j \in \F_{bad}^{(1)}$, hence for each
such $Q_j$ there is a family of Chord-arc domains $\{\Omega^i_{Q_j}\}_i$ satisfying the desired properties.

Now consider $Q_j\in\F^{(1)}_{good}$.  Since $\F^{(1)}_{good}\subset \F_{good}$, 
by pigeon-holing $Q_j$ has a dyadic
child $Q_j'$  satisfying 
\begin{equation}\label{eq3.60*}
\mut\big(\dd(Q'_j)\big) \leq a\sigma(Q'_j)\,,
\end{equation}
(there may be more than one such child, but we just pick one).
Our immediate goal is to find a child $Q_j''$ of $Q_j$, which may or may not equal $Q_j'$,
for which we may construct a family of Chord-arc domains  $\{\Omega^i_{Q''_j}\}_i$ satisfying the desired properties.
To this end, we assume first that $Q_j'$ satisfies 
\begin{equation}\label{eq3.61}
\sigma(F_Q\cap Q_j') \ge  (1-\vartheta)\theta \sigma (Q_j')\,.
\end{equation}
In this case, we set $Q_j'':=Q_j'$, and using Lemma \ref{lemma:VQ},
by the induction hypothesis
$H[a,(1-\vartheta)\theta]$, we obtain the desired family of Chord-arc domains.

We therefore consider the case 
\begin{equation}\label{eq3.80}
\sigma(F_Q\cap Q_j') < (1-\vartheta)\theta \sigma (Q_j')\,.
\end{equation}
In this case, we shall select $Q_j''\neq Q_j'$.
Recall that we use the notation $Q''\lhd Q$ to mean that $Q''$ is a dyadic child of $Q$.
Set $$\F_j'':=\left\{Q_j''\lhd Q_j: \, Q_j''\neq Q_j'\right\}\,.$$
Note that we have the dyadic doubling estimate
\begin{equation}
\sum_{Q_j''\in \F_j''}\sigma(Q_j'') \leq \sigma(Q_j) \leq M_1 \sigma(Q_j')\,,
\label{eq:dd-est}
\end{equation}
where $M_1=M_1(n, ADR)$. 
We also note that
\begin{equation}\label{eq3.81}
\big(1-4\vartheta\beta\rho^{-1}\big) \theta= \big(1-4\beta\rho^{-1}\big)\vartheta\theta + (1-\vartheta)\theta\,.
\end{equation}
By definition of $\F_{good}^{(1)}$, 
$$\big(1-4\vartheta\beta\rho^{-1}\big) \theta \sigma(Q_j)\, \leq \,\sigma(F_Q\cap Q_j) \,=\, \sigma(F_Q\cap Q_j') + 
\sum_{Q_j''\in \F_j''}\sigma(F_Q\cap Q_j'')\,.$$
By \eqref{eq3.80}, it follows that
\begin{multline*}\big(1-4\vartheta\beta\rho^{-1}\big)\theta \sigma(Q_j')+ \big(1-4\vartheta\beta\rho^{-1}\big)\theta \!\!
\sum_{Q_j''\in\F_j''}\sigma(Q_j'')=\big(1-4\vartheta\beta\rho^{-1}\big)\theta \sigma(Q_j)\\[4pt] \leq \, (1-\vartheta)\theta \sigma(Q_j') + 
\sum_{Q_j''\in\F_j''}\sigma(F_Q\cap Q_j'')\,.
\end{multline*}
In turn, using \eqref{eq3.81}, we obtain
$$\big(1-4\beta\rho^{-1}\big)\vartheta\theta\sigma(Q_j') + \big(1-4\vartheta\beta\rho^{-1}\big)\theta \!\!
\sum_{Q_j''\in\F_j''}\sigma(Q_j'') \,\leq\, \sum_{Q_j''\in\F_j''}\sigma(F_Q\cap Q_j'')\,.$$
By the dyadic doubling estimate \eqref{eq:dd-est}, this leads to
$$
\left[\big(1-4\beta\rho^{-1}\big)\vartheta M_1^{-1} + \big(1-4\vartheta\beta\rho^{-1}\big) \right]\theta
\sum_{Q_j''\in\F_j''}\sigma(Q_j'') \,\leq\, \sum_{Q_j''\in\F_j''}\sigma(F_Q\cap Q_j'')\,.
$$
Choosing $\beta \leq \rho/ (4 (M_1+1))$, we find that the
expression in square brackets is at least 1, and therefore, by pigeon holing,
we can pick $Q_j''\in\F_j''$ satisfying
\begin{equation}\label{eq3.82}
\sigma(F_Q\cap Q_j'')\geq \theta \sigma(Q_j'')\,.
\end{equation}  

Hence, using Lemma \ref{lemma:VQ}, we see that
the induction hypothesis 
$H[M_0,\theta]$ holds for $Q''_j\in \F_j''$, and once again we obtain the desired family of Chord-arc domains.

Recall that we have constructed our packing measure $\mut$ in such a way that
each $Q_j\in \F$, as well as all of its children, along with the cubes in $\dd_\F\cap\dd(Q)$,
belong to the same  tree $\sbf$; see Claim \ref{claim3.16}.  This means in particular
that for each such $Q_j$, the Whitney region $U_{Q_j}$ has exactly two components $U_{Q_j}^\pm\subset
\Omega_\sbf^\pm$, and the analogous statement is true for each child of $Q_j$.   This fact has the following 
consequences:

\begin{remark}\label{remark3.66}
For each $Q_j \in \F_{bad}^{(1)}$, and for the selected child $Q_j''$ of each $Q_j\in \F_{good}^{(1)}$,
the conclusion of the induction hypothesis produces at most two 
Chord-arc domains $\Omega^\pm_{Q_j} \supset U_{Q_j}^\pm$ (resp. $\Omega_{Q_j''}^\pm
\supset U_{Q_j''}^\pm$), which we enumerate as $\Omega^i_{Q_j}$ (resp. $\Omega_{Q_j''}^i)$,
$i=1,2$, with $i=1$ corresponding 
 ``+", and  $i=2$ corresponding to ``-", respectively.   
\end{remark}
 
\begin{remark}\label{remark3.67}
For each $Q_j\in \F_{good}^{(1)}$, the connected component $U_{Q_j}^\pm$ overlaps with
the corresponding component $U^\pm_{Q_j''}$ for its child,  so we may augment
$\Omega_{Q_j''}^i$ by adjoining to it the appropriate component $U_{Q_j}^\pm$, to form
a chord arc domain $$\Omega^i_{Q_j} := \Omega_{Q_j''}^i \cup U_{Q_j}^i\,.$$  
\end{remark}


By the induction hypothesis, for each $Q_j \in \F_{bad}^{(1)} \cup \F_{good}^{(1)}$
(and by $n$-ADR, in the case of $\F_{good}^{(1)}$),
the Chord-arc domains $\Omega^i_{Q_j}$ that we have constructed satisfy
$$\sum_i \sigma (Q_j\cap \pom_{Q_j}^i) \gtrsim \sigma(Q_j)\,,$$
where the sum has either one or two terms,  and where 
the implicit constant depends either on $M_0$ and $\theta$, or on $a$ and $(1-\vartheta)\theta$,
depending on which part of the induction hypothesis we have used.
In particular, for each such $Q_j$, there is at least one choice of index $i$  such that
$\Omega^i_{Q_j}=:\Omega_{Q_j}$ satisfies
\begin{equation}\label{eq3.69a}
\sigma (Q_j\cap \pom_{Q_j}) \gtrsim \sigma(Q_j)
\end{equation}
(if the latter is true for both choices $i=1,2$, we arbitrarily choose $i=1$, 
which we recall corresponds to ``+").
Combining the latter bound with Claim \ref{claim3.70}, and recalling that $\eps$ has now 
been fixed depending only on allowable parameters,
we see that
\begin{equation*}
\sum_{Q_j\, \in \, \F_{bad}^{(1)} \,\cup\, \F_{good}^{(1)}}\sigma (Q_j\cap \pom_{Q_j}) \gtrsim \sigma(Q)
\end{equation*}
 For $Q_j \in  \F_{bad}^{(1)} \,\cup\, \F_{good}^{(1)}$, as above set 
$B^*_{Q_j}:= B(x_{Q_j}, K \ell(Q_j))$. 
 By a covering lemma argument, 
we may extract a subfamily 
$\F^*\subset \F_{bad}^{(1)} \,\cup\, \F_{good}^{(1)}$ 
such that $ \{\kappa B^*_{Q_j}\}_{Q_j\in \F^*}$ is 
pairwise disjoint, where again $\kappa\gg K^4$ is a large dilation factor, and such that 
\begin{equation}\label{eq3.69*}
\sum_{Q_j\, \in \, \F^*}\sigma (Q_j\cap \pom_{Q_j}) \gtrsim_\kappa \sigma(Q)
\end{equation}

Let us now build (at most two) Chord-arc domains $\Omega^i_Q$ satisfying the desired properties.
Recall that for each $Q_j \in \F^*$, we defined the corresponding Chord-arc
domain $\Omega_{Q_j}:= \Omega_{Q_j}^i$, where the choice of index $i$ (if there was a choice),
was made so that \eqref{eq3.69a} holds.  We then assign each $Q_j\in \F^*$ either to
$\F^*_+$ or to $\F^*_-$, depending on whether we chose $\Omega_{Q_j}$ satisfying
\eqref{eq3.69a} to be $\Omega_{Q_j}^1=\Omega_{Q_j}^+$, or $\Omega_{Q_j}^2
=\Omega_{Q_j}^-$.  We note that at least one of the sub-collections $\F^*_\pm$ 
is non-empty, since
for each $j$, there was at least one choice of index $i$ such that
\eqref{eq3.69a} holds with $\Omega_{Q_j}:= \Omega_{Q_j}^i$.  Moreover, the 
two collections are disjoint, since we have arbitrarily designated $\Omega_{Q_j} = \Omega_{Q_j}^1$
(corresponding to ``+") in the case that there were two choices for a particular $Q_j$.
We further note that if $Q_j \in \F^*_\pm$, then $\Omega_{Q_j} = \Omega_{Q_j}^\pm\supset
U_{Q_j}^\pm$. 

We are now in position to apply Lemma \ref{lemmaCAD}.  Set
\[
	\sbf^*=\left\{Q'\in \dd(Q): Q_j\subset Q' \text{ for some }Q_j\in\F^*\right\}\,,
	\] 
which is a semi-coherent subtree of $\sbf$, with maximal cube $Q$.	
Without loss of generality, we may suppose that $\F^*_+$ is non-empty, and we then define
$$  \Omega_Q^+:= \Omega_{\sbf^*}^+ \bigcup
 \left(\bigcup_{Q_j \,\in\, \F^*_+}\Omega_{Q_j}\right)\,,$$
 and similarly with ``+" replaced by ``-",  provided that $\F^*_-$ is also non-empty.
 Observe that by the induction hypothesis, 
 and our construction (see Remarks \ref{remark3.66} and \ref{remark3.67}, and Lemma \ref{lemma2.7}),
 for an appropriate choice of ``$\pm$'',
 $U_{Q_j}^\pm\subset \Omega_{Q_j} \subset B_{Q_j}^*$, and since $\ell(Q_j) \leq 2^{-k_1} \ell(Q)$,
by \eqref{eq3.69*} and Lemma  \ref{lemmaCAD}, 
each (non-empty) choice defines a Chord-arc domain with the requisite properties. This completes the proof of \textbf{Case 2b} of Claim \ref{claim3.52} and hence that of Theorem \ref{t1}.
 \end{proof}

\medskip

\vspace{1cm}



\begin{center} 
\Large Part \refstepcounter{parte}\theparte\label{part-2}:  Proof of Theorem \ref{teo1a}
\end{center}
\smallskip

\addcontentsline{toc}{section}{Part 2:  Proof of Theorem \ref{teo1a}}


\section{Preliminaries for the Proof of Theorem \ref{teo1a}}\label{s5}

\subsection{Uniform rectifiability}

Recall the definition of $n$-uniform rectifiable ($n$-UR) sets in Definition \ref{defur}. Given a ball $B\subset \RR^{n+1}$, we denote
\begin{equation}\label{defbbeta}
b\beta_E(B) = \inf_L \frac1{r(B)}\Bigl(\sup_{y\in E\cap B} \dist(y,L) + \sup_{y\in L\cap B}\dist(y,E)\Bigr),
\end{equation}
where the infimum is taken over all the affine $n$-planes that intersect $B$.
The following result is due to David and Semmes:

\begin{theorem}\label{teods}
Let $E\subset\RR^{n+1}$ be $n$-ADR. Denote $\sigma=\HH^n\lfloor_{\,E}$ and let $\DD$ be the associated dyadic lattice. Then, $E$ is $n$-UR if and only if, for any $\eps>0$,
$$\sum_{\substack{Q\in\DD:Q\subset R,\\ b\beta(3\ball_Q)>\eps}} \sigma(Q) \leq C(\eps)\,\sigma(R)\quad \mbox{ for all $R\in\DD$.}$$
\end{theorem}

 For the proof, see
\cite[Theorem 2.4, p.32]{DS2} (this provides a slight variant of Theorem \ref{teods}, and it is straightforward to check that both formulations are equivalent).
Remark that the constant $3$ multiplying $\ball_Q$ in the estimate above can be replaced by any number larger than $1$.

Recall also the following result (see \cite{HLMN} or \cite{MT}).

\begin{theorem}\label{teo*}
Let $\Omega\subset\RR^{n+1}$, $n\ge 1$, be an open set satisfying an interior corkscrew condition, 
with $n$-ADR boundary, such that the harmonic measure in $\Omega$ belongs to
weak-$A_\infty$. Then $\partial\Omega$ is $n$-UR.
\end{theorem}
\vv

\subsection{Harmonic measure}
From now on we assume that $\Omega\subset\RR^{n+1}$ is an open set with $n$-ADR boundary such that the harmonic measure in $\Omega$ belongs to
weak-$A_\infty$. We denote by $\sigma$ the surface measure in $\partial\Omega$, that is, $\sigma = \HH^n\lfloor_{\,\partial\Omega}$.
We also consider the dyadic lattice $\DD$ associated with $\sigma$ as in Lemma \ref{lemmaCh}.
The AD-regularity constant of $\partial\Omega$ is denoted by $C_0$.

We denote by $\omega^p$ the harmonic measure with pole at $p$ of $\Omega$, and by $g(\cdot,\cdot)$ the Green function. Much as before we write $\delta_\Omega(x) = \dist(x,\partial\Omega)$.

The following well known result is sometimes called ``Bourgain's estimate":

\begin{lemma}\cite{B}.\label{l:bourgain}
Let $\Omega\subsetneq \RR^{n+1}$ be open with $n$-ADR boundary,  $x\in \partial\Omega$, and $0<r\leq\diam(\partial\Omega)/2$.  Then 
\begin{equation}\label{e:bourgain}
\omega^y(B(x,2r))\geq c >0, \;\; \mbox{ for all }y\in \Omega\cap \overline B(x,r)
\end{equation}
where $c$ depends on $n$ and the $n$-ADRity constant of $\partial\Omega$.
\end{lemma}

The following is also well known.

\begin{lemma}\label{lem1}
 Let $\Omega\subsetneq \RR^{n+1}$ be open with $n$-ADR boundary. Let $p,q\in\Omega$ be such $|p-q|\geq 4\,\delta_\Omega(q)$.
Then,
$$g(p,q)\leq C\,\frac{\omega^p(B(q,4\delta_\Omega(q)))}{\delta_\Omega(q)^{n-1}}.$$
\end{lemma}

We remark that the previous lemma is also valid in the case $n>1$ without the $n$-ADR assumption. In the case $n=1$ this holds under the $1$-ADR assumption, and also
in the more general situation where $\Omega$ satisfies the CDC. This follows easily from \cite[Lemmas 3.4 and 3.5]{AH}. Notice that $n$-ADR implies the CDC in $\RR^{n+1}$ (for any $n$), by standard arguments.

The following lemma is also known. See \cite[Lemma 3.14]{HLMN}, for example.

\begin{lemma}\label{lem1'}
Let $\Omega\subsetneq \RR^{n+1}$ be open with $n$-ADR boundary and let $p\in\Omega$. Let $B$ be a ball centered at 
$\partial\Omega$ such that $p\not\in 8B$. Then
$$\avint_B g(p,x)\,dx \leq C\,\frac{\omega^p(4B)}{r(B)^{n-1}}.$$
\end{lemma}

\vv

\begin{lemma}\label{lem333}
Let $\Omega\subsetneq\RR^{n+1}$ be open with $n$-ADR boundary. Let $x\in\partial\Omega$ and $0<r<\diam(\Omega)$.
Let $u$ be a non-negative harmonic function in $B(x,4r)\cap \Omega$ and continuous in $B(x,4r)\cap \overline\Omega$
such that $u\equiv 0$ in $\partial\Omega\cap B(x,4r)$. Then extending $u$ by $0$ in $B(x,4r)\setminus \overline\Omega$,
there exists a constant $\alpha>0$ such that, for all $y,z\in B(x,r)$,
$$|u(y)-u(z)|\leq C\,\left(\frac{|y-z|}r\right)^\alpha \!\sup_{B(x,2r)}u
\leq C\,\left(\frac{|y-z|}r\right)^\alpha \;\avint_{B(x,4r)}u,
$$
where $C$ and $\alpha$ depend on $n$ and the AD-regularity of $\partial \Omega$.
In particular,
$$u(y)\leq C\,\left(\frac{\delta_\Omega(y)}r\right)^\alpha \!\sup_{B(x,2r)}u
\leq C\,\left(\frac{\delta_\Omega(y)}r\right)^\alpha \;\avint_{B(x,4r)}u.$$
\end{lemma}

\vv
The next result provides a partial converse to Lemma \ref{lem1}.

\begin{lemma}\label{lem2}
Let $\Omega\subsetneq\RR^{n+1}$ be open with $n$-ADR boundary. 
Let $p\in\Omega$ and let $Q\in\DD$ be such that $p\not\in 2B_Q$. 
Suppose that $\omega^p(Q)\approx \omega^p(2Q)$.
Then there exists some $q\in\Omega$ such that
$$\ell(Q)\lesssim \delta_\Omega(q)\approx \dist(q,Q)\leq 4\diam(Q)$$
and 
$$\frac{\omega^p(2Q)}{\ell(Q)^{n-1}}\leq c\,g(p,q).$$
\end{lemma}

\begin{proof}
For a given $k_0\geq 2$ to be fixed below, we can pick $P\in\DD(Q)$ with $\ell(P)=
2^{-k_0}\ell(Q)$ such that 
$$\omega^p(P)\approx_{k_0}\omega^p(Q).$$
 Let $\vphi_P$ be a $C^\infty$ function supported in $\ball_P$,  $\vphi_P\equiv 1$  on $P$, and such that
 $\|\nabla\vphi_P\|_\infty\lesssim 1/\ell(P)$.
Then, choosing $k_0$ small enough so that $p\not\in 50\ball_P$, say, and applying Caccioppoli's inequality,
\begin{align*}
\omega^p(2Q)&\approx \omega^p(Q)
\approx_{k_0} \omega^p(P) \leq \int \vphi_P\,d\omega^p = -\int\nabla_y g(p,y)\,\nabla\vphi_P(y)\,dy\\
& \lesssim \frac1{\ell(P)}\int_{\ball_P} |\nabla_y g(p,y)|\,dy \lesssim \ell(P)^n
\left(\;\avint_{\ball_P} |\nabla_y g(p,y)|^2\,dy \right)^{1/2}\\
&\lesssim \ell(P)^{n-1}
\left(\;\avint_{2\ball_P} |g(p,y)|^2\,dy \right)^{1/2}\lesssim \ell(P)^{n-1}
\;\avint_{3\ball_P} g(p,y)\,dy.
 \end{align*}
Applying  now Lemmas \ref{lem333} and \ref{lem1'} and taking $k_0$ small enough so that $24\ball_P\cap\partial\Omega\subset2Q$,  for any $a\in (0,1)$ we get
$$\;\avint_{y\in3\ball_P:\delta_\Omega(y)\leq a\ell(P)} g(p,y)\,dy \lesssim a^\alpha \;\avint_{6\ball_P} g(p,y)\,dy
\lesssim a^\alpha\,\frac{\omega^p(24\ball_P)}{\ell(P)^{n-1}}\lesssim a^\alpha\,\frac{\omega^p(2Q)}{\ell(P)^{n-1}}.
$$
From the estimates above we infer that
$$\omega^p(2Q)\lesssim_{k_0}	\ell(P)^{n-1}\;
\avint_{y\in3\ball_P:\delta_\Omega(y)\geq a\ell(P)} g(p,y)\,dy + a^\alpha\,\omega^p(2Q).$$
Hence, for $a$ small enough, we derive
$$\omega^p(2Q)\lesssim_{k_0}	\ell(P)^{n-1}\;
\avint_{y\in3\ball_P:\delta_\Omega(y)\geq a\ell(P)} g(p,y)\,dy,$$
which implies the existence of the point $q$ required in the lemma.
\end{proof}

\vv

\subsection{Harnack chains and carrots}

It will be more convenient for us to work with Harnack chains instead of curves. The existence of a carrot curve is equivalent to having what we call a good chain between points.

Let $x\in \Omega$, $y\in\overline\Omega$ be such that $\delta_\Omega(y)\leq \delta_\Omega(x)$, and let $C>1$.
A {\it $C$-good chain} (or $C$-good Harnack chain) from $x$ to $y$ is a sequence of balls $B_{1},B_{2},\dots$ (finite or infinite)  contained in $\Omega$ such that 
$x\in B_1$ and either
\begin{itemize}
\item $\lim_{j\to\infty} \dist(y,B_j)=0$ if $y\in \partial\Omega$, or
\item $y\in B_N$ if $y\in \Omega$, where $N$ is the number of elements of the sequence if this is finite,
\end{itemize}
and moreover the following hold:
\begin{itemize}
\item $B_j\cap B_{j+1}\neq \varnothing$ for all $j$, 
\item $C^{-1}\,\dist(B_j,\partial\Omega) \leq r(B_j)\leq C\,\dist(B_j,\partial\Omega)$ for all $j$,
\item $r(B_j)\leq C\,r(B_i)$ if $j>i$,
\item for each $t>0$ there are at most $C$ balls $B_j$ such that $t< r(B_j)\leq2t$.
\end{itemize}
Abusing language, sometimes we will omit the constant $C$ and we will just say ``good chain'' or ``good Harnack chain''.

\vv
Observe that in the definitions of carrot curves and good chains, the order of $x$ and $y$ is important: having a carrot curve from $x$ to $y$ is not equivalent to having one from $y$ to $x$, and similarly with good chains. 

\begin{lemma}\label{l:johntochain}
There is a carrot curve from $x\in \Omega$ to $y\in \overline{\Omega}$ if and only if there is a good Harnack chain from $x$ to $y$. 
\end{lemma}


\begin{proof}
Let $\gamma$ be a carrot curve from $x$ to $y$.  We can assume $y\in \Omega$, since if $y\in \partial\Omega$, we can obtain this case by taking a limit of points $y_{j}\in \Omega$ converging to $y$. Let $\{B_{j}\}_{j=1}^{N}$ be a Vitali subcovering  of 
the family 
$\{B(z,\delta_{\Omega}(z)/10): z\in \gamma\}$ and let $r_{B_j}$ stand for the radius  and $x_{B_j}$ for the center of $B_j$. So the balls $B_j$ are disjoint and $5B_{j}$ 
cover $\gamma$. Note that for $t>0$, if $t< r_{B_{j}}\leq 2t$, 
\[
|x_{B_{j}}-y|
\leq \HH^1(\gamma(x_{B_{j}},y))
\lesssim \delta_\Omega (x_{B_{j}})
\approx r_{B_{j}}\leq 2t.
\]
In particular, since the $B_{j}$'s are disjoint, by volume considerations, there can only be boundedly many $B_{j}$ of radius between $t/2$ and $t$, say. Moreover, 
 we may order the balls $B_j$ so that $x\in 5B_1$ and 
$B_{j+1}$ is a ball $B_k$ such that $5B_k\cap 5B_j\neq\varnothing$
and $5\overline{B_k}$ contains the point from $\gamma\cap \bigcup_{h:5B_h\cap 5B_j\neq\varnothing}5\overline B_h$ which is maximal
in the natural order induced by $\gamma$ (so that $x$ is the minimal point in $\gamma$).
Then for $j>i$,
$$r_{B_{j}} \approx \delta_\Omega(x_{B_j}) \leq   |x_{B_j}- x_{B_i}|+ \delta_\Omega(x_{B_i})
\leq 
\HH^1(\gamma(x_{B_{i}},y)) + \delta_\Omega(x_{B_i}) 
\lesssim r_{B_{i}}.
$$
This implies $5B_1,5B_2,\ldots$ is a $C$-good chain for a sufficiently big $C$.

Now suppose that we can find a good chain from $x$ to $y$, call it $B_{1},\dots,B_{N}$. Let $\gamma$ be the path obtained by connecting their centers in order. Let $z\in \gamma$. Then there is a $j$ such that $z\in [x_{B_{j}},x_{B_{j+1}}]$, the segment joining $x_{B_{j}}$ and $x_{B_{j+1}}$. Since $\{B_{i}\}_i$ is a good chain,\[
\HH^1(\gamma(z,y))
\leq |z-x_{B_{j+1}}| + \HH^1(\gamma(x_{B_{j+1}},y))
\leq r_{B_{j+1}}+\sum_{i=j}^{N} 2r_{B_{i}}
\lesssim r_{B_{j}}\approx \delta_\Omega(z).
\]
We would like to note that the implicit constants do not depend on $N$. Indeed, from the properties of the good chain it easily follows that
\[
\sum_{i=j}^{N} r_{B_{i}}
\le
\sum_{k=0}^\infty \sum_{i: 2^{-k-1}< \frac{r_{B_{i}}}{Cr_{B_{j}}}\le  2^{-k}} r_{B_{i}}
\le 2\,C^2\,r_{B_{j}}.
\]
Thus, $\gamma$ is a carrot curve from $x$ to $y$. 
\end{proof}

\vv



\section{The Main Lemma for the proof of Theorem \ref{teo1a}}\label{s6}

Because of the absence of doubling conditions on harmonic measure under the weak-$A_\infty$ assumption, to prove Theorem \ref{teo1a} we cannot use arguments similar to the ones in \cite{AH} or \cite{Az}. 
Instead,  we  prove a local result which involves 
only one pole and one ball which has its own interest. This is the Main Lemma \ref{lemhc} below. 
 
Let $B\subset\RR^{n+1}$ be a ball centered at $\partial\Omega$ and let $p
\in\Omega$.
 We restate Definition \ref{deflocalAinfty} in the following form: $\omega^p$ satisfies the {\it weak-$A_\infty$ condition in $B$} if for every 
$\eps_0 \in (0,1)$
 there exists $\delta_0 \in (0,1)$ such that the
 following holds: for any subset $E \subset B\cap\partial\Omega$,
 \begin{equation}\label{fwFRQAFRQAF}
\sigma(E)\leq \delta_0\,\sigma(B\cap\partial\Omega) \quad \Longrightarrow\quad \omega^{p}(E)\leq \eps_0\,\omega^{p}(2B).
\end{equation}

In the next sections we will prove the following.

\begin{mlemma}\label{lemhc} 
Let $\Omega\subset\RR^{n+1}$ have $n$-uniformly rectifiable boundary. 
Let $R_0\in\DD$ and let $p\in\Omega\setminus 4\ball_{R_0}$ be a point such that 
$$c\,\ell(R_0)\leq \dist(p,\partial\Omega)\leq \dist(p,R_0)\leq c^{-1}\,\ell(R_0)$$
and $\omega^{p}(R_0)\geq c'>0$. Suppose that
$\omega^p$ satisfies the weak-$A_\infty$ condition in $\ball_{R_0}$.
Then there exists a subset $\Con(R_0)\subset R_0$ and a constant $c''>0$ with
$\sigma(\Con(R_0))\geq c''\,\sigma(R_0)$ such that each point $x\in \Con(R_0)$ can be joined to $p$ by a carrot curve.
The constant $c''$ and the constants involved in the carrot condition only depend on $c,c',n$, the weak-$A_\infty$ condition, and the $n$-UR character of $\partial\Omega$. 
\end{mlemma}

The notation $\Con(\cdot)$ stands for ``connectable''.

It is easy to check that Theorem \ref{teo1a} follows from this result. First notice that the assumptions of the theorem imply that $\partial\Omega$
is $n$-uniformly rectifiable by Theorem \ref{teo*}.
Consider now any $x\in\Omega$ and take a point $\xi\in\partial\Omega$ such that $|x-\xi|=\delta_\Omega(x)$. Then we consider the point $p$ in the segment  $[x,\xi]$ such 
that $|p-\xi|=\frac1{16}\,\delta_\Omega(x)$. By  Lemma \ref{l:bourgain}, we have
$$
\omega^p(B(\xi,\tfrac18\delta_\Omega(x)))\gtrsim 1,
$$
because $p\in \overline{\frac12B(\xi,\tfrac18\delta_\Omega(x))}$. Hence, by covering $B(\xi,\tfrac18\delta_\Omega(x))\cap\partial\Omega$ with cubes $R\in\DD$ contained in $B(\xi,\tfrac14\delta_\Omega(x))\cap\partial\Omega$ with side length comparable to $\delta_\Omega(x)$
we deduce that at least one these cubes, call it $R_0$, satisfies $\omega^p(R_0)\gtrsim 1$. Further, by taking the side length small enough,
we may also assume that $p\not\in 4\ball_{R_0}$. Since  $\omega^p$ satisfies the weak-$A_\infty$ property in 
$B_{R_0}$ (by the assumptions in Theorem \ref{teo1a}), we can apply the Main Lemma \ref{lemhc} above and infer that there exists a subset $F:=\Con(R_0)\subset R_0$ with
$\sigma(F)\geq c'\,\sigma(R_0)\gtrsim \delta_\Omega(x)^n$ such that all $y\in F$ can be joined to $x$ by
a carrot curve, which proves that $\Omega$
satisfies the weak local John condition and concludes the proof of Theorem \ref{teo1a}.
\vv

Two essential ingredients of the proof of the Main Lemma \ref{lemhc} 
are a corona type decomposition (whose existence is ensured
by the $n$-uniform rectifiability of the boundary) and the Alt-Caffarelli-Friedman monotonicity formula
\cite{ACF}. This formula is used in some of the connectivity arguments below. This allows to connect by carrot curves corkscrew points where the Green function is not too small to other corkscrew points at a larger distance from the boundary
where the Green function is still not too small (see Lemma \ref{lemshortjumps} for the precise statement).
The use of the Alt-Caffarelli-Friedman formula is not new to problems involving harmonic
measure and connectivity (see, for example, \cite{AGMT}). However, the way it is applied here 
is new, as far as we know.

Two important steps of the proof of the Main Lemma \ref{lemhc} (and so of Theorem \ref{teo1a}) are 
the Geometric Lemma \ref{lemgeom} and the Key Lemma \ref{keylemma}. 
An essential idea consists of distinguishing cubes with ``two well separated
big corkscrews'' (see Section \ref{subsepcork} for the precise definition).
In the Geometric Lemma \ref{lemhc}
we construct two disjoint open sets satisfying a John condition associated to trees involving this type of cubes, so that the boundaries of
the open sets are located in places where the Green function is very small. This construction
is only possible because the associated tree involves  only cubes with two well separated
big corkscrews. 
The existence of these cubes is an obstacle for the construction of carrot curves. However, in a sense,
in the Key Lemma \ref{keylemma} we take advantage of their existence to obtain some delicate estimates
for the Green function on some corkscrew points.

\vv

We claim now that to prove he Main Lemma \ref{lemhc}  {\bf we can assume that $\Omega=\RR^{n+1}\setminus \partial\Omega$}. To check this, let $\Omega$, $p$, and $R_0$ satisfy the assumptions in the Main Lemma.
Consider the open set $V=\re^{n+1}\setminus \partial\Omega$. Then the harmonic measure $\omega^p$ in $\Omega$ coincides with the harmonic measure $\omega_V^p$ in $V$ (the fact that $V$ is not connected does
not disturb us). Thus $V$, $p$, and $R_0$ satisfy the assumptions in the Main Lemma, and moreover
$V=\re^{n+1}\setminus \partial\Omega =\RR^{n+1}\setminus \partial V$. Assuming the Main Lemma
to be valid in this particular case, 
we deduce that 
 there exists a subset $\Con(R_0)\subset R_0$ and a constant $c''>0$ with
$\sigma(\Con(R_0))\geq c''\,\sigma(R_0)$ such that each point $x\in \Con(R_0)$ can be joined to $p$ by a carrot curve in $V$. Now just observe that if $\gamma$ is one of this carrot curves and it joints $p$ and  
$x\in \Con(R_0)\subset \partial V=\partial\Omega$, then $\gamma$ is contained in $V$ except for its end-point
$x$. By connectivity, since $p \in \Omega\cap\gamma$, $\gamma$ must be contained in $\Omega$, except for the end-point $x$. Hence, $\gamma$ is a carrot curve with respect to $\Omega$.

Sections \ref{secshortp}-\ref{sec8} are devoted to the proof of Main Lemma \ref{lemhc}. To this end, we will
assume that $\Omega=\RR^{n+1}\setminus \partial\Omega$.

\vv


\section{The Alt-Caffarelli-Friedman formula and the existence of short paths}\label{secshortp}

\subsection{The Alt-Caffarelli-Friedman formula}\label{s7}

Recall the following well known result of Alt-Caffarelli-Friedman (see \cite[Theorems 12.1 and 12.3]{CS}):

\begin{theorem} \label{t:ACF} Let $B(x,R)\subset \RR^{n+1}$, 
	and let $u_1,u_2\in
W^{1,2}(B(x,R))\cap C(B(x,R))$ be nonnegative subharmonic functions. Suppose that $u_1(x)=u_2(x)=0$ and
that $u_1\cdot u_2\equiv 0$. Set 
\begin{equation*}
J_i(x,r) = \frac{1}{r^{2}} \int_{B(x,r)} \frac{|\nabla u_i(y)|^{2}}{|y-x|^{n-1}}dy,
\end{equation*}
and
\begin{equation}\label{e:Jxr}
J(x,r) = J_1(x,r)\,J_2(x,r).
\end{equation}
Then $J(x,r)$ is a non-decreasing function of $r\in (0,R)$ and $J(x,r)<\infty$ for all $r\in (0,R)$. That is,
\begin{equation}\label{e:gamma}
J(x,r_{1})\leq J(x,r_{2})<\infty \;\;  \mbox{ for } \;\; 0<r_{1}\leq r_{2}<R.
\end{equation}
Further, 
\begin{equation}\label{eqjr}
J_i(x,r)\lesssim \frac1{r^2}\,\|u_i\|_{\infty,B(x,2r)}^2.
\end{equation}
\end{theorem}

In the case of equality we have the following result (see \cite[Theorem 2.9]{PSU}).

\begin{theorem}\label{teoACFequ}
Let $B(x,R)$ and $u_1,u_2$ be as in Theorem \ref{t:ACF}. Suppose that $J(x,r_a)=J(x,r_b)$ for some $0<r_a<r_b<R$.
Then either one or the other of the following holds:
\begin{itemize}
\item[(a)] $u_1=0$ in $B(x,r_b)$ or $u_2=0$ in $B(x,r_b)$;

\item[(b)] there exists a unit vector $e$ and constants $k_1,k_2>0$ such that
$$u_1(y)= k_1\,((y-x)\cdot e)^+,\qquad u_2(y)= k_2\,((y-x)\cdot e)^-, \qquad\text{in $B(x,r_b)$}.$$
\end{itemize}
\end{theorem}

 We will also need the following auxiliary lemma.
\vv
\begin{lemma} \label{lem:ACF1} Let $B(x,R)\subset \RR^{n+1}$, and let $\{u_i\}_{i\geq 1}
\subset W^{1,2}(B(x,R))\cap C(B(x,R))$ a sequence of
functions which are nonnegative, subharmonic, such that each $u_i$ is
harmonic in $\{y\in B(x,R):u_i(y)>0\}$ and $u_i(x)=0$. Suppose also that
$$\|u_i\|_{\infty,B(x,R)}\leq C_1\,R\quad\text{ and }\quad \|u_i\|_{{\rm Lip}^\alpha,B(x,R)}\leq C_1\,R^{1-\alpha}$$
for all $i\geq 1$.
Then, for every $0<r<R$  there exists a subsequence $\{u_{i_k}\}_{k\geq1}$ which converges uniformly in
$B(x,r)$ and weakly in $W^{1,2}(B(x,r))$ to some function $u\in W^{1,2}(B(x,r))\cap C(B(x,r))$, and moreover,
\begin{equation}\label{eqlim4}
\lim_{k\to\infty}  \int_{B(x,r)} \frac{|\nabla u_{i_k}(y)|^{2}}{|y-x|^{n-1}}dy =
 \int_{B(x,r)} \frac{|\nabla u(y)|^{2}}{|y-x|^{n-1}}dy.
 \end{equation}
\end{lemma}

\begin{proof}
The existence of a subsequence $\{u_{i_k}\}_{k\geq1}$ converging weakly in $W^{1,2}(B(x,r))$ and uniformly in
$B(x,r)$ to some function $u\in W^{1,2}(B(x,r))\cap C(B(x,r))$
is an immediate consequence of the Arzel\`a-Ascoli and the Banach-Alaoglu theorems.
The identity \rf{eqlim4} is clear when $n=1$, and 
quite likely, for $n>1$ this is also well known. However, for completeness, we will show the details (for $n>1$).

Consider a non-negative subharmonic function $v\in W^{1,2}(B(x,R))\cap C(B(x,R))$ which is harmonic in $\{y\in B(x,R):v(y)>0\}$ so that $v(x)=0$. 
For $0<r<R$ and $0<\delta<R-r$, let $\vphi$ be a radial $C^\infty$ function such that $\chi_{B(x,r)} \leq  \vphi\leq
\chi_{B(x,r+\delta)}$. Let $\EE(y)=c_n^{-1}\,|y|^{1-n}$ be the fundamental solution of the Laplacian.
For $\eps>0$, denote $v_\eps=\max(v,\eps)-\eps$.
Then we have
\begin{align*}
\int \frac{|\nabla v_\eps(y)|^{2}}{|y-x|^{n-1}}\,\vphi(y)\,dy &= c_n\int \nabla v_\eps(y)\,\nabla(\EE(x-\cdot)\,v_\eps\,\vphi)(y) \,dy\\
&\quad - c_n\int \nabla v_\eps(y)\,\EE(x-y)\,v_\eps(y)\,\nabla\vphi(y) \,dy\\
& \quad - c_n\int \nabla v_\eps(y)\,\nabla_y \EE(x-y)\,v_\eps(y)\,\vphi(y) \,dy = c_n( I_1 - I_2 - I_3).
\end{align*}
Using the fact that $v_\eps$ is harmonic in $\{v_\eps>0\}$ and that $\EE(x-\cdot)\,v_\eps\,\vphi \in W^{1,2}_0(\{v_\eps>0\} \cap B(x,R))$ since  $\vphi$ is compactly supported in $B(x,R)$, $v_\eps=0$ on $\partial \{v_\eps>0\}$, and $x$ is far away from $\overline{\{v_\eps>0\}}$,
it follows easily that $I_1=0$. On the other hand, we have
\begin{align*}
2\,I_3 & = \int \nabla(v_\eps^2\,\vphi)(y)\,\nabla_y \EE(x-y)\,dy -
\int v_\eps(y)^2\,\nabla_y \EE(x-y)\,\nabla\vphi(y) \,dy \\
& = - v_\eps(x)^2-
\int v_\eps(y)^2\,\nabla_y \EE(x-y)\,\nabla\vphi(y) \,dy.
\end{align*}
Thus,
\begin{align*}
\int \frac{|\nabla v_\eps(y)|^{2}}{|y-x|^{n-1}}\,\vphi(y)\,dy & = - c_n\int \nabla v_\eps(y)\,\EE(x-y)\,v_\eps(y)\,\nabla\vphi(y) \,dy\\
&\quad + \frac{c_n}2\int v_\eps(y)^2\,\nabla_y \EE(x-y)\,\nabla\vphi(y) \,dy.
\end{align*}
Taking into account that $\supp \nabla\vphi$ is far away from $x$, letting $\eps\to0$, we obtain
\begin{align*}
\int \frac{|\nabla v(y)|^{2}}{|y-x|^{n-1}}\,\vphi(y)\,dy & = - c_n\int \nabla v(y)\,\EE(x-y)\,v(y)\,\nabla\vphi(y) \,dy\\
&\quad + \frac{c_n}2\int v(y)^2\,\nabla_y \EE(x-y)\,\nabla\vphi(y) \,dy.
\end{align*}

Using the preceding identity, it follows easily that
$$\lim_{k\to\infty}\int \frac{|\nabla u_{i_k}(y)|^{2}}{|y-x|^{n-1}}\,\vphi(y)\,dy = \int \frac{|\nabla u(y)|^{2}}{|y-x|^{n-1}}\,\vphi(y)\,dy.$$
Indeed, $\lim_{k\to \infty}u_{i_k}(x)^2 = u(x)^2$. Also, it is clear that 
$$\lim_{k\to \infty} \int u_{i_k}(y)^2\,\nabla_y \EE(x-y)\,\nabla\vphi(y) \,dy = \int u(y)^2\,\nabla_y \EE(x-y)\,\nabla\vphi(y) \,dy.$$
Further,
\begin{align*}
\int \nabla u_{i_k}(y)\,\EE(x-y)\,u_{i_k}(y)\,\nabla\vphi(y) \,dy & = 
\int \nabla u_{i_k}(y)\,\EE(x-y)\,u(y)\,\nabla\vphi(y) \,dy \\
&\quad + \int \nabla u_{i_k}(y)\,\EE(x-y)\,(u_{i_k}(y)-u(y))\,\nabla\vphi(y) \,dy\\
&\stackrel{k\to\infty}\to \int \nabla u(y)\,\EE(x-y)\,u(y)\,\nabla\vphi(y) \,dy,
\end{align*}
by the weak convergence of $u_{i_k}$ in $W^{1,2}(B(x,R))$ and the uniform convergence in $B(x,r+\delta)$,
since $\supp\nabla\vphi$ is far away from $x$.

Let $\psi$ be a radial $C^\infty$ function such that $\chi_{B(x,r-\delta)} \leq  \psi\leq
\chi_{B(x,r)}$. The same argument as above shows that
$$\lim_{k\to\infty}\int \frac{|\nabla u_{i_k}(y)|^{2}}{|y-x|^{n-1}}\,\psi(y)\,dy = \int \frac{|\nabla u(y)|^{2}}{|y-x|^{n-1}}\,\psi(y)\,dy.$$
Consequently,
$$\limsup_{k\to\infty}   \int_{B(x,r)} \frac{|\nabla u_{i_k}(y)|^{2}}{|y-x|^{n-1}}dy \leq 
\lim_{k\to\infty}\int \frac{|\nabla u_{i_k}(y)|^{2}}{|y-x|^{n-1}}\,\vphi(y)\,dy = \int \frac{|\nabla u(y)|^{2}}{|y-x|^{n-1}}\,\vphi(y)\,dy,$$
and also
$$\liminf_{k\to\infty}   \int_{B(x,r)} \frac{|\nabla u_{i_k}(y)|^{2}}{|y-x|^{n-1}}dy \geq 
\lim_{k\to\infty}\int \frac{|\nabla u_{i_k}(y)|^{2}}{|y-x|^{n-1}}\,\psi(y)\,dy = \int \frac{|\nabla u(y)|^{2}}{|y-x|^{n-1}}\,\psi(y)\,dy.$$
Since $\delta>0$ can be taken arbitrarily small, \rf{eqlim4} follows.
\end{proof}

\vv

\begin{lemma} \label{lem:ACF2} Let $B(x,2R)\subset \RR^{n+1}$, and let $u_1,u_2\in
W^{1,2}(B(x,2R))\cap C(B(x,2R))$ be nonnegative subharmonic functions such that each $u_i$ is harmonic
in $\{y\in B(x,2R):u_i(y)>0\}$.
 Suppose that $u_1(x)=u_2(x)=0$ and
that $u_1\cdot u_2\equiv 0$. 
Assume also that
$$\|u_i\|_{\infty,B(x,2R)}\leq C_1\,R\quad\text{ and }\quad \|u_i\|_{{\rm Lip}^\alpha,B(x,2R)}\leq C_1\,R^{1-\alpha}\quad \mbox{ for $i=1,2$.}$$
For any $\eps>0$, there exists some $\delta>0$ such that if
$$J(x,R)\leq (1+\delta)\, J(x,\tfrac12R),$$
with $J(\cdot,\cdot)$ defined in \rf{e:Jxr},
then either
one or the other of the following holds:
\begin{itemize}
\item[(a)] $\|u_1\|_{\infty,B(x,R)}\leq \eps\,R$ or $\|u_2\|_{\infty,B(x,R)}\leq \eps\,R$;

\item[(b)] there exists a unit vector $e$ and constants $k_1,k_2>0$ such that
$$\|u_1- k_1\,((\cdot-x)\cdot e)^+\|_{\infty,B(x,R)}\leq \eps\,R
,\qquad \|u_2-k_2\,((\cdot-x)\cdot e)^-\|_{\infty,B(x,R)}\leq \eps\,R 
.$$
\end{itemize}
The constant $\delta$ depends only on $n,\alpha,C_1,\eps$.
\end{lemma}

\begin{proof}
Suppose that the conclusion of the lemma fails. Then, by replacing $u_i(y)$ by $\frac1{R}\,u_i(Ry+x)$,  we  can assume that $x=0$ and $R=1$.
Let $\eps>0$, and for each $\delta=1/k$ and $i=1,2$, consider functions $u_{i,k}$ satisfying the
assumptions of the lemma and such that neither (a) nor (b) holds for them.
By Lemma \ref{lem:ACF1}, there exist subsequences (which we still denote by $\{u_{i,k}\}_k$) 
which converge uniformly in
$B(0,\frac32)$ and weakly in $W^{1,2}(B(0,\frac32))$ to some functions $u_i\in W^{1,2}(B(0,\frac32))\cap C(B(0,\frac32))$, and moreover,
$$\lim_{k\to\infty}  \int_{B(0,r)} \frac{|\nabla u_{i,k}(y)|^{2}}{|y|^{n-1}}dy =
 \int_{B(0,r)} \frac{|\nabla u_i(y)|^{2}}{|y|^{n-1}}dy$$
both for $r=1$ and $r=1/2$. Clearly, the functions $u_i$ are non-negative, subharmonic, and $u_1\cdot u_2=0$. Hence, by Theorem \ref{teoACFequ}, one of the following holds:
\begin{itemize}
\item[(a')] $u_1=0$ in $B(0,1)$ or $u_2=0$ in $B(0,1)$;

\item[(b')] there exists a unit vector $e$ and constants $k_1,k_2>0$ such that
$$u_1(y)= k_1\,(y\cdot e)^+,\qquad u_2(y)= k_2\,(y\cdot e)^-, \qquad\text{in $B(0,1)$}.$$
\end{itemize}
However, the fact that neither (a) nor (b) holds for any pair $u_{1,k}$, $u_{2,k}$, together with the uniform
convergence of $\{u_{i,k}\}_k$, implies that neither (a') nor (b') can hold, and thus we get a contradiction.
\end{proof}

\subsection{Existence of short paths}\label{sect:short-paths}

Let $p\in\Omega$ and $\Lambda>1$.
For $x\in\partial\Omega$, we write $x\in \WA(p,\Lambda)$ if 
for all $0<r\leq \delta_\Omega(p)$,
$$\Lambda^{-1}\,\frac{\sigma(B(x,r))}{\sigma(B(x,\delta_\Omega(p)))}\leq
\omega^p(B(x,r))\leq \Lambda\,\frac{\sigma(B(x,r))}{\sigma(B(x,\delta_\Omega(p)))}.$$
We will see in Section \ref{sechdld} that, under the assumptions of the Main Lemma \ref{lemhc},
 for some $\Lambda$ big enough,
\begin{equation}\label{eqwap}
\sigma(\WA(p,\Lambda)\cap R_0)\gtrsim \sigma(R_0).
\end{equation}

\vv

\begin{lemma}\label{lemfacc}
Let $p\in\Omega$, $x_0\in\WA(p,\Lambda)$, and $r\in (0,\delta_\Omega(p))$. Then there exists
$q\in B(x_0,r)$ such that, for some constant $\kappa\in (0,1/10)$,
\begin{itemize}
\item[(a)] $\delta_\Omega(q) \geq \kappa \,r$, and
\item [(b)]
$$\kappa\,\frac{\omega^p(B(x_0,r))}{r^{n-1}} \leq  g(p,q)\leq \kappa^{-1}\,\frac{\omega^p(B(x_0,r))}{r^{n-1}}
.$$
\end{itemize}
The constant $\kappa$ depends only on $\Lambda$, $n$, and $C_0$, the AD-regularity constant of $\partial\Omega$.
\end{lemma}

\begin{proof}
This follows easily from Lemmas \ref{lem1} and \ref{lem2}.
\end{proof}

\vv

\begin{lemma}[Short paths]\label{lemshortjumps}
Let $p\in\Omega$, $x_0\in\WA(p,\Lambda)$,  and for $0<r_0\leq \delta_\Omega(p)/4$, $0<\tau_0,\lambda_0 \leq1$, let $q\in\Omega$ be such that
\begin{equation}\label{eqass11}
q\in B(x_0,r_0),\quad \delta_\Omega(q)\geq \tau_0\,r_0,\quad g(p,q)\geq \lambda_0\,\frac{\delta_\Omega(q)}{\delta_\Omega(p)^n}.
\end{equation}
Then there exist constants $A_1>1$ and $0<a_1,\lambda_1<1$ such that for every $r\in (r_0,\delta_\Omega(p)/2)$,  there exists some point $q'\in\Omega$ such that 
\begin{equation}\label{eqconc11}
q'\in B(x_0,A_1r),\quad \delta_\Omega(q')\geq \kappa\,|x_0-q'|\geq \kappa\,r,\quad g(p,q')\geq \lambda_1\,\frac{\delta_\Omega(q')}{\delta_\Omega(p)^n},
\end{equation}
(with $\kappa$ as in Lemma \ref{lemfacc})
and such that
$q$ and $q'$ can be joined by a curve $\gamma$ such that
$$\gamma\subset\{y\in B(x_0,A_1r):\dist(y,\partial\Omega)
>a_1\, r_0\}.$$
The parameters $\lambda_1,A_1,a_1$ depend only 
on $C_0,\Lambda,\lambda_0,\tau_0$ and the ratio $r/r_0$.
\end{lemma}

\begin{proof}
All the parameters in the lemma will be fixed along the proof. We assume that $A_1\gg \kappa^{-1}>1$.
First note that we may assume that $r< 2A_1^{-1} |x_0-p|$. Otherwise, we just take a point
$q'\in\Omega$ such that $|p-q'|=\delta_\Omega(p)/2$, which clearly satisfies the properties in \rf{eqconc11}.
Further, both $q$ and $q'$ belong to the open connected set
$$U:=\{x\in\Omega:g(p,x)>c_2\,r_0\,\delta_\Omega(p)^{-n}\}$$
for a sufficiently small $c_2>0$.  The fact that $U$ is connected is well known. This follows from the fact that, for any $\lambda>0$, any connected component of $\{g(p,\cdot)>\lambda\}$ should contain $p$. Otherwise there would be a connected component where $g(p,\cdot)-\lambda$ is positive and harmonic with zero boundary values. So, by maximum principle,  $g(p,\cdot)-\lambda$ should equal $\lambda$ in the whole component, which is a contradiction.  So there is only one connected component.

We just let $\gamma$ be a curve contained in $U$. Note that
$$\dist(U,\partial\Omega)\geq c\,r_0^{\frac1\alpha}\,\delta_\Omega(p)^{1-\frac1\alpha}\geq a\,r_0,$$
for a sufficiently small $a>0$
because, by boundary H\"older continuity,
$$g(p,x)\lesssim\left(\frac{\delta_\Omega(x)}{\delta_\Omega(p)}\right)^\alpha\,\frac1{\delta_\Omega(p)^{n-1}}$$
if $\dist(x,\partial\Omega)\leq \delta_\Omega(p)/2$. Further, the fact that 
$g(p,x)\leq c|x-p|^{1-n}$ ensures that $U\subset B(p,C\delta_\Omega(p))$, for a sufficiently big constant $C$ depending on $r/r_0$.


\vv
So from now on we assume that $r< 2A_1^{-1} |x_0-p|$. 
By Lemma \ref{lemfacc} we know there exists some point $\wt q\in\Omega$ such that
\begin{equation}\label{eqass12}
\wt q\in B(x_0,\kappa^{-1}r),\quad \delta_\Omega(\wt q)\geq r\geq \kappa\,|x_0-\wt q|\geq \kappa \,\delta_\Omega(\wt q)\geq \kappa\, r
,\quad g(p,\wt q)\geq c\,\frac{\delta_\Omega(\wt q)}{\delta_\Omega(p)^n},
\end{equation}
with $c$ depending on $\kappa$ and $\Lambda$.

Assume that $q$ and $\wt q$ cannot be joined by a curve $\gamma$ as in the statement  of the lemma.  Otherwise, we choose $q'=\wt q$ and we are done. 
For $t>0$, consider the open set
$$V^t =\bigl\{x\in B(x_0, \tfrac14A_1r):g(p,x)>t \,r_0\,\delta_\Omega(p)^{-n}\bigr\}.$$
We fix $t>0$ small enough such that $q,\wt q\in V^{2t}\subset V^t$. Such $t$ exists by \rf{eqass11} and \rf{eqass12}, and it may depend on $\Lambda,\lambda,r/r_0$.

Let $V_1$ and $V_2$ be the respective components of $V^t$ to which $q$ and $\wt q$ belong.
We have
$$V_1\cap V_2=\varnothing,$$
because otherwise there is a curve contained in $V^t\subset B(x_0,\tfrac14A_1r)$ which connects $q$ and $\wt q$, and further 
this is far away from $\partial\Omega$. Indeed, we claim that
\begin{equation}\label{eqvt5}
\dist(V^t,\partial \Omega)\gtrsim_{A_1,\Lambda,t,r/r_0} r_0.
\end{equation}
To see this, note
that by the H\"older continuity of $g(p,\cdot)$ in $B(x_0,\tfrac12A_1r)$, for all $x\in V^t$, we have
\begin{align*}
t \,\frac{r_0}{\delta_\Omega(p)^{n}}\leq 
g(p,x) &\lesssim \sup_{y\in B(x_0,\frac12 A_1r)} g(p,y)\,\left(\frac{\delta_\Omega(x)}{A_1r}\right)^\alpha \\
& \leq 
\;\avint_{B(x_0,\frac34 A_1r)} g(p,y)\,dy \left(\frac{\delta_\Omega(x)}{A_1r}\right)^\alpha
\\
& \lesssim_{A_1,\Lambda} \frac{A_1r}
{\delta_\Omega(p)^n}\,\left(\frac{\delta_\Omega(x)}{A_1r}\right)^\alpha,
\end{align*}
 where in the last inequality we used Lemma \ref{lem1'} and that $x_0 \in \WA(p,\Lambda)$. This yields our claim.

Next we wish to apply the Alt-Caffarelli-Friedman formula with 
\begin{align*}
u_1(x) &= \chi_{V_1}\,(\delta_\Omega(p)^n\,g(p,x)-t \,r_0)^+,\\
u_2(x) &= \chi_{V_2}\,(\delta_\Omega(p)^n\,g(p,x)-t \,r_0)^+.
\end{align*}
 It is clear that  both satisfy the hypotheses of Theorem \ref{t:ACF}.
For $i=1,2$ and $0<s<A_1r$, we denote
$$J_i(x_0,s) = \frac{1}{s^{2}} \int_{B(x_0,s)} \frac{|\nabla u_i(y)|^{2}}{|y-x_0|^{n-1}}dy,$$
so that
$J(x_0,s) = J_1(x_0,s)\,J_2(x_0,s)$.
 We claim that:
\begin{itemize}
\item[(i)] $J_i(x_0,s)\lesssim_\Lambda 1$ for $i=1,2$ and $0<s<\tfrac14A_1r$.
\item[(ii)] $J_i(x_0,2r) \gtrsim_{\Lambda,\lambda,r/r_0} 1$ for $i=1,2$.
\end{itemize}
The condition (i) follows from \rf{eqjr} and the fact that
\begin{equation}\label{eqgs99}
g(p,y)\lesssim \frac{s}{\delta_\Omega(p)^n}\quad \mbox{ for all $y\in B(x_0,s)$,}
\end{equation}
which holds by Lemma \ref{lem1'} and subharmonicity, since $x_0 \in \WA(p,\Lambda)$.
Concerning (ii), note first that
$$|\nabla u_1(y)| \lesssim \delta_\Omega(p)^n\,\frac{g(p,y)}{\delta_\Omega(y)} \lesssim_{\tau_0} \delta_\Omega(p)^n\,\frac{r_0}{\delta_\Omega(p)^n}=1\quad \mbox{ for all $y\in B(q,\tau_0r_0/2)$},$$
where we first used Cauchy estimates and then the pointwise bounds of $g(\cdot,\cdot)$ 
in \rf{eqgs99} with $s\approx\delta_\Omega(y)$.
 Thus, using also that $q \in V^{2t}$, we infer that $u_1(y)>t\,r_0/2
$ in some ball $B(q,ctr_0)$ with $c$ possibly depending on $\Lambda,\lambda,r/r_0$. Analogously, we deduce that $u_2(y)>t\,r_0/2$ in some ball $B(\wt q,ctr_0)$. 
Let $B$ be the largest open ball centered at $q$ not intersecting $\partial V_1$ and let $y_{0}\in \partial V_1\cap \partial B$.
 Then, by considering the convex hull $H\subset B$ of $B(q,ctr_0)$ and
$y_0$ and integrating in spherical coordinates (with the origin in $y_0$), one can check that 
$$\int_{H} |\nabla u_1|\,dy \gtrsim_t \,r_0^{n+1}.$$
 An analogous estimate holds for $u_2$, and then it easily follows that
$$J_i(x_0,2r_0)\gtrsim_t 1,$$
which implies (ii). We leave the details for the reader.

From the conditions (i) and (ii) and the fact that $J(x,r)$ is non-decreasing we infer that 
$$J(x_0,s)\approx_{\Lambda,\lambda,r/r_0} 1\quad \mbox{ for $2r<s<\tfrac14A_1r$.}$$
and also 
\begin{equation}\label{eqji*}
J_i(x_0,s)\approx_{\Lambda,\lambda,r/r_0} 1\quad \mbox{ for $i=1,2$ and $2r<s<\tfrac14A_1r$.} 
\end{equation}

Assume that $\tfrac14A_1=2^m$ for some big $m>1$. Since $J(x_0,s)$ is non-decreasing we infer that there exists some $h\in [1,m-1]$
such that
$$J(x_0,2^{h+1}r)  \leq C(\Lambda,\lambda,r/r_0)^{1/m} J(x_0,2^hr),$$
because otherwise, by iterating the reverse inequality, we  get a contradiction.
Now from Lemma \ref{lem:ACF2} we deduce that, given any $\eps>0$, for $m$ big enough, 
 there are constant $k_i\approx_{\Lambda,\lambda,r/r_0} 1$ and a unit vector $e$ such that
\begin{equation}\label{eqsum1122}
\|u_1- k_1\,((\cdot-x_0)\cdot e)^+\|_{\infty,B(x_0,2^hr)} + \|u_2- k_2\,((\cdot-x_0)\cdot e)^-\|_{\infty,B(x_0,2^hr)} \leq \eps\,2^h\,r.
\end{equation}
As a matter of fact, $\|u_i\|_{\infty,B(x_0,2^{h+1}r)}\approx_{\Lambda,\lambda,r/r_0} 2^hr$ by 
 \eqref{eqjr}, \rf{eqji*},  and \eqref{eqgs99}; $\|u_i\|_{\rm{Lip}^\alpha,B(x_0,2^{h+1}r)}\lesssim_{\Lambda,\lambda,r/r_0}(2^{h}r)^{1-\alpha}$ by Lemma \ref{lem333}; 
and the option (a) in Lemma \ref{lem:ACF2}
cannot hold (since we have  $\|u_i\|_{\infty,B(x_0,2^hr)}\approx_{\Lambda,\lambda,r/r_0} 2^hr$).

In particular, for $\eps$ small, \rf{eqsum1122} implies that if $q':=x_0 + 2^{h-1}re$, then one has $u_1(q') \approx_{\Lambda,\lambda,r/r_0} 2^{h-1}r$, and also that 
$$u_1(y) \approx_{\Lambda,\lambda,r/r_0} 2^{h-1}r>0 \quad\mbox{ for all $y\in B(q',2^{h-2}r)$}.$$
Thus  $B(q',2^{h-2}r)\subset\Omega$ and so $q'$ is at a distance at least $2^{h-2}r$ from $\partial\Omega$, and also 
$$g(p,q')\geq\frac{u_1(q')}{\delta_\Omega(p)^n}\approx_{\Lambda,\lambda,r/r_0} \frac{2^h\,r}{\delta_\Omega(p)^n}.$$
Further, since $q$ and $q'$ are both in $V_1$  by definition, there is a curve $\gamma$ which joins $q$ and $q'$ contained in $V_1$ satisfying
$$\dist(\gamma,\partial \Omega)\gtrsim_{A_1,\Lambda,t,r/r_0} r_0,$$
by \rf{eqvt5}. So $q'$ satisfies all the required properties in the lemma and we are done.
\end{proof}

\vv


\section{Types of cubes}\label{sechdld}

From now on we fix $R_0\in\DD$ and $p\in\Omega$ and we assume that we are under the assumptions
of the Main Lemma \ref{lemhc}.

We need now to define two families
$\HD$ and $\LD$ of high density and low density cubes, respectively.
Let $A\gg1$ be some fixed constant. We 
denote by $\HD$ (high density) the family of maximal cubes $Q\in\DD$ which are contained in $R_0$
and satisfy
$$\frac{\omega^{p}(2Q)}{\sigma(2Q)}\geq A \,\frac{\omega^{p}(2R_0)}{\sigma(2R_0)}.$$
We also denote by $\LD$ (low density) the family of maximal cubes $Q\in\DD$ which are contained in $R_0$
and satisfy
$$\frac{\omega^{p}(Q)}{\sigma(Q)}\leq A^{-1} \,\frac{\omega^{p}(R_0)}{\sigma(R_0)}$$
(notice that $\omega^{p}(R_0)\approx \omega^{p}(2R_0)\approx 1$ by assumption). Observe that the definition of the family $\HD$ involves the density of $2Q$, while the one of $\LD$ involves the density of $Q$.

We denote
$$B_H=\bigcup_{Q\in \HD} Q \quad\mbox{ and }\quad B_L=\bigcup_{Q\in \LD} Q.$$
\vv

\begin{lemma}\label{lemhd}
We have
$$\sigma(B_H) \lesssim \frac1A\,\sigma(R_0)\quad \text{ and }\quad \omega^{p}(B_L) \leq \frac1A\,\omega^{p}(R_0).$$
\end{lemma}

\begin{proof}
By Vitali's covering theorem, there exists a subfamily $I\subset\HD$ so that the cubes $2Q$, $Q\in I$, are pairwise disjoint and 
$$\bigcup_{Q\in\HD} 2Q\subset \bigcup_{Q\in I} 6Q.$$
Then, since $\sigma$ is doubling, we obtain
\begin{align*}
\sigma(B_H) \lesssim \sum_{Q\in I}\sigma(2Q) \leq \frac 1A
\sum_{Q\in I}\frac{\omega^{p}(2Q)}{\omega^{p}(2R_0)}\,\sigma(2R_0) \lesssim
\frac1A\,\sigma(R_0).
\end{align*}

Next we turn our attention to the low density cubes. Since the cubes from $\LD$ are pairwise disjoint, we have
\begin{align*}
\omega^{p}(B_L) = \sum_{Q\in\LD}\omega^{p}(Q) \leq\frac1A
\sum_{Q\in\LD}\frac{\sigma(Q)}{\sigma(R_0)}\,\omega^{p}(R_0) \leq \frac1A\,\omega^{p}(R_0). 
\end{align*}
\end{proof}

\vv
From the above estimates and the fact that the harmonic measure belongs to weak-$A_\infty$ (cf. \eqref{fwFRQAFRQAF}), we infer that
if $A$ is chosen big enough, then 
$$\omega^p(B_H) \leq \eps_0\,\omega^p(2\ball_{R_0}) \leq \frac14\,\omega^p(R_0)$$
and thus
$$\omega^p(B_H\cup B_L) \leq  \frac14\,\omega^p(R_0)
+ \frac1A\,\omega^p(R_0)
\leq \frac12\,\omega^p(R_0).$$
As a consequence, denoting $G_0= R_0\setminus (B_H\cup B_L))$, we deduce that
$$\omega^p(G_0)\geq\frac12\,\omega^p(R_0)\approx \omega^p(2\ball_{R_0}),$$
which implies that
$$\sigma(G_0)\gtrsim \sigma(2\ball_{R_0})\approx \sigma(R_0),$$
again using the fact that $\omega^p$ belongs to weak-$A_\infty$ in $\ball_{R_0}$.
So we have:

\begin{lemma}\label{lemg0}
Assuming $A$ big enough,
the set $G_0:= R_0\setminus (B_H\cup B_L))$ satisfies
$$\omega^p(G_0)\approx 1\quad \text{ and }\quad \sigma(G_0)\approx \sigma(R_0),$$
with the implicit constants depending on $C_0$ and the weak-$A_\infty$ condition in $\ball_{R_0}$.
\end{lemma}

\vv
 We denote by $\fG$ the family of those cubes $Q\in \DD(R_0)$ which are not contained in 
$\bigcup_{P\in  \HD\cup\LD}P$. In particular, such cubes $Q\in\fG$ do not belong to $\HD\cup\LD$ and 
\begin{equation}\label{defgr0}
A^{-1}\frac{\omega^p(R_0)}{\sigma(R_0)}\leq \frac{\omega^p(Q)}{\sigma(Q)}\lesssim\frac{\omega^p(2Q)}{\sigma(2Q)}\leq A\,\frac{\omega^p(2R_0)}{\sigma(2R_0)}.
\end{equation}
From this fact, it follows easily that $G_0$ is contained in the set $\WA(p,\Lambda)$ defined in Section
\ref{sect:short-paths}, assuming $\Lambda$ big enough, and so Lemma \ref{lemg0} ensures that \rf{eqwap} holds.
\vv

The following lemma is an immediate consequence of Lemma \ref{lemfacc}.

\begin{lemma}\label{lemcork1}
For every cube $Q\in\fG$ there exists some point $\exe_Q\in 2\ball_Q\cap\Omega$ such that 
$\delta_\Omega(\exe_Q)\geq \kappa_0\,\ell(Q)$ and
\begin{equation}\label{eqgg1}
g(p,\exe_Q)> c_3\,\frac{\ell(Q)}{\sigma(R_0)},
\end{equation}
for some $\kappa_0,c_3>0$, which depend on $A$ and on the weak-$A_\infty$ constants in $\ball_{R_0}$.
\end{lemma}

If  $\exe_Q\in 2\ball_Q\cap\Omega$ and $\delta_\Omega(\exe_Q)\geq \kappa_0\,\ell(Q)$, we say that $\exe_Q$ is {\it $\kappa_0$-corkscrew}
for $Q$. If \rf{eqgg1} holds, we say that $\exe_Q$ is a  {\it $c_3$-good corkscrew} for $Q$.
Abusing notation, quite often we will not write ``for $Q$".


\vv
We will need the following auxiliary result:
\begin{lemma}\label{lemclosejumps}
Let $Q\in\DD$ and let $\exe_Q$ be a $\lambda$-good $c_4$-corkscrew, for some $\lambda,c_4>0$. Suppose that $\ell(Q)\geq c_5\,\ell(R_0)$.
Then there exists some
$C$-good Harnack chain that joins $\exe_Q$ and $p$, with $C$ depending on $\lambda,c_5$.
\end{lemma}

\begin{proof}
Consider the open set $U=\{x\in\Omega:g(p,x)>\lambda\,\ell(Q)/\sigma(R_0)\}$. This is connected and thus there exists a curve $\gamma
\subset U$ that connects $\exe_Q$ and $p$.
By H\"older continuity, any point $x\in\Omega$  such that $\delta_\Omega(x)\leq \delta_\Omega(p)/2$,  satisfies
$$g(p,x)\leq c\,\left(\frac{\delta_\Omega(x)}{\ell(R_0)}\right)^\alpha\,\frac1{\ell(R_0)^{n-1}}.$$
Since $g(p,x)>\lambda\,\ell(Q)/\sigma(R_0)\gtrsim_{c_5,\lambda}\ell(R_0)^{1-n}$ for all $x\in U$, we then deduce that
$\dist(U,\partial\Omega)\geq c_6\,\ell(R_0)$
for some $c_6>0$ depending on $\lambda$ and $c_5$. Thus, 
$$\dist(\gamma,\partial\Omega)\geq c_6\,\ell(R_0).$$

From the fact that $g(p,x)\leq|p-x|^{1-n}$ for all $x\in\Omega$, we infer that any $x\in U$ satisfies
$$\lambda\,\frac{\ell(Q)}{\sigma(R_0)} < g(p,x)\leq \frac1{|p-x|^{n-1}}.$$
Therefore, 
$$|p-x|<\left(\frac{\sigma(R_0)}{\lambda\,\ell(Q)}\right)^{1/(n-1)}\lesssim_{c_5,\lambda}\ell(R_0).$$
So $U\subset B(p,C_2\,\ell(R_0))$ for some $C_2$ depending on $\lambda$ and $c_5$.
Next we consider a Besicovitch covering of $\gamma$ with balls $B_i$ of radius $c_6\ell(R_0)/2$. By volume considerations, it easily follows that the number of balls $B_i$ is bounded above by some constant $C_3$ depending on $\lambda$ and $c_5$, and thus this
is a $C$-good Harnack chain, with $C=C(\lambda,c_5)$.
\end{proof}

\vv

\begin{lemma}\label{lemshortjumps2}
There exists some constant $\kappa_1$ with $0<\kappa_1\leq\kappa_0$ such that the following holds for all $\lambda>0$. Let $Q\in\fG$, $Q\neq R_0$, and let $\exe_Q$ be a $\lambda$-good $\kappa_1$-corkscrew.
Then there exists some cube $R\in\fG$ with $Q\subsetneq R\subset R_0$ and $\ell(R)\leq C\,\ell(Q)$
and a $\lambda'$-good $\kappa_1$-corkscrew
$\exe_R$ such that $\exe_Q$ and $\exe_R$ can be joined by a $C'(\lambda)$-good Harnack chain, with $\lambda'>0$ and $C$
depending on $\lambda$.
\end{lemma}

The proof below yields a constant $\lambda'<\lambda$. On the other hand, the lemma ensures that 
$\exe_R$ is still a $\kappa_1$-corkscrew, which will be important for the arguments to come.

\begin{proof} This follows easily from Lemma \ref{lemshortjumps}. For completeness we will show the details.

 By choosing $\Lambda=\Lambda(A)>0$ big enough, 
$G_0\cap Q\subset\WA(p,\Lambda)$ and thus there exists some $x_0\in Q\cap \WA(p,\Lambda)$. 
We let 
$$\kappa_1=\min\big(\kappa_0, \kappa\big),$$ 
where $\kappa_0$ is defined in Lemma \ref{lemcork1} and
$\kappa$ in Lemma \ref{lemfacc} (and thus it depends only on $A$ and $C_0$).
We apply Lemma \ref{lemshortjumps}
to $x_0$, $q=z_Q$, 
with $r_0=3r(\ball_Q)$, $\lambda_0\approx \lambda$, and $r=4r(\ball_Q)$. To this end, note that
$$\delta_\Omega(q) \geq \kappa_1\,\ell(Q) =  \kappa_1\,\frac14\,\ell(r(\ball_Q)) = \kappa_1\,\frac1{12}\,r_0.$$
Hence there exists $q'\in B(x_0,A_1r)$ such that
\begin{equation}\label{eqkap12}
\delta_\Omega(q')\geq \kappa\,|x_0-q'|\geq \kappa\,r,\qquad g(p,q')\geq \lambda_1\,\frac{\delta_\Omega(q')}{\delta_\Omega(p)^n},
\end{equation}
and such that
$q$ and $q'$ can be joined by a curve $\gamma$ such that
\begin{equation}\label{eqgamma83}
\gamma\subset\{y\in B(x_0,A_1r):\dist(y,\partial\Omega)
>a_1\, r_0\},
\end{equation}
with $\lambda_1,A_1,a_1$ depending on $C_0,A,\lambda,\kappa_1$. Now let $R\in\DD$ be the cube containing $x_0$
such that
$$\frac12\,r(\ball_R)< |x_0-q'|\leq r(\ball_R).$$
Observe that
$$r(\ball_R) \geq |x_0-q'| \geq r = 4r(\ball_Q)\quad 
\text{ and }\quad r(\ball_R)< 2|x_0-q'|\leq 2A_1\,r\lesssim_\lambda \ell(Q).$$
 Also, we may assume that $\ell(R)\leq\ell(R_0)$ because otherwise we have $\ell(Q)\gtrsim A_1\,
\delta_\Omega(p)$ and then the statement in the lemma follows from Lemma \ref{lemclosejumps}.
So we have $Q\subsetneq R\subset R_0$.

From \rf{eqkap12} we get
$$\delta_\Omega(q') \geq\kappa\,|x_0-q'| \geq \frac12\,\kappa\,r(\ball_R) = 2\kappa\,\ell(R)
> \kappa_1\,\ell(R)$$
and 
$$ g(p,q')\geq c\,\lambda_1\,\frac{2\kappa\,\ell(R)}{\sigma(R_0)}.
$$
Hence, $q'$ is a $\lambda'$-good $\kappa_1$-corkscrew, for $\lambda' = c\lambda_12\kappa$.

From \rf{eqgamma83} and arguing as in the end of the proof of Lemma \ref{lemclosejumps} we infer that 
$\exe_Q=q$ and $\exe_R=q'$ can be joined by a $C(\lambda)$-good Harnack chain.
\end{proof}

\vv

From now on we will assume that all corkscrew points for cubes $Q\in\fG$ are $\kappa_1$-corkscrews, unless otherwise
stated. 

\vv

\section{The corona decomposition and the Key Lemma}\label{s9}

\subsection{The corona decomposition}

Recall that the $b\beta$ coefficient of a ball was defined in \rf{defbbeta}.
For each $Q\in\DD$, we denote
$$b\beta(Q) = b\beta_{\partial\Omega}(100\ball_Q).$$

Now we fix a constant $0<\eps\ll\min(1,\kappa_1)$.
Given $R\in \DD(R_0)$, we denote 
by $\Stop(R)$ the  maximal family of cubes $Q\in\DD(R)\setminus \{R\}$ satisfying that either $Q\not \in\fG$ or
$b\beta\bigl(\wh Q\bigr)>\eps$, where $\wh Q$ is the parent of $Q$.
Recall that the family $\fG$ was defined in \rf{defgr0}.
Note that,  by maximality, $\Stop(R)$ is a family of pairwise disjoint cubes. 

We define 
$$\tree(R):=\{ Q \in \DD (R): \nexists \,\, S \in \Stop(R) \,\, \textup{such that}\,\, Q\subset S\}.$$ 
 In particular, note that $\Stop(R)\not\subset\tree(R)$.
\vv

We now define the family of the top cubes with respect to $R_0$ as follows:
first we define the families $\ttt_k$ for $k\geq1$ inductively. We set
$$\ttt_1=\{R\in \DD(R_0)\cap\fG: \ell(R)= 2^{-10}\ell(R_0)\}.$$
Assuming that $\ttt_k$ has been defined, we set
$$\ttt_{k+1} = \bigcup_{R\in\ttt_k}(\Stop(R)\cap\fG),$$
and then we define
$$\ttt=\bigcup_{k\geq1}\ttt_k.$$
Notice that the family of cubes $Q\in\DD(R_0)$ with $\ell(Q)\leq 2^{-10}\ell(R_0)$ which are not contained in any cube $P\in\HD\cup\LD$
is contained in $\bigcup_{R\in\ttt}\tree(R)$,
and this union is disjoint. Also, all the cubes in that union belong to
$\fG$.

The following lemma is an easy consequence of our construction. Its proof is left for the reader.

\begin{lemma}\label{lem5.1}
We have
$$\ttt\subset \fG.$$
Also, for each $R\in\ttt$,
$$\tree(R) \subset \fG.$$
Further, for all $Q\in\tree(R)\cup \Stop(R)$,
$$\omega^p(2Q)\leq C\,A\,\frac{\sigma(Q)}{\sigma(R_0)}.$$
\end{lemma}

Remark that the last inequality holds for any cube $Q\in\Stop(R)$ because its parent $\wh Q$ belongs to $\tree(R)$ and so $\wh Q$ is not contained in any cube from $\HD$, which implies that $\omega^p(2Q)\leq \omega^p(2\wh Q)\lesssim A\,\frac{\sigma(\wh Q)}{\sigma(R_0)}\approx A\,\frac{\sigma(Q)}{\sigma(R_0)}.$

Using that $\partial\Omega$ is $n$-UR (by the assumption in the Main Lemma  \ref{lemhc}), it is easy to prove that the cubes from $\ttt$ satisfy a Carleson packing condition. This is shown in the next lemma.

\begin{lemma}\label{lempack}
We have
$$\sum_{R\in\ttt}\sigma(R) \leq M(\eps)\,\sigma(R_0).$$
\end{lemma}

\begin{proof}
For each $Q\in\ttt$ we have
$$\sigma(Q) = \sum_{P\in\Stop(Q)\cap\fG}\sigma(P) + \sum_{P\in\Stop(Q)\setminus\fG}\sigma(P)
+\sigma\biggl(Q\setminus\bigcup_{P\in\Stop(Q)} P\biggr).$$
Then we get
\begin{align}\label{eqspl23}
\sum_{Q\in\ttt} \sigma(Q) & \leq \sum_{Q\in\ttt}\sum_{P\in\Stop(Q)\cap\fG}\sigma(P)\\&\quad + \sum_{Q\in\ttt}\sum_{P\in\Stop(Q)\setminus\fG}\sigma(P)+\sum_{Q\in\ttt}\sigma\biggl(Q\setminus\bigcup_{P\in\Stop(Q)} P\biggr).\nonumber
\end{align}
Note now that, because of the stopping conditions, for all $Q\in\ttt
$, if $P\in\Stop(Q)\cap\fG$, then the parent $\wh P$ of $P$ satisfies
 $b\beta_{\partial\Omega}(100\ball_{\wh P})>\eps$. Hence, by Theorems 
\ref{teods} and \ref{teo*},
$$\sum_{Q\in\ttt}\sum_{P\in\Stop(Q)\cap\fG}\sigma(P) \leq \sum_{P\in\DD(R_0):b\beta_{\partial\Omega}(100\ball_{\wh P})>\eps}\sigma(P)
\leq C(\eps)\,\sigma(R_0).$$

On the other hand, the cubes $P\in\Stop(Q)\setminus\fG$ with $Q\in\ttt$ do not contain any cube from $\ttt$, by 
construction. Hence, they are disjoint and thus
$$\sum_{Q\in\ttt}\sum_{P\in\Stop(Q)\setminus\fG}\sigma(P)\leq \sigma(R_0).$$
By an analogous reason,
$$\sum_{Q\in\ttt}\sigma\biggl(Q\setminus\bigcup_{P\in\Stop(Q)} P\biggr)\leq \sigma(R_0).$$
Using \rf{eqspl23} and the estimates above, the lemma follows.
\end{proof}

\vv

Given a constant $K\gg1$, next we define 
\begin{equation}\label{defg0k}
G_0^K=\biggl\{x\in G_0: \sum_{R\in\ttt}\chi_R(x) \leq K\biggr\},
\end{equation}
By Chebyshev and the preceding lemma, we have
$$\sigma(G_0\setminus G_0^K) \leq \sigma(R_0\setminus G_0^K) \leq \frac1K \int_{R_0}\sum_{R\in\ttt}\chi_R\,
d\sigma \leq \frac{M(\eps)}K\,\sigma(R_0).$$
Therefore, if $K$ is chosen big enough (depending on $M(\eps)$ and the constants on the weak-$A_\infty$ condition), by Lemma \ref{lemg0} we get
$$\sigma(G_0\setminus G_0^K)\leq \frac12\,\sigma(G_0),$$
and thus
$$\sigma(G_0^K)\geq\frac12\,\sigma(G_0)\gtrsim \sigma(R_0).$$

\vv
We distinguish now two types of cubes from $\ttt$. We denote by $\ttt_a$ the family of cubes $R\in\ttt$ such that
$\tree(R)=\{R\}$, and we set $\ttt_b=\ttt\setminus \ttt_a$. Notice that, by construction, if $R\in\ttt_b$, then
$b\beta(R)\leq\eps$. On the other hand, this estimate may fail if $R\in\ttt_a$.

\vv


\subsection{The truncated corona decomposition}\label{secnnn}

For technical reasons, we need now to define a truncated version of the previous corona decomposition.
We fix a big natural number $N\gg 1$. Then we let $\ttt^{(N)}$ be the family of the cubes
from $\ttt$ with side length larger than $2^{-N}\ell(R_0)$. Given $R\in\ttt^{(N)}$ we let $\tree^{(N)}(R)$ be the 
subfamily of the cubes from $\tree(R)$  with side length larger than $2^{-N}\ell(R_0)$, and we let $\Stop^{(N)}(R)$
be a maximal subfamily from $\Stop(R)\cup\DD_{N}(R_0)$, where $\DD_{N}(R_0)$ is the subfamily of
the cubes from $\DD(R_0)$ with side length $2^{-N}\ell(R_0)$.  We also denote $\ttt_a^{(N)} =\ttt^{(N)}\cap \ttt_a$ and 
$\ttt_b^{(N)} =\ttt^{(N)}\cap \ttt_b$.

Observe that, since $\ttt^{(N)}\subset \ttt$, we also have
$$\sum_{R\in\ttt^{(N)}}\chi_R(x) \leq \sum_{R\in\ttt}\chi_R(x) \leq K\quad
\mbox{ for all $x\in G_0^K$.}$$

\vv


\subsection{The Key Lemma}

The main ingredient for the proof of the Main Lemma \ref{lemhc}  is the following result.

\begin{lemma}[Key Lemma] \label{keylemma3}
Given $\eta\in (0,1)$ and
$\lambda\in (0,c_3]$ (with $c_3$ as in \rf{eqgg1}),
there exists an exceptional family $\mathsf{Ex}(R)\subset \Stop(R)\cap\fG$ 
satisfying
$$\sum_{P\in\mathsf{Ex}(R)}\sigma(P)\leq \eta\,\sigma(R)$$
such that,
for every $Q\in \Stop(R)\cap\fG\setminus \mathsf{Ex}(R)$, 
 any $\lambda$-good corkscrew for $Q$ can be joined to some $\lambda'$-good corkscrew for $R$ by a $C(\lambda,\eta)$-good Harnack chain, with $\lambda'$ depending on $\lambda,\eta$.
\end{lemma}

This lemma will be proved in the next Sections \ref{sec6} and \ref{sec7}. Using this result, in Section
\ref{sec8} we will build the required carrot curves for the  Main Lemma \ref{lemhc}, which join the pole $p$ to points from a suitable big piece
of $R_0$. If the reader prefers to see how this is applied before its long proof, they may go directly to Section \ref{sec8}.  A crucial point in the Key Lemma is that the constant $\eps$ in the definition of the stopping cubes of the corona decomposition does not depend on the constants $\lambda$ or $\eta$ above.

To prove the Key Lemma \ref{keylemma3} we will need first to introduce the notion of ``cubes with well
separated big corkscrews" and we will split $\tree^{(N)}(R)$ into subtrees by introducing an additional stopping condition involving this type of cubes. Later on, in Section \ref{sec6} we will prove the ``Geometric Lemma", which
relies on a geometric construction which plays a fundamental role in the proof of the Key Lemma.


\vv

\subsection{The cubes with well separated big corkscrews}\label{subsepcork}

Let $Q\in\DD$ be a cube such that $b\beta(Q)\leq C_4\eps$. For example, $Q$ might be a cube from $Q\in\tree^{(N)}(R)\cup\Stop^{(N)}(R)$, with  $R\in\ttt^{(N)}_b$ (which in particular implies that $b\beta(R)\leq \eps$). We denote by $L_Q$ a best approximating $n$-plane for $b\beta(Q)$,
and we choose $\exe_Q^1$ and $\exe_Q^2$ to be two fixed points in $\ball_Q$ such that $\dist(\exe_Q^i,L_Q) = r(\ball_Q)/2$ and lie
in different components of $\RR^{n+1}\setminus L_Q$. So $\exe_Q^1$ and $\exe_Q^2$ are corkscrews for $Q$. We will call them ``big corkscrews".

Since any corkscrew $x$ for $Q$ satisfies $\delta_\Omega(x)\geq \kappa_1\,\ell(Q)$ and we have chosen
$\eps\ll\kappa_1$, it turns out that
$$\dist(x,L_Q)\geq \frac12\,\kappa_1\,\ell(Q)\gg \eps\,\ell(Q).$$
As a consequence, $x$ can be joined either to $\exe_Q^1$ or to $\exe_Q^2$ by a $C$-good Harnack chain, with 
$C$ depending only on $n,C_0,\kappa_1$, and thus only on $n$, $C_0$ and the weak-$A_\infty$ constants in $\ball_{R_0}$.
The following lemma follows by the same reasoning:

\begin{lemma}\label{lembigcork}
Let $Q,Q'\in\DD$ be cubes such that $b\beta(Q),b\beta(Q')\leq C_4\eps$ and $Q'$ is the parent of $Q$. Let $\exe_Q^i,\exe_{Q'}^i$, for $i=1,2$, be big corkscrews for $Q$ and $Q'$ respectively. Then, after relabeling the corkscrews if necessary, $\exe_Q^i$ can be joined to
$\exe_{Q'}^i$ by a $C$-good Harnack chain, with
$C$ depending only on $n,C_0,\kappa_1$.
\end{lemma}

Given $\Gamma>0$, we will write $Q\in \WSBC(\Gamma)$ (or just $Q\in\WSBC$, which stands for ``well separated big corkscrews") if $b\beta(Q)\leq C_4\eps$ and the big corkscrews $\exe_Q^1$, $\exe_Q^2$
 can {\em not} be joined by any $\Gamma$-good Harnack chain. The parameter $\Gamma$ will be chosen below. For the moment, let us say
 that $\Gamma^{-1}\ll\eps$. The reader should think that in spite of $b\beta(Q)\leq C_4\eps$, the possible existence of ``holes of size $C\,\eps\ell(Q)$ in $\partial\Omega$''
 makes possible the connection of the big corkscrews by means of $\Gamma$-Harnack chains passing through these holes.
Note that if $b\beta(Q)\leq C_4\eps$ and $Q\not\in \WSBC(\Gamma)$, then any pair of corkscrews for $Q$ can be connected by a $C(\Gamma)$-good Harnack chain, since any of these corkscrews can be joined by a good chain to one of the big corkscrews for $Q$, as mentioned above.

\vv

\subsection{The tree of cubes of type $\WSBC$ and the subtrees}\label{sec5.5}

Given $R\in\ttt^{(N)}_b$, denote
by $\Stop_\WSBC(R)$ the maximal subfamily of cubes $Q\in\DD(R)$ which satisfy that either
\begin{itemize}
\item $Q\not\in\WSBC(\Gamma)$, or
\item $Q\not\in\tree^{(N)}(R)$.
\end{itemize}
Also, denote by $\tree_\WSBC(R)$ the cubes from $\DD(R)$ which are not contained in any cube from 
$\Stop_\WSBC(R)$. So this tree is empty if $R\not\in\WSBC(\Gamma)$.  Notice also that
$\Stop_\WSBC(R)\not\subset\tree_\WSBC(R)$.

Observe that if $Q\in\Stop_\WSBC(R)$, it may happen that $Q\not\in\WSBC(\Gamma)$. However, 
unless $Q=R$, it holds that $Q\in\WSBC(\Gamma')$, with $\Gamma'>\Gamma$ depending only on $\Gamma$ and $C_0$
(because the parent of $Q$ belongs to $\WSBC(\Gamma)$).

For each $Q\in\Stop_\WSBC(R)\setminus\Stop(R)$, we denote 
$$\stree(Q) = \DD(Q)\cap\tree^{(N)}(R),\qquad \sstop(Q)=\Stop(R)\cap\DD(Q).$$
So we have
$$\tree^{(N)}(R) = \tree_\WSBC(R) \cup \bigcup_{Q\in\Stop_\WSBC(R)}\stree(Q),$$
and the union is disjoint.  Observe also that we have the partition
\begin{equation}\label{eqstop221}
\Stop(R) = \bigl(\Stop_\WSBC(R)\cap \Stop(R)\bigr) \cup \bigcup_{Q\in\Stop_\WSBC(R)\setminus\Stop(R)}\sstop(Q).
\end{equation}


\vv

\section{The geometric lemma}\label{sec6}

\subsection{The geometric lemma for the tree of  cubes of type $\WSBC$}

Let $R\in\ttt^{(N)}_b$ and suppose that $\tree_\WSBC(R)\neq \varnothing$.
We need now to define a family $\eend(R)$ of cubes from $\DD$, which in a sense can be considered as a 
regularized  version of  $\Stop_\WSBC(R)$. The first step consists of introducing the following auxiliary function:
$$d_R(x) :=\inf_{Q\in \tree_\WSBC(R)}(\ell(Q) + \dist(x,Q)),\quad\mbox{ for $x\in\RR^{n+1}$.}$$
Observe that $d_R$ is $1$-Lipschitz.

For each $x\in \partial\Omega$ we take the largest cube $Q_x\in\DD$ 
such that $x\in Q_x$ and
\begin{equation}\label{eqdefqx}
\ell(Q_x) \leq \frac1{300}\,\inf_{y\in Q_x} d_R(y).
\end{equation}
We consider the collection of the different cubes $Q_x$, $x\in \partial\Omega$, and we denote it by $\eend(R)$.

\vv

\begin{lemma}\label{lem74}
Given $R\in\ttt^{(N)}_b$, the cubes from $\eend(R)$ are pairwise disjoint and satisfy the following properties:
\begin{itemize}
\item[(a)] If $P\in\eend(R)$ and $x\in 50\ball_P$, then $100\,\ell(P)\leq d_R(x) \leq 900\,\ell(P)$. 

\item[(b)] There exists some absolute constant $C$ such that if $P,P'\in\eend(R)$ and $50\ball_P\cap 50\ball_{P'}
\neq\varnothing$, then
$C^{-1}\ell(P)\leq \ell(P')\leq C\,\ell(P).$
\item[(c)] For each $P\in \eend(R)$, there at most $N$ cubes $P'\in\eend(R)$ such that
$50\ball_P\cap 50\ball_{P'}
\neq\varnothing,$
 where $N$ is some absolute constant.
 
\item[(d)] If $P\in\eend(R)$ and $\dist(P,R)\leq 20\,\ell(R)$, then there exists some $Q\in\tree_\WSBC(R)$
such that $P\subset 22Q$ and $\ell(Q)\leq 2000\,\ell(P)$.
\end{itemize}
\end{lemma}

\begin{proof}
The proof is a routine task. For the reader's convenience we show the details.
To show (a), consider  $x\in 50 \ball_P$. Since $d_R(\cdot)$ is $1$-Lipschitz
and, by definition, $d_R(\zed_{P})\geq 300\,\ell(P)$, we have
$$d_R(x)\geq d_R(\zed_{P}) - |x-\zed_{P}| \geq d_R(\zed_{P}) - 50\,r(\ball_P)\geq 300 \,\ell(P)-200\,\ell(P)=100\,\ell(P) .$$

To prove the converse inequality, by the definition of $\eend(R)$, there exists some $z'\in
\wh P$, the parent of $P$, such that 
$$d_R(z')\leq 300\,\ell(\wh P)\ = 600\,\ell(P).$$
Also, we have
$$|x-z'|\leq |x-\zed_{P}| + |\zed_{P}-z'|\leq 50\,r(\ball_P) + 2\ell(P)\leq 300\,\ell(P).$$
Thus,
$$d_R(x)\leq d_R(z') + |x-z'| \leq (600+ 300)\,\ell(P).$$

The statement (b) is an immediate consequence of (a), and (c) follows easily from (b).
To show (d), observe that, for any $S\in\tree_\WSBC(R)$,
$$\ell(P)\leq\frac{d_R(\zed_P)}{300}\leq \frac{\ell(S) + \dist(\zed_P,S)}{300}
\leq \frac{\ell(P) + \ell(S) + \dist(P,S)}{300}.$$
Thus,
$$\ell(P)\leq\frac{\ell(S) + \dist(P,S)}{299}.$$
In particular, choosing $S=R$, we deduce
$$\ell(P)\leq\frac{\ell(R) + \dist(P,R)}{299}
\leq \frac{21}{299}\,\ell(R)\leq \ell(R),$$
and thus, using again that $\dist(P,R)\leq20\ell(R)$, it follows that 
$P\subset 22R$. Let $S_0\in\tree_\WSBC(R)$ be such that $d_R(\zed_P)=\ell(S_0) + \dist(\zed_P,S_0)$, and let $Q\in\DD$ be the smallest cube such that $S_0\subset Q$ and $P\subset 22Q$. Since $S_0\subset R$ and $P\subset 22R$, 
we deduce that $S_0\subset Q\subset R$, implying that $Q\in\tree_\WSBC(R)$. 

So it just remains to check that
$\ell(Q)\leq 2000\,\ell(P)$. To this end, consider a cube $\wt Q\supset S_0$ such that
$$\ell(P)+\ell(S_0)+\dist(P,S_0)\leq \ell(\wt Q)\leq 2\bigl(\ell(P)+\ell(S_0)+\dist(P,S_0)\bigr).$$
From the first inequality, it is clear that $P\subset 2\wt Q$ and then, by the definition of $Q$, we infer that $Q\subset\wt Q$. This inclusion and the second inequality above imply that 
$$\ell(Q)\leq\ell(\wt Q) \leq 2\bigl(\ell(P)+\ell(S_0)+\dist(\zed_P,S_0)\bigr) = 4\ell(P) + 2\,d_R(\zed_P).$$
By (a) we know that $d_R(\zed_P)\leq900\,\ell(P)$, and so we derive $\ell(Q) \leq 2000\,\ell(P)$.
\end{proof}

\vv

\begin{lemma}\label{lemc6}
Given $R\in\ttt^{(N)}_b$, if $Q\in\eend(R)$ and $\dist(P,R)\leq 20\,\ell(R)$, then $b\beta(Q)\leq C\,\eps$ and $Q\in \WSBC(\Gamma')$, with $\Gamma'=c_6\,\Gamma$, for some absolute
constants $C,c_6>0$.
\end{lemma}

\begin{proof}
This immediate from the fact that, by (d) in the previous lemma, there exists some cube $Q'\in\tree_\WSBC(R)$
such that $Q\subset 22Q'$ and $\ell(Q')\leq 2000\,\ell(Q)$, so that $b\beta(Q')\leq\eps$ and 
$Q'\in \WSBC(\Gamma)$.
\end{proof}

\vv

As in Section \ref{s2}, we make a standard Whitney decomposition of the open set $\Omega$. With a harmless abuse of notation we let 
$\mathcal{W}=\W(\Omega)$ denote a collection
of (closed) dyadic Whitney cubes of $\Omega$, so that the cubes in $\mathcal{W}$
form a pairwise non-overlapping covering of $\Omega$, which satisfy for some $M_0>20$ and $D_0\geq1$
\begin{itemize}
	\item[(i)] $10I \subset \Omega$;
	\item[(ii)] $M_0 I \cap \partial\Omega \neq \varnothing$;
	\item[(iii)] there are at most $D_0$ cubes $I'\in\WW$
	such that $10I \cap 10I' \neq \varnothing$. Further, for such cubes $I'$, we have $\ell(I')\approx
	\ell(I)$, where $\ell(I')$ stands for the side length of $I'$.
\end{itemize}
From the properties (i) and (ii) it is clear that $\dist(I,\partial\Omega)\approx\ell(I)$. We assume that
the Whitney cubes are small enough so that
\begin{equation}\label{eqeq29}
\diam(I)< \frac1{100}\,\dist(I,\partial\Omega).
\end{equation}
To construct this Whitney decomposition one just needs to  replace each cube $I\in\WW$, as in  \cite[Chapter VI]{St}, by its descendants $I'\in\DD_k(I)$, for some fixed $k\geq1$. 

For each $I\in\WW$, as much as in Lemma \ref{lem1}, we denote by $B^I$ a ball concentric with $I$ and radius $C_5\ell(I)$, where $C_5$ is a
universal constant big enough
so that
$$g(p,x) \lesssim \frac{\omega^p(B^I)}{\ell(I)^{n-1}}\quad\mbox{ for all $x\in 4I$,}$$
and whenever $p\notin 5I$. 
Obviously, the ball $B^I$ intersects $\partial\Omega$, and the family $\{B^I\}_{I\in\WW}$ does not have
finite overlapping.

Given a  bounded measurable set $F\subset \mathbb R^{n+1}$ with $|F|>0$, and a function $f\in L^1_{\rm loc}(\RR^{n+1})$, we denote by $m_F f$ the mean of $f$ in $F$ with respect to Lebesgue measure. That is,
$$m_F f = \avint_F f\,dx.$$

\vv
To state the Geometric Lemma we need some additional notation.
Given a cube $R'\in\tree_\WSBC(R)$,
we denote by $\wt\tree_\WSBC(R')$ the family of cubes from $\DD$ with side length at most $\ell(R')$ which 
are contained in $100\ball_{R'}$ and are
not contained in any cube from $\eend(R)$. We also denote by $\wt\eend(R')$ the subfamily of the cubes from $\eend(R)$ which 
are contained in some cube from $\wt\tree_\WSBC(R')$. Note that $\wt\tree_\WSBC(R')$ is not a tree, in general, but a union of trees. Further, from Lemma \ref{lem74}(a), it follows easily that
$$\tree_\WSBC(R)\cup \Stop_\WSBC(R)\subset \wt\tree_\WSBC(R)\cap \DD(R).$$

\vv

\begin{lemma}[Geometric Lemma]\label{lemgeom}
Let $0<\gamma<1$, and assume that the constant $\Gamma=\Gamma(\gamma)$ in the definition of $\WSBC$ is big enough.
Let $R\in\ttt^{(N)}_b\cap\WSBC(\Gamma)$ and let $R'\in\tree_\WSBC(R)$ be such that
$\ell(R')=2^{-k_0}\ell(R)$, with $k_0=k_0(\gamma)\geq1$ big enough. 
Then there are two open sets $V_1,V_2\subset C\ball_{R'}\cap\Omega$ with disjoint closures which satisfy the following properties: 
\begin{itemize}
\item[(a)] There are subfamilies $\WW_i\subset\WW$ such that
$V_i=\bigcup_{I\in \WW_i} 1.1\interior(I).$ 

\item[(b)] Each $V_i$ contains a ball $B_i$ with $r(B_i)\approx\ell(R')$, and each corkscrew point for $R'$ contained in $2\ball_{R'}\cap V_i$ can  be joined
to the center $z_i$ of $B_i$ by a good Harnack chain contained in $V_i$. Further, any point $x\in V_i$
can be joined to $z_i$ by a good Harnack chain (not necessarily contained in $V_i$).

\item[(c)] For each $Q\in\big(\tree_\WSBC(R)\cup\Stop_\WSBC(R)\big) \cap\DD(R')$ there are big corkscrews $\exe_Q^1\in V_1\cap 2\ball_Q$ and $\exe_Q^2\in V_2\cap 2\ball_Q$, and if $\wh Q$ is an
ancestor of $Q$ which also belongs to $\tree_\WSBC(R)\cap\DD(R')$, then $\exe_Q^i$ can be joined to $\exe_{\wh Q}^i$ by a good Harnack
chain, for each $i=1,2$.

\item[(d)] 
$(\partial V_1\cup\partial V_2)\cap 10\ball_{R'}
\subset \bigcup_{P\in\wt\eend(R')}2\ball_P$.

\item[(e)]  For each $P\in\wt\eend(R')$  such that $2\ball_P\cap10\ball_{R'}\neq\varnothing$,
let $\WW_P$ be the family of Whitney cubes $I\subset V_1\cup V_2$ such that $1.1\overline I\cap\partial (V_1\cup V_2)\cap 2\ball_P\neq\varnothing$, so that 
$$\partial V_i\cap 2\ball_P \subset \bigcup_{I\in\WW_P}1.1\overline I.$$ 
Then
\begin{itemize}
\item[(i)]  $$m_{4I} g(p,\cdot)\leq \gamma\,\frac{\ell(P)}{\sigma(R_0)}\quad \mbox{ for each $I\in\WW_P$,}$$ and
\item[(ii)] $$\sum_{I\in\WW_P} \ell(I)^n \lesssim \ell(P)^n
\quad\text{ and }\quad \sum_{I\in\WW_P} \omega^p(B^I) \lesssim \omega^p(C\ball_P),$$
for some universal constant $C>1$.
\end{itemize}
\end{itemize} 
The constants involved in the Harnack chain and corkscrew conditions may depend on $\eps$, $\Gamma$, and $\gamma$.\footnote{To guarantee the existence of the sets $V_i$ and the fact that they are
contained in $\Omega$ we use the assumption that $\Omega=\ree\setminus\partial\Omega$.}
\end{lemma}


\vv

\subsection{Proof of the Geometric Lemma \ref{lemgeom}}

In this whole subsection we fix $R\in\ttt^{(N)}_b$ and we assume $\tree_\WSBC(R)\neq \varnothing$, as in 
Lemma \ref{lemgeom}. We let $R'\in\tree_\WSBC(R)$ be such that
$\ell(R')=2^{-k_0}\ell(R)$, with $k_0=k_0(\gamma)\geq1$ big enough, as in Lemma \ref{lemgeom},
and we consider the associated families 
 $\wt\tree_\WSBC(R')$ and  $\wt\eend(R')$.

\begin{remark}\label{rem**}
By arguments analogous to the ones in Lemma \ref{lemc6}, it follows easily that if $Q\in\wt\tree_\WSBC(R')$, for
$R'\in\tree_\WSBC(R)$ such that $\ell(R')=2^{-k_0}\ell(R)$,
 then there exists some cube $S\in \tree_\WSBC(R)$ such that $Q\subset 22S$ and $\ell(S)\leq2000\ell(Q)$. This implies that 
 $b\beta(Q)\leq C\,\eps$ and $Q\in \WSBC(c_6\Gamma)$ too.
\end{remark}

In order to define the open sets $V_1$, $V_2$ described in the lemma, first we need to associate
some open sets $U_1(Q)$, $U_2(Q)$ to each $Q\in\wt\tree_\WSBC(R')\cup\wt\eend(R')$. We distinguish two cases:
\begin{itemize}
\item For $Q\in\wt\tree_\WSBC(R')$, we let $\mathcal J_i(Q)$ be the family of Whitney cubes $I\in\WW$ which intersect
$$\{y\in\ 20\ball_Q:\dist(y,L_Q)>\eps^{1/4}\,\ell(Q)\}$$
and are contained in the same connected component of $\RR^{n+1}\setminus L_Q$ as $\exe_Q^i$, and then we set
$$U_i(Q) = \bigcup_{I\in \mathcal J_i(Q)} 1.1\interior(I).$$

\item For $Q\in\wt\eend(R')$ the definition of $U_i(Q)$ is more elaborated. First we consider an auxiliary ball
$\wt \ball_Q$, concentric with $\ball_Q$, such that $19\ball_Q\subset \wt \ball_Q\subset 20\ball_Q$ and having thin boundaries for $\omega^p$. This means that, for some absolute constant $C$,
\begin{equation}\label{eqthinbd}
\omega^p\bigl(\bigl\{x\in 2\wt \ball_Q:\dist(x,\,\partial \wt \ball_Q)\leq t\,r(\wt \ball_Q)\bigr\}\bigr) \leq C\, t\,\omega^p(2\wt \ball_Q)
\quad\mbox{ for all $t>0$.}
\end{equation}
The existence of such ball $\wt \ball_Q$ follows by well known arguments (see for example \cite[p.370]{Tolsa-llibre}).

Next we denote by $\mathcal J(Q)$ the family of Whitney cubes $I\in\WW$ which
 intersect $\wt \ball_Q$ and satisfy $\ell(I)\geq \theta\,\ell(Q)$ for $\theta\in(0,1)$ depending on $\gamma$
(the reader should think that $\theta\ll\eps$ and that $\theta=2^{-j_1}$ for some $j_1\gg1$), and we set
\begin{equation}\label{eqdefuq}
U(Q) = \bigcup_{I\in \mathcal J(Q)} 1.1\interior(I).
\end{equation}
For a fixed $i=1$ or $2$,
let $\{D_j^i(Q)\}_{j\geq0}$ be the connected components of $U(Q)$ which satisfy one of the following properties:
\begin{itemize}
\item either $\exe_Q^i\in D_j^i(Q)$ (recall that $\exe_Q^i$ is a big corkscrew for $Q$), or
\item there exists some $y\in D_j^i(Q)$ such that $g(p,y)> \gamma\,\ell(Q)\,\sigma(R_0)^{-1}$ and there is a $C_6(\gamma,\theta)$-good
Harnack chain that joins $y$ to $\exe_Q^i$, for some constant $C_6(\gamma,\theta)$ to be chosen below.
\end{itemize}
Then we let $U_i(Q)=\bigcup_j D_j^i(Q)$. After reordering the sequence, we assume that $\exe_Q^i\in D_0^i(Q)$.
We let $\mathcal J_i(Q)$ be the subfamily of cubes from $\mathcal J(Q)$ contained in $U_i(Q)$.
\end{itemize}

In the case $Q\in\wt\tree_\WSBC(R')$, from the definitions, it is clear that the sets $U_i(Q)$ are open and connected and 
\begin{equation}\label{eqinter1}
\overline{U_1(Q)}\cap \overline{U_2(Q)}=\varnothing.
\end{equation}
 In the case $Q\in\wt\eend(R')$, the sets $U_i(Q)$ may fail to be connected. However, \rf{eqinter1} still holds if $\Gamma$ is chosen big enough (which will be the case). Indeed, if some component $D_j^i$ can be joined by $C_6(\gamma,\theta)$-good
Harnack chains both to $\exe_Q^1$ and $\exe_Q^2$, then there is a $C(\gamma,\theta)$-good Harnack chain that joins $\exe_Q^1$ to 
$\exe_Q^2$, and thus  $Q$ does not belong to $\WSBC(c_6\Gamma)$ if $\Gamma$ is taken big enough, which cannot happen by Lemma \ref{lemc6}.
Note also that the two components of
$$\{y\in\wt \ball_Q:\dist(y,L_Q)>\eps^{1/2}\,\ell(Q)\}$$
are contained in $D_0^1(Q)\cup D_0^2(Q)$, because  $b\beta(Q)\leq C\eps$ and we assume $\theta\ll\eps$.

\vv

The following is immediate:

\begin{lemma}\label{lemgeom3}
Assume that we relabel appropriately the sets $U_i(P)$ and corkscrews $\exe_P^i$ for $P\in\wt\tree_\WSBC(R')\cup\wt\eend(R')$. Then for all $Q,\wh Q\in\wt\tree_\WSBC(R')\cup\wt\eend(R')$ such that $\wh Q$ is the parent of $Q$ we have
\begin{equation}\label{eqlabel}
\bigl[\exe_Q^1,\exe_{\wh  Q}^1\bigr]\subset U_1(Q) \cap 
U_1(\wh Q)\quad \text{ and }\quad \bigl[\exe_Q^2,\exe_{\wh  Q}^2\bigr]\subset U_2(Q) \cap 
U_2(\wh Q).
\end{equation}
Further, 
$$\dist\bigl([\exe_Q^i,\exe_{\wh  Q}^i],\partial\Omega\bigr)\geq c\,\ell(Q)\quad\mbox{ for $i=1,2$,}$$
where $c$ depends at most on $n$ and $C_0$.
\end{lemma}

The labeling above can be chosen inductively. First we fix the sets $U_i(T)$ and corkscrews $x_{T}^i$
for every maximal cube $T$ from $\wt\tree_\WSBC(R')$ (contained in $100\ball_{R'}$ and with side length equal to $\ell(R')$). Further we assume
that, for any maximal cube $T$, the corkscrew $x^i_T$ is at the same side of $L_{R'}$ as $\exe_{R'}^i$, for each $i=1,2$ (this property will be used below).
Later we label the sons of each $T$ so that \rf{eqlabel} holds for any son $Q$ of $T$. Then we proceed with the grandsons of
$T$, and so on. We leave the details for the reader. 

The following result will be used later to prove the property (e)(i).

\begin{lemma}\label{lemei}
Suppose that the constant $k_0(\gamma)$ in  Lemma \ref{lemgeom} is big enough.
Let $Q\in\wt\eend(R')$ and assume $\theta$ small enough and $C_6(\gamma,\theta)$ big enough in the definition of $U_i(Q)$. If $y\in \wt \ball_Q$ satisfies
$g(p,y)> \gamma\,\ell(Q)\,\sigma(R_0)^{-1}$, then $y\in U_1(Q)\cup U_2(Q)$.
\end{lemma}

Recall $\wt \ball_Q$ is the ball with thin boundary appearing in \rf{eqthinbd}.

\begin{proof}
By the definition of $U_i(Q)$, it suffices to show that $y$ belongs to some component $D_j^i(Q)$  and that there is a $C_6(\gamma,\theta)$-good
Harnack chain that joins $y$ to $\exe_Q^i$.
To this end, observe that by the boundary H\"older continuity of $g(p,\cdot)$,
$$\gamma\,\frac{\ell(Q)}{\sigma(R_0)} \leq g(p,y)\leq C\,\left(\frac{\delta_\Omega(y)}{\ell(Q)}\right)^\alpha\,m_{30\ball_Q}g(p,\cdot)
\leq C\,\left(\frac{\delta_\Omega(y)}{\ell(Q)}\right)^\alpha\,\frac{\ell(Q)}{\sigma(R_0)},$$
where in the last inequality we used Lemma \ref{lem1'}. Thus,
$$\delta_\Omega(y)\geq c\,\gamma^{1/\alpha}\,\ell(Q),$$
and if $\theta$ is small enough, then $y$ belongs to some connected component of the set $U(Q)$ in \rf{eqdefuq}. By Lemma \ref{lem74}(d) there is a cube $Q'\in\tree_\WSBC(R)$
such that $Q\subset 22Q'$ and $\ell(Q')\approx\ell(Q)$. In particular, $\WA(p,\Lambda)\cap Q'\supset G_0\cap Q'\neq\varnothing$
and thus, by applying Lemma
\ref{lemshortjumps} with $q=y$ and $r_0=Cr(\ball_Q)$ (for a suitable $C>1$), 
it follows that there exists a $\kappa_1$-corkscrew $y'\in C(\gamma)\,\ball_Q$, with $C(\gamma)>20$ say, such that
$y$ can be joined to $y'$ by a $C'(\gamma)$-good Harnack chain. Assuming that the constant $k_0(\gamma)$ in Lemma \ref{lemgeom} is big enough,
it turns out that $y'\in 2\ball_{Q''}$ for some $Q''\in \tree_\WSBC(R)$ such that $22Q''\supset Q$. Since all the cubes $S$ such that $Q\subset S\subset 22Q''$ satisfy
$b\beta(S)\leq C\,\eps$, by applying Lemma \ref{lembigcork} repeatedly, it follows that $y'$ can be joined either to $\exe_Q^1$
or $\exe_Q^2$ by a $C''(\gamma)$-good Harnack chain. Then, joining both Harnack chains, it follows that $y$ can be joined either
to $\exe_Q^1$ or $\exe_Q^2$ by a $C'''(\gamma)$-good Harnack chain. So $y$ belongs to one of the components $D_j^i$, assuming
$C_6(\gamma,\theta)$ big enough.
\end{proof}

\vv

From now on we assume $\theta$ small enough and $C_6(\gamma,\theta)$ big enough so that the preceding
lemma holds. Also, we assume $\theta\ll \eps^4$.
We define
$$V_1 = \bigcup_{Q\in\wt\tree_\WSBC(R')\cup\wt\eend(R')} U_1(Q),\qquad V_2 = \bigcup_{Q\in\wt\tree_\WSBC(R')\cup\wt\eend(R')} U_2(Q).$$
Next we will show that 
$$\overline {V_1}\cap \overline{V_2}=\varnothing.$$
Since the number of cubes $Q\in\wt\tree_\WSBC(R')\cup\wt\eend(R')$ is finite (because of the truncation in the
corona decomposition), this is a consequence of the following:

\begin{lemma}
Suppose $\Gamma$ is big enough in the definition of $\WSBC$ (depending on $\theta$). For all $P,Q\in\wt\tree_\WSBC(R')\cup\wt\eend(R')$, we have
$$\overline{U_1(P)} \cap \overline{U_2(Q)}=\varnothing.$$
\end{lemma}

\begin{proof}
We suppose that $\ell(Q)\geq \ell(P)$
We also assume that $\overline{U_1(P)} \cap \overline{U_2(Q)}\neq\varnothing$ and then we will get a contradiction. Notice first that if $\ell(P)=\ell(Q)=2^{-j}\ell(R')$ for some $j\geq0$, then the corkscrews $\exe_P^i$ and $\exe_Q^i$ are at the same side of $L_Q$ for each $i=1,2$. This follows easily by induction on $j$.

\vv
\noi {\bf Case 1.}
Suppose first that $P,Q\in\wt\tree_\WSBC(R')$. 
Since the cubes from 
$\mathcal J_2(Q)$ have side length at least $c\,\eps^{1/4}\,\ell(Q)$, it follows that at least one of the cubes 
from $\mathcal J_1(P)$ has side length at least $c'\,\eps^{1/4}\,\ell(Q)$, which implies that
$\ell(P)\geq c''\,\eps^{1/4}\,\ell(Q)$, by the construction of $U_1(P)$.

Since $U_1(P) \cap U_2(Q)\neq\varnothing$, there exists some curve $\gamma=\gamma(\exe_P^1,\exe_Q^2)$ that joins $\exe_P^1$ and $\exe_Q^2$
such that $\dist(\gamma,\partial\Omega)\geq c\,\eps^{1/2}\,\ell(Q)$ because all the cubes from 
$\mathcal J_2(Q)$ have side length at least $c\,\eps^{1/4}\,\ell(Q)$, and the ones from $\mathcal J_1(P)$ 
have side length $\geq c\,\eps^{1/4}\,\ell(P)\geq c\,\eps^{1/2}\,\ell(Q)$. 

Let $\wh P$ be the ancestor of $P$ such that $\ell(\wh P)=\ell(Q)$. From the fact that 
$\overline{U_1(P)} \cap \overline{U_2(Q)}\neq\varnothing$, we deduce that $20\ball_P\cap 20\ball_Q\neq\varnothing$ and thus 
$20\ball_{\wh P}\cap 20\ball_Q\neq\varnothing$, and so $20\ball_{\wh P}\subset 60\ball_Q$.
This implies that $\exe_{\wh  P}^1$ is in the same connected
component as $\exe_Q^1$ and also that $\dist([\exe_Q^1,\exe_{\wh  P}^1],\partial\Omega)\gtrsim\ell(Q)$, because $b\beta(100\ball_Q)\leq\eps\ll1$
and they are at the same side of $L_Q$.

Consider now the chain $P=P_1\subset P_2\subset\ldots\subset P_m=\wh P$, so that $P_{i+1}$ is the parent
of $P_i$. Form the curve $\gamma'= \gamma'(\exe_{\wh  P}^1,\exe_P^1)$ with endpoints $\exe_{\wh  P}^1$ and $\exe_P^1$ by joining the segments
$[\exe_{P_i}^1,\exe_{P_{i+1}}^1]$. Since these segments
satisfy
$$\dist\bigl([\exe_{P_i}^1,\exe_{P_{i+1}}^1],\partial\Omega\bigr)\geq c\,\ell(P_i)\geq c\,\ell(P)\geq c\,\eps^{1/4}\,\ell(Q),$$
it is clear that $\dist(\gamma',\partial\Omega)\geq c\,\eps^{1/4}\,\ell(Q)$. 

Next 
we form  a curve $\gamma''= \gamma''(\exe_Q^1,\exe_Q^2)$ which joins  $\exe_Q^1$ to $\exe_Q^2$
by joining $[\exe_Q^1,\exe_{\wh  P}^1]$, 
$\gamma'(\exe_{\wh  P}^1,\exe_P^1)$, and
 $\gamma(\exe_P^1,\exe_Q^2)$. It follows easily that this is contained in $90\ball_Q$ and that $\dist(\gamma'',\partial\Omega)\geq c\,\eps^{1/2}\,\ell(Q)$. However, this is not possible because 
$\exe_Q^1$ and $\exe_Q^2$ are in different connected components of $\RR^{n+1}\setminus L_Q$ and 
$b\beta(Q)\leq\eps\ll\eps^{1/2}$ (since we assume
$\eps\ll1$).

\vv
\noi {\bf Case 2.}
Suppose now that $Q\in \wt\eend(R')
$. The arguments are quite similar to the ones above. In this case,
the cubes from
$\mathcal J_2(Q)$ have side length at least $\theta\,\ell(Q)$ and thus 
 at least one of the cubes 
from $\mathcal J_1(P)$ has side length at least $c\,\theta\,\ell(Q)$, which implies that
$\ell(P)\geq c'\,\theta\,\ell(Q)$.

Now there exists a curve $\gamma=\gamma(\exe_P^1,\exe_Q^2)$ that joints $\exe_P^1$ and $\exe_Q^2$
such that $\dist(\gamma,\partial\Omega)\geq c\,\theta^2\,\ell(Q)$ because all the cubes from 
$\mathcal J_2(Q)$ have side length at least $\theta\,\ell(Q)$, and the ones from $\mathcal J_1(P)$ 
have side length $\theta\,\ell(P)\geq c\,\theta^2\,\ell(Q)$. 

We consider again cubes $\wh P$ and $P_1,\ldots,P_m$ defined exactly as above. By the same reasoning as above, $\dist([\exe_Q^1,\exe_{\wh  P}^1],\partial\Omega)\gtrsim\ell(Q)$. We also define the curve $\gamma'
=\gamma'(\exe_{\wh  P}^1,\exe_P^1)$ which 
joins  $\exe_{\wh  P}^1$ to $\exe_P^1$ in the same way. In the present case we have
$$\dist(\gamma',\partial\Omega)\gtrsim \ell(P)\geq c\,\theta\,\ell(Q).$$
Again construct a curve $\gamma''= \gamma''(\exe_Q^1,\exe_Q^2)$ which joins  $\exe_Q^1$ to $\exe_Q^2$
by gathering $[\exe_Q^1,\exe_{\wh  P}^1]$, 
$\gamma'(\exe_{\wh  P}^1,\exe_P^1)$, and
 $\gamma(\exe_P^1,\exe_Q^2)$.  This is contained in $C\ball_Q$ (for some $C>1$ possibly depending on $\gamma$) and satisfies  $\dist(\gamma'',\partial\Omega)\geq c\,\theta^2\,\ell(Q)$. From this fact we deduce that $\exe_Q^1$ and $\exe_Q^2$
can be joined by $C(\theta)$-good Harnack chain. Taking $\Gamma$ big enough (depending on $C(\theta)$),
this implies that the big corkscrews for $Q$ can be joined by a $(c_6\Gamma)$-good Harnack chain, which contradicts
Lemma \ref{lemc6}.

\vv
\noi {\bf Case 3.}
Finally suppose that $P\in\wt\eend(R')$. We consider the same auxiliary cube $\wh P$ and the same curve
$\gamma=\gamma(\exe_P^1,\exe_Q^2)$ satisfying $\dist(\gamma,\partial\Omega)\geq c\,\theta\,\ell(P)$.
 By joining the segments $[\exe_{P_i}^2,\exe_{P_{i+1}}^2]$, we construct a curve $\gamma'_2=\gamma'_2(\exe_{\wh  P}^2,\exe_P^2)$ analogous
to $\gamma'=\gamma'(\exe_{\wh  P}^1,\exe_P^1)$ from the case 2, so that this joins $\exe_{\wh  P}^2$ to $\exe_P^2$
and satisfies $\dist(\gamma'_2,\partial\Omega)\gtrsim \ell(P)$.

We construct a curve $\gamma'''$ that joins $\exe_P^1$ to $\exe_P^2$ by joining $\gamma(\exe_P^1,\exe_Q^2)$,
$[\exe_Q^2,\exe_{\wh  P}^2]$, and $\gamma'_2(\exe_{\wh  P}^2,\exe_P^2)$. Again this is contained in 
 $C\ball_Q$ and it holds $\dist(\gamma''',\partial\Omega)\geq c\,\theta\,\ell(P)$.
This implies that $\exe_P^1$ and $\exe_P^2$
can be joined by $C(\theta)$-good Harnack chain. Taking $\Gamma$ big enough,
we deduce the big corkscrews for $P$ can be joined by a $(c_6\Gamma)$-good Harnack chain, which
is a contradiction.
\end{proof}

\vv
By the definition of $V_1$ and $V_2$ it is clear that the properties (a), (b) and (c) in Lemma \ref{lemgeom} hold. So to complete the
proof of the lemma it just remains to prove (d) and (e). 
\vv

\begin{proof}[\bf Proof of Lemma \ref{lemgeom}(d)]
Let $x\in(\partial V_1\cup\partial V_2)\cap 10\ball_{R'}$. We have to show that there exists some $S\in\wt\eend(R')$
such that $x\in 2\ball_S$.
To this end we consider $y\in\partial\Omega$ such that $|x-y|=\delta_\Omega(x)$. Since $\zed_{R'}\in\partial
\Omega$, it follows that $y\in 20\ball_{R'}$. Let $S\in\wt\eend(R')$ be such that $y\in S$. 
Observe that 
\begin{equation}\label{eqls9}
\ell(S)\leq\frac1{300}\,d_R(y)\leq \frac1{300}\,\bigl(\ell(R') + 20\,r(\ball_{R'})\bigr) = \frac{81}{300}
\,\ell(R')\leq \frac13\,\ell(R').
\end{equation}

We claim that
$x\in 2\ball_S$. 
Indeed, if $x\not\in 2\ball_S$, taking also into account \rf{eqls9}, 
there exists some ancestor $Q$ of $S$ contained in $100\ball_{R'}$  such that $x\in 2\ball_Q$ and  
$\delta_\Omega(x)=|x-y|\approx \ell(Q)$. 
From the fact that $S\subsetneq Q\subset 100 \ball_{R'}$ we deduce that 
$Q\in\wt\tree_\WSBC(R')$. By the construction of the sets $U_i(Q)$, it is immediate to check that the condition that $\delta_\Omega(x)\approx \ell(Q)$ implies that $x\in U_1(Q)\cup U_2(Q)$. Thus $x\in V_1\cup V_2$ and so $x\not \in\partial (V_1\cup V_2) = \partial V_1 \cup \partial V_2$ (for this identity we use that $\dist(V_1,V_2)>0$), which is a contradiction.
\end{proof}
\vv

To show (e), first we need to prove the next result:

\begin{lemma}\label{lembdr1}
For each $i=1,2$, we have
$$\partial V_i \cap 10\ball_{R'}
\subset \bigcup_{Q\in\wt\eend(R')} \partial U_i(Q).$$
\end{lemma}

\begin{proof}
Clearly, we have 
$$\partial V_i \cap 10\ball_{R'}
\subset \bigcup_{\substack{P\in\wt\tree_\WSBC(R'):\\P\cap 10\ball_{R'}\neq\varnothing}} \partial U_i(P) \cup 
\bigcup_{\substack{Q\in\wt\eend(R'):\\Q\cap 10\ball_{R'}\neq\varnothing}} \partial U_i(Q).$$
So it suffices to show that 
\begin{equation}\label{eqfg66}
\bigcup_{\substack{P\in\wt\tree_\WSBC(R'):\\P\cap 10\ball_{R'}\neq\varnothing}} \partial U_i(P)\cap 
\partial V_i\cap 10\ball_{R'}=\varnothing.
\end{equation}

Let $x\in \partial U_i(P)\cap \partial V_i\cap 10\ball_{R'}$, with $P\in\wt\tree_\WSBC(R')$, $P\cap 10\ball_{R'}\neq\varnothing$. From the definition
of $U_i(P)$, it follows easily that 
\begin{equation}\label{eqfg67}
\delta_\Omega(x)\gtrsim\eps^{1/4}\ell(P).
\end{equation}
On the other hand, by Lemma \ref{lemgeom}(d), there exists some $Q\in\wt\eend(R')$ such that $x\in 2\ball_Q$.
By the definition of $U_i(Q)$, since $\theta\ll \eps$, it also follows easily that 
$$\bigl\{y\in 2\ball_Q:\delta_\Omega(y)>\eps^{1/2}\ell(Q)\bigr\}\subset V_1\cup V_2.$$
Hence,
$\dist(\partial V_i\cap 2\ball_Q,\partial\Omega)\leq  \eps^{1/2}\,\ell(Q),$
and so 
\begin{equation}\label{eqfg68}
\delta_\Omega(x)\leq  \eps^{1/2}\,\ell(Q).
\end{equation}

 We claim that $\ell(Q)\lesssim\ell(P)$. Indeed, from the fact that $x\in \partial U_i(P)\subset 30\ball_P$, we infer that
$$30\ball_P\cap 2\ball_Q\neq \varnothing.$$
Suppose that $\ell(Q)\geq\ell(P)$. This implies that $\ball_P\subset33\ball_Q$. Consider now a cube $S\subset P$ 
belonging to $\wt\eend(R')$. Since $\ball_S\cap 33\ball_Q\neq \varnothing$, by Lemma \ref{lem74} (b) we have
$$\ell(Q)\approx \ell(S)\leq \ell(P),$$
which proves our claim. Together with \rf{eqfg67} and \rf{eqfg68}, this yields
$$
\eps^{1/4}\ell(P)\lesssim \delta_\Omega(x)\lesssim \eps^{1/2}\,\ell(Q)\lesssim \eps^{1/2}\,\ell(P),$$
which is a contradiction for $\eps$ small enough.
So there does not exist any $x\in \partial U_i(P)\cap \partial V_i\cap 10\ball_{R'}$, which proves \rf{eqfg66}.
\end{proof}

\vv

\begin{proof}[\bf Proof of Lemma \ref{lemgeom}(e)]
Let $P\in\wt\eend(R')$ be such that $2\ball_P\cap10\ball_{R'}\neq\varnothing$. 
The statement (i) is an immediate consequence of 
Lemma \ref{lemei}. In fact, this lemma implies that any $y\in 2\ball_P$ such that
$g(p,y)> \gamma\,\ell(P)\,\sigma(R_0)^{-1}$
is contained in $U_1(P)\cup U_2(P)$ and thus in $V_1\cup V_2$. In particular, $y\not \in\partial (V_1\cup V_2)=
\partial V_1 \cup\partial V_2$. Thus, if $y\in2\ball_P\cap\partial V_i$, then
$$g(p,y)\leq\gamma\,\frac{\ell(P)}{\sigma(R_0)}.$$
 It is easy to check that this implies the statement (i) in Lemma \ref{lemgeom}(e) (possibly after replacing $\gamma$ by $C\gamma$).

Next we turn our attention to (ii). To this end,
denote by $J_P$ the subfamily of the cubes $Q\in\wt\eend(R')$ such that $30\ball_Q\cap2\ball_P\neq\varnothing$.
By Lemma \ref{lembdr1},
\begin{equation}\label{eqsk54}
\partial V_i \cap 2\ball_P
\subset \bigcup_{Q\in J_P} \partial U_i(Q)\cap 2\ball_P.
\end{equation}
We will show that 
\begin{equation}\label{eq2eq}
\sum_{I\in\WW_P} \ell(I)^n \lesssim \ell(P)^n
\quad\text{ and }\quad \sum_{I\in\WW_P} \omega^p(B^I) \lesssim \omega^p(C\ball_P),
\end{equation}
where $\WW_P$ is the family of Whitney cubes $I\subset V_1\cup V_2$ such that $1.1\overline I\cap\partial (V_1\cup V_2)\cap 2\ball_P\neq\varnothing$.
To this end, observe that, by \rf{eqsk54} and the construction of $U_i(Q)$, for each $I\in\WW_P$ there exists some $Q\in J_P$ such that 
$I\subset 30\ball_Q$ and either $\ell(I)\approx\theta\ell(Q)$ or $1.1\overline I\cap\partial\wt \ball_Q\neq\varnothing$. Using the $n$-ADRity of $\sigma$, it is immediate to check
that for each $Q\in J_P$,
$$\sum_{\substack{I\subset 30\ball_Q:\\ \ell(I)=\theta\ell(Q)}}\ell(I)^n \lesssim \ell(Q)^n.$$
Also,
$$\sum_{\substack{I\in\WW:\\ 1.1\overline I\cap \partial\wt \ball_Q\neq\varnothing}}\ell(I)^n\lesssim \sum_{\substack{I\in \WW\\ 1.1\overline I\cap \partial\wt \ball_Q\neq\varnothing}}\HH^n(2I\cap \partial\wt \ball_Q) \lesssim \HH^n(\partial\wt \ball_Q)\lesssim \ell(Q)^n.$$
Since the number of cubes $Q\in J_P$ is uniformly bounded (by Lemma \ref{lem74}(b)) and $\ell(Q)\approx\ell(P)$,
the above inequalities yield the first estimate in \rf{eq2eq}. 

To prove the second one we also distinguish among the two types of cubes $I\in J_P$ above. First, by the
bounded overlap of the balls $B^I$ such that $\ell(I)=\theta\,\ell(Q)$, we get
\begin{equation}\label{eq1p9}
\sum_{\substack{I\subset 30\ball_Q\\ \ell(I)\approx\theta\ell(Q)}}\omega^p(B^I) \lesssim \omega^p(C\ball_P),
\end{equation}
since the balls $B^I$ in the sum are contained $C\ball_P$ for a suitable universal constant $C>1$.
To deal with the cubes $I\in \WW$ such that $1.1\overline I\cap \partial\wt \ball_Q\neq\varnothing$ we intend to use the thin boundary property of $\wt \ball_Q$ in \rf{eqthinbd}. To this end, we write
\begin{align*}
\sum_{\substack{I\in \WW:\\ 1.1\overline I\cap \partial\wt \ball_Q\neq\varnothing}}\omega^p(B^I) = 
\sum_{k\geq0}\sum_{\substack{I\in \WW:\\ 1.1\overline I\cap \partial\wt \ball_Q\neq\varnothing\\
\ell(I)=2^{-k}\ell(Q)}}\omega^p(B^I) \lesssim \sum_{k\geq0}\omega^p(\mathcal U_{2^{-k+1}\diam(Q)}(\partial\wt \ball_Q)),
\end{align*}
where $\mathcal U_d(A)$ stands for the $d$-neighborhood of $A$.
By \rf{eqthinbd} it follows that 
$$\omega^p(\mathcal U_{2^{-k}\ell(Q)}(\partial\wt \ball_Q))\lesssim 2^{-k}\omega^p(C'\ball_Q),$$
and thus
$$\sum_{\substack{I\in\WW:\\ 1.1\overline I\cap \partial\wt \ball_Q\neq\varnothing}}\omega^p(B^I)\lesssim \omega^p(C'\ball_Q)\lesssim \omega^p(C\ball_P),$$
for a suitable $C>1$. Together with \rf{eq1p9}, this yields the second inequality in \rf{eq2eq}, which completes the proof of Lemma \ref{lemgeom}(e).
\end{proof}

\vv


\section{Proof of the Key Lemma}\label{sec7}

We fix $R_0\in\DD$ and a corkscrew point $p\in\Omega$ as in the preceding sections.
We consider
$R\in\ttt^{(N)}_b$ and we assume $\tree_\WSBC(R)\neq \varnothing$, as in 
Lemma \ref{lemgeom}. We let $R'\in\tree_\WSBC(R)$ be such that
$\ell(R')=2^{-k_0}\ell(R)$, with $k_0=k_0(\gamma)\geq1$ big enough. 
Given $\lambda>0$ and $i=1,2$, we set
\begin{equation}\label{defhi}
\fH_i(R')=\bigl\{Q\in \Stop_\WSBC(R)\cap\DD(R')\cap\fG: g(p,\exe_{Q}^i)> \lambda\,\ell(Q)\,\sigma(R_0)^{-1}\bigr\},
\end{equation}
so that by Lemma \ref{lemcork1}, $\Stop_\WSBC(R)\cap\DD(R')\cap\fG = \fH_1(R')\cup \fH_2(R')$.
Here we are assuming that the corkscrews $\exe_Q^i$ belong to the set $V_i$ from Lemma \ref{lemgeom}, that $\lambda$ is small enough,  and we are taking into account that, by the arguments in
Section \ref{subsepcork}, any corkscrew for $Q$ can be joined to one of the big corkscrews $\exe_Q^1$ by some $C$-good Harnack chain.


\vv

\begin{lemma}[Baby Key Lemma]\label{keylemma}
Let $p,R_0,R,R'$ be as above.
Given $\lambda>0$, define also $\fH_i(R')$ as above.
For a given $\tau>0$, suppose that
$$\sigma\biggl(\,\bigcup_{Q\in \fH_i(R')}Q\biggr) \geq \tau\,\sigma(R').$$
If $\gamma$ is small enough in the definition of $V_i$ in Lemma \ref{lemgeom} (depending on $\tau$ and $\lambda$), then
$$g(p,\exe_{R'}^i)\geq c(\lambda,\tau)\,\frac{\ell(R')}{\sigma(R_0)}.$$
\end{lemma}

Remark that $\Gamma$ depends on $\gamma$ (see Lemma \ref{lemgeom}), and thus the families $\WSBC(\Gamma)$,
$\Stop_\WSBC(R)$, $\fH_i(R')$ also depend on $\gamma$. The reader should thing that $\Gamma\to\infty$
as $\gamma\to0$.

A key fact in this lemma is that the constants $\lambda,\tau$ can be taken arbitrarily small, without requiring
$\eps\to0$ as $\lambda\tau\to0$. Instead, the lemma requires $\gamma\to0$, which does not affect the packing condition in Lemma \ref{lempack}.

We denote
$$\bdy(R')=\bigcup_{P\in \wt\eend(R'):2\ball_P\cap 10\ball_{R'}\neq\varnothing}\WW_P,$$
with $\WW_P$ as in the Lemma \ref{lemgeom}. That is, 
$\WW_P$ is the family of Whitney cubes $I\subset V_1\cup V_2$ such that $1.1\overline I\cap\partial (V_1\cup V_2)\cap 2\ball_P\neq\varnothing$.
So the family $\bdy(R')$ contains Whitney cubes which intersect the boundaries of $V_1$ or $V_2$ and are close to $10\ball_{R'}$. 

Let us introduce some extra piece of notation. Given $q\in\ree$ and $0<r<s$ we let 
\[
A(q,r,s)=B(q,s)\setminus \overline{B(q,r)}.
\]
To prove Lemma \ref{keylemma}, first we need the following auxiliary result.

\begin{lemma}\label{lemgg}
Let $p,R_0,R,R'$ be as above and, for $i=1$ or $2$, let $Q\in\fH_i(R')$.
Let $V_i$ be as in Lemma \ref{lemgeom} and let $q\in \Omega$ be a corkscrew point for $Q$ which belongs to $V_i$.
Denote $r=2\ell(R')$ and for $\delta\in (0,1/100)$ set
$$A_r^\delta = \bigl\{x\in A(q,r,2r)\cap\Omega:\delta_\Omega (x)>\delta\,r\bigr\}.$$
Then we have
\begin{align*}
g(p,q)&\lesssim \frac{1}{r} \sup_{y\in A_r^\delta\cap V_i} \frac{g(p,y)}{\delta_\Omega (y)}\,\,\int_{A_r^\delta}  g(q,x)\,dx \\
&\quad + \frac{\delta^{\alpha/2}}{r^{n+3}} \int_{A(q,r,2r)}g(p,x)\,dx\,
\int_{A(q,r,2r)} g(q,x)\,dx\\
&\quad + \sum_{I\in\bdy(R')}\frac1{\ell(I)}\int_{2I} \bigl| g(p,x)\,\nabla  g(q,x)- \nabla g(p,x)\, g(q,x)\bigr|\,dx.
\end{align*}
\end{lemma}

Let us note that the fact that $q$ is a corkscrew for $Q$ contained in $V_i$ implies that
$\dist(q,\partial V_i)\approx \ell(Q)$, by the construction of the sets $V_i$ in Lemma
\ref{lemgeom}.

\begin{proof}
 We fix $i=1$, for definiteness. Recall that 
$V_1=\bigcup_{I\in \WW_1} 1.1\interior(I)$. 
For each $I\in\WW_1$, consider a smooth function $\eta_I$ such that $\chi_{0.9I}\leq \eta_I\leq \chi_{1.09I}$
with $\|\nabla\eta_I\|_\infty\lesssim \ell(I)^{-1}$ and
$$\eta:=\sum_{I\in\WW_1} \eta_I\equiv 1\quad \text{ on $V_1\cap 10\ball_{R'}\setminus \bigcup_{I\in\bdy(R')} 2I$}.$$
It follows that $\supp\eta\subset V_1$ and so
$\supp\eta \cap V_2=\varnothing,$
and also
$$\supp(\nabla\eta)\cap10\ball_{R'}\subset \bigcup_{I\in\bdy(R')} 2I.$$

 Let $\vphi_0$ be a smooth function such that $\chi_{B(q,1.2r)}\leq \vphi_0\leq \chi_{B(q,1.8r)}$, with $\|\nabla\vphi_0\|_\infty
\lesssim 1/r$. Then we set
$$\vphi = \eta\,\vphi_0.$$
So $\vphi$
is smooth, and it satisfies
$$\supp\nabla\vphi\subset \bigl(A(q,r,2r)\cap V_1\bigr)\cup \bigcup_{I\in \bdy(R')} 2I.$$
Observe that, in a sense, 
$\vphi$ is a smooth version of the function $\chi_{B(q,r)\cap V_1}$.

Since $g(p,q) = g(p,q)\,\vphi(q)$ and $g(p,\cdot)\,\vphi$ is a continuous function from $W_0^{1,2}(\Omega)$, we have
\begin{align*}
g(p,q)   & = \int_{\Omega} \nabla (g(p,\cdot)\,\vphi)(x)\,\nabla  g(q,x)\,dx\\
& = \int_{\Omega} g(p,x)\,\nabla\vphi(x)\,\nabla  g(q,x)\,dx
+ \int_{\Omega} \vphi(x)\,\nabla g(p,x)\,\nabla  g(q,x)\,dx\\
&=: I_1+ I_2.
\end{align*}

First we estimate $I_2$. For $\eps$ with $0<\eps<1/10$, we consider a smooth function $\vphi_\eps$ such that $\chi_{B(q,\eps \delta_\Omega (q))}\leq \vphi_\eps\leq \chi_{B(q,2\eps \delta_\Omega (q))}$, with $\|\nabla\vphi_\eps\|_\infty\lesssim 1/(\eps \delta_\Omega (q))$.
Since $\vphi_\eps\,\vphi = \vphi_\eps$, we have
\begin{multline*}
I_2= \int_{\Omega} \vphi_\eps(x)\,\nabla g(p,x)\,\nabla  g(q,x)\,dx + \int_{\Omega} \vphi(x)(1-\vphi_\eps(x))\,\nabla g(p,x)\,\nabla  g(q,x)\,dx 
\\ =: I_{2,a} + I_{2,b}.
\end{multline*}
To deal with $I_{2,a}$ we use the fact that for $x\in B(q,2\eps \delta_\Omega (q))$ we have
$$|\nabla  g(q,x)|\lesssim \frac1{|x-q|^n} \quad \text{ and }\quad |\nabla g(p,x)|\lesssim \frac{g(p,q)}{\delta_\Omega(q)}.$$
Then we get
$$|I_{2,a}|\lesssim \frac{g(p,q)}{\delta_\Omega(q)}\int_{B(q,2\eps \delta_\Omega(q))} \frac1{|x-q|^n}\,dx
\lesssim \frac{g(p,q)}{\delta_\Omega(q)}\,\eps\,\delta_\Omega(q) = \eps \,g(p,q).
$$

Let us turn our attention to $I_{2,b}$. We denote $\psi = \vphi(1-\vphi_\eps)$. Integrating by parts, we get
$$I_{2,b} = \int  \nabla g(p,x)\,\nabla (\psi\, g(q,\cdot))(x)\,dx - \int \nabla g(p,x)\,\nabla \psi(x)\,\, g(q,x)\,dx.$$
Observe now that the first integral vanishes because $\psi\, g(q,\cdot)\in W_0^{1,2}(\Omega)\cap C(\overline
\Omega)$ and vanishes at $\partial
\Omega$ and at $p$. Hence, since $\nabla\psi = \nabla \vphi - \nabla\vphi_\eps$, we derive
$$I_{2,b} = - \int \nabla g(p,x)\,\nabla \vphi(x)\,\, g(q,x)\,dx
+ \int \nabla g(p,x)\,\nabla \vphi_\eps(x)\,\, g(q,x)\,dx = I_3+ I_4.$$

To estimate $I_4$ we take into account that
$|\nabla \vphi_\eps|\lesssim \chi_{A(q,\eps \delta_\Omega(q),2\eps \delta_\Omega(q))}/ (\eps \delta_\Omega(q))$, and then we derive
$$|I_4|\lesssim  \frac1{\eps\,\delta_\Omega(q)}\int_{A(q,\eps \delta_\Omega(q),2\eps \delta_\Omega(q))} |\nabla g(p,x)|\, g(q,x)\,dx.$$
Using now that, for $x$ in the domain of integration,
$$ g(q,x)\lesssim \frac1{(\eps\,\delta_\Omega (q))^{n-1}} \quad \text{ and }\quad |\nabla g(p,x)|\lesssim \frac{g(p,q)}{\delta_\Omega (q)},$$
we obtain
$$|I_4|\lesssim  \frac1{\eps\,\delta_\Omega (q)}\,\frac1{(\eps\,\delta_\Omega (q))^{n-1}}\,\frac{g(p,q)}{\delta_\Omega (q)}\,(\eps \,\delta_\Omega (q))^{n+1}\lesssim \eps\,g(p,q).$$

From the above estimates we infer that
$$g(p,q) \leq |I_1 + I_3| + c\,\eps\,g(p,q).$$
Since neither $I_1$ nor $I_3$ depend on $\eps$, letting $\eps\to0$ we get
\begin{align*}
g(p,q)   & \leq |I_1 + I_3|\\
& \leq\left|\int g(p,x)\,\nabla\vphi(x)\,\nabla  g(q,x)\,dx
 - \!\int \nabla g(p,x)\,\nabla \vphi(x)\, g(q,x)\,dx\right|\\
 &\leq \int |\nabla\vphi(x)|\bigl| g(p,x)\,\nabla  g(q,x)- \nabla g(p,x)\, g(q,x)\bigr|\,dx.
\end{align*} 
We denote 
$$\wt F= \bigcup_{I\in\bdy(R')} 2I,$$
$$ \wt A_r^\delta = \bigl\{x\in A(q,1.2r,1.8r)\cap V_1\setminus \wt F:\delta_\Omega (x)>\delta\,r\bigr\},$$
and
$$\wt A_{r,\delta} = \bigl\{x\in A(q,1.2,1.8r)\cap V_1\setminus \wt F:\delta_\Omega (x)\leq\delta\,r\bigr\}.$$
Next we split the last integral as follows:
\begin{align}\label{eqj1j2}
g(p,q) &\leq \int_{\wt A_r^\delta} |\nabla\vphi(x)|\,\bigl| g(p,x)\,\nabla  g(q,x)- \nabla g(p,x)\, g(q,x)\bigr|\,dx\\
&\quad +
\int_{\wt A_{r,\delta}} |\nabla\vphi(x)|\,\bigl| g(p,x)\,\nabla  g(q,x)- \nabla g(p,x)\, g(q,x)\bigr|\,dx\nonumber
\\ 
&\quad+ \int_{\wt F} |\nabla\vphi(x)|\,\bigl| g(p,x)\,\nabla  g(q,x)- \nabla g(p,x)\, g(q,x)\bigr|\,dx\nonumber\\
&=: J_1 + J_2+ J_3.
\nonumber
\end{align}

Concerning $J_1$, we have
$$|\nabla  g(p,x)|\lesssim \frac{ g(p,x)}{\delta_\Omega (x)} \quad\text{ and }\quad |\nabla g(q,x)|\lesssim \frac{g(q,x)}{\delta_\Omega (x)}\quad
\text{ for all $x\in \wt A_r^\delta$.}$$
Thus, using also that $|\nabla\vphi|\lesssim 1/r$ outside $\wt F$,
\begin{equation}\label{eqj1*}
J_1 \lesssim \frac{1}{r} \sup_{x\in A_r^\delta\cap V_1} \frac{g(p,x)}{\delta_\Omega (x)}\,\,\int_{A_r^\delta}  g(q,x)\,dx.
\end{equation}

Regarding $J_2$, using  Cauchy-Schwarz, we get
\begin{align}\label{eqcasw1}
J_2&\lesssim \frac1r\int_{\wt A_{r,\delta}} \bigl| g(p,x)\,\nabla  g(q,x)- \nabla g(p,x)\, g(q,x)\bigr|\,dx\\
& \leq \frac1r
\left(\int_{\wt A_{r,\delta}} g(p,x)^2\,dx\right)^{1/2} \,\left(\int_{\wt A_{r,\delta}} |\nabla  g(q,x)|^2\,dx\right)^{1/2}
\nonumber
\\ &\,+\frac{1}{r}
\left(\int_{\wt A_{r,\delta}} |\nabla g(p,x)|^2\,dx\right)^{1/2} \left(\int_{\wt A_{r,\delta}} g(q,x)^2\,dx\right)^{1/2}.
\nonumber
\end{align}
To estimate the integral $\int_{\wt A_{r,\delta}} g(p,x)^2\,dx$, we take into account that, for all $x\in \wt A_{r,\delta}$,
$$g(p,x)\lesssim \delta^\alpha\;\avint_{A(q,r,2r)}g(p,y)\,dy.$$
Then we deduce
\begin{equation*}
\int_{\wt A_{r,\delta}} g(p,x)^2\,dx \lesssim \frac{\delta^\alpha}{r^{n+1}} \left(\int_{A(q,r,2r)}g(p,x)\,dx\right)^2.
\end{equation*}

Next we estimate the integral $\int_{\wt A_{r,\delta}} |\nabla  g(q,x)|^2\,dx$.
By covering $\wt A_{r,\delta}$ by a finite family of balls of radius $r/100$ and applying Caccioppoli's inequality to each one, it follows that
$$\int_{\wt A_{r,\delta}} |\nabla  g(q,x)|^2\,dx \lesssim \frac1{r^2}\int_{A(q,1.1r,1.9r)}  g(q,x)^2\,dx.$$
Since
$$ g(q,x)\lesssim \;\avint_{A(q,r,2r)} g(q,y)\,dy\quad \mbox{ for all $x\in A(q,1.1r,1.9r)$,}$$ we get
$$\int_{\wt A_{r,\delta}} |\nabla  g(q,x)|^2\,dx \lesssim
\frac1{r^2}\int_{A(q,1.1r,1.9r)}  g(q,x)^2\,dx\lesssim
\frac{1}{r^{n+3}}
 \left(\int_{A(q,r,2r)} g(q,x)\,dx\right)^2.$$
So we obtain
\begin{multline*}
\left(\int_{\wt A_{r,\delta}} \!\!g(p,x)^2\,dx\right)^{1/2} \,\left(\int_{\wt A_{r,\delta}} |\nabla  g(q,x)|^2\,dx\right)^{1/2}
\\ \lesssim \frac{\delta^{\alpha/2}}{r^{n+2}} \int_{A(q,r,2r)}g(p,x)\,dx\,
\int_{A(q,r,2r)} g(q,x)\,dx.
\end{multline*}
By interchanging, $p$ and $q$, it is immediate to check that an analogous estimate holds for the second summand on the
right hand side of \rf{eqcasw1}.
Thus we get
\begin{equation}\label{eqfi8}
J_2
\lesssim \frac{\delta^{\alpha/2}}{r^{n+3}} \int_{A(q,r,2r)}g(p,x)\,dx\,
\int_{A(q,r,2r)} g(q,x)\,dx.
\end{equation}

Concerning $J_3$, we just take into account that $|\nabla\vphi|\lesssim 1/\ell(I)$ in $2I$, and then we obtain
$$J_3\lesssim \sum_{I\in\bdy(R')} \frac1{\ell(I)}\int_{2I} \bigl| g(p,x)\,\nabla  g(q,x)- \nabla g(p,x)\, g(q,x)\bigr|\,dx.$$
Together with \rf{eqj1j2}, \rf{eqj1*}, and \rf{eqfi8}, this yields the lemma.
\end{proof}
\vv

\begin{proof}[\bf Proof of Lemma \ref{keylemma}]
 We fix $i=1$, for definiteness.
By a Vitali type covering theorem, there exists a subfamily $\wt\fH_1(R')\subset \fH_1(R')$ such that the balls
$\{8\ball_Q\}_{Q\in  \wt\fH_1(R')}$ are disjoint and 
$$\sum_{Q\in \fH_1(R')}\sigma(Q) \lesssim \sum_{Q\in \wt\fH_1(R')}\sigma(Q).$$

 By Lemma \ref{lemgg}, for each $Q\in\wt\fH_1(R')$  we have
\begin{align*}
g(p,\exe_Q^1)&\lesssim \frac{1}{r} \sup_{y\in 2\ball_{R'}\cap V_1:\delta_\Omega(y)\geq \delta\ell(R')} \frac{g(p,y)}{\delta_\Omega (y)}\,\,\int_{A(\exe_Q^1,r,2r)}  g(\exe_Q^1,x)\,dx \\
&\quad + \frac{\delta^{\alpha/2}}{r^{n+3}} \int_{A(\exe_Q^1,r,2r)}g(p,x)\,dx\,
\int_{A(\exe_Q^1,r,2r)} g(\exe_Q^1,x)\,dx\\
&\quad + 
\sum_{I\in\bdy(R')}\frac1{\ell(I)}\int_{2I}
 \bigl| g(p,x)\,\nabla  g(\exe_Q^1,x)- \nabla g(p,x)\, g(\exe_Q^1,x)\bigr|\,dx\\
& =: I_1(Q) + I_2(Q) + I_3(Q),
\end{align*}
with $r=2\ell(R')$.
Since $g(p,\exe_{Q}^1)> \lambda\,\ell(Q)/\sigma(R_0)$, we derive
\begin{multline}\label{eqi123}
\lambda\tau\,\sigma(R')\lesssim \lambda\!\sum_{Q\in \wt\fH_1(R')}\!\sigma(Q) \lesssim 
\sum_{Q\in \wt\fH_1(R')} \!g(p,\exe_Q^1)\,\ell(Q)^{n-1}\,\sigma(R_0)
\\
\lesssim \sum_{j=1}^3 \sum_{Q\in \wt\fH_1(R')}  I_j(Q)\,\ell(Q)^{n-1}\,\sigma(R_0).
\end{multline}
\vv
\subsection*{Estimate of $\sum_{Q\in \wt\fH_1(R')}  I_1(Q)\,\ell(Q)^{n-1}$}

We have
\begin{multline*}\sum_{Q\in \wt\fH_1(R')}  I_1(Q)\,\ell(Q)^{n-1}
\\ \leq 
\frac{1}{r} \sup_{y\in 2\ball_{R'}\cap V_1:\delta_\Omega(y)\geq \delta\ell(R')} \frac{g(p,y)}{\delta_\Omega (y)}\,\sum_{Q\in \wt\fH_1(R')}\int_{A(\exe_Q^1,r,2r)}  g(\exe_Q^1,x)\,dx\,\ell(Q)^{n-1}.
\end{multline*}
Note now that
\begin{multline*}
\sum_{Q\in \wt\fH_1(R')}\int_{A(\exe_Q^1,r,2r)}  g(\exe_Q^1,x)\,dx \,\ell(Q)^{n-1}
\lesssim \int_{2\ball_{R'}} \sum_{Q\in \wt\fH_1(R')} \omega^x(4Q)\,dx 
\\
\leq \int_{2\ball_{R'}} 1\,dx\lesssim \ell(R')^{n+1},
\end{multline*}
where we used the fact that the cubes $4Q$, with $Q\in\wt\fH_1(R')$, are pairwise disjoint.
Since $r\approx\ell(R')$, we derive
$$\sum_{Q\in \wt\fH_1(R')}  I_1(Q)\,\ell(Q)^{n-1}\lesssim \sup_{y\in 2\ball_{R'}\cap V_1:\delta_\Omega(y)\geq \delta\ell(R')} \frac{g(p,y)}{\delta_\Omega (y)} \,\sigma(R').$$

\vv
\subsection*{Estimate of $\sum_{Q\in \wt\fH_1(R')}  I_2(Q)\,\ell(Q)^{n-1}$}

First we estimate $\int_{A(\exe_Q^1,r,2r)}g(p,x)\,dx$ by applying Lemma \ref{lem1'}:
\begin{multline*}
\int_{A(\exe_Q^1,r,2r)}g(p,x)\,dx\leq \int_{2\ball_{R'}}g(p,x)\,dx\lesssim 
\ell(R')^{n+1}\,\frac{\omega^p(8\ball_{R'})}{\ell(R')^{n-1}}
\\
 \lesssim \ell(R')^2\,\frac{\sigma(R')}{\sigma(R_0)}\approx
\frac{r^{n+2}}{\sigma(R_0)}.
\end{multline*}
So we have
\begin{align*}
\sum_{Q\in \wt\fH_1(R')}  I_2(Q)\,\ell(Q)^{n-1} & \lesssim
\frac{\delta^{\alpha/2}}{r\,\sigma(R_0)} \sum_{Q\in \wt\fH_1(R')}\int_{A(\exe_Q^1,r,2r)} g(\exe_Q^1,x)\,dx\,\ell(Q)^{n-1}\\
 & \lesssim\frac{\delta^{\alpha/2}}{r\,\sigma(R_0)}\int_{2\ball_{R'}} \sum_{Q\in \wt\fH_1(R')} \omega^x(4Q)\,dx\\
 & \lesssim
\frac{\delta^{\alpha/2}}{r\,\sigma(R_0)} \int_{2\ball_{R'}} 1\,dx \lesssim
\frac{\delta^{\alpha/2}\,\sigma(R')}{\sigma(R_0)}.
\end{align*}

\vv
\subsection*{Estimate of $\sum_{Q\in \wt\fH_1(R')}  I_3(Q)\,\ell(Q)^{n-1}$}
Note first that, for each $I\in\bdy(R')$, since $\exe_Q^1\not \in 4I$, using the subharmonicity of $g(p,\cdot)$ and
$g(\exe_Q^1,\cdot)$ in $4I$, and Caccioppoli's inequality,
\begin{align*}
\frac1{\ell(I)}\int_{2I}\bigl| g(p,x)\,\nabla  g(\exe_Q^1,x)\bigr|\,dx&\lesssim
\frac1{\ell(I)}\sup_{x\in 2I} g(p,x) \int_{2I} |\nabla  g(\exe_Q^1,x)|\,dx\\
&\lesssim \ell(I)^{n-1}\,m_{4I}g(p,\cdot) \,\, m_{4I}g(\exe_Q^1,\cdot).
\end{align*}
By very similar estimates, we also get
$$\frac1{\ell(I)}\int_{2I}\bigl| \nabla g(p,x)\, g(\exe_Q^1,x)\bigr|\,dx\lesssim\ell(I)^{n-1}\,m_{4I}g(p,\cdot) \,\, m_{4I}g(\exe_Q^1,\cdot).$$
Recall now that, by Lemma \ref{lemgeom}(e)(i),
$$
m_{4I} g(p,\cdot)\leq \gamma\,\frac{\ell(P)}{\sigma(R_0)}\quad $$ 
for each $I\in\WW_P$, 	with $P\in \wt\eend(R')$ such that $2\ball_P\cap10\ball_{R'}\neq\varnothing$.

We distinguish two types of Whitney cubes $I\in\bdy(R')$. We write $I\in T_1$ if $\ell(I)\geq \gamma^{1/2}\ell(P)$ for some $P$ such that $I\in\WW_P$ and $2\ball_P\cap 10\ball_{R'}\neq\varnothing$, and we write $I
\in T_2$ otherwise (there may exist more than one $P$ such that
$I\in\WW_P$, but if $\WW_P\cap\WW_{P'}\neq\varnothing$, then $\ell(P)\approx\ell(P')$).
So we split
\begin{align}\label{eqs1s2}
\sum_{Q\in \wt\fH_1(R')}  I_3(Q)\,\ell(Q)^{n-1} & 
\leq 
\sum_{Q\in \wt\fH_1(R')} \sum_{I\in\bdy(R')}\ell(I)^{n-1}\,m_{4I}g(p,\cdot) \,\, m_{4I}g(\exe_Q^1,\cdot)\,\ell(Q)^{n-1}\nonumber\\
&= \sum_{Q\in \wt\fH_1(R')} \sum_{I\in T_1}\ldots + \sum_{Q\in \wt\fH_1(R')} \sum_{I\in T_2}\ldots =: S_1+S_2.
\end{align}

Concerning the sum $S_1$ we have
\begin{align*}
S_1 &
\lesssim  \gamma
\sum_{Q\in \wt\fH_1(R')} \sum_{\substack{P\in \wt\eend(R'):\\2\ball_P\cap10\ball_{R'}\neq\varnothing}} \sum_{I\in \WW_P\cap T_1}\frac{\ell(P)}{\sigma(R_0)}\,
\ell(I)^{n-1}\, m_{4I}g(\exe_Q^1,\cdot)\,\ell(Q)^{n-1}\\
&\lesssim\gamma^{1/2}
\sum_{Q\in \wt\fH_1(R')} \sum_{\substack{P\in \wt\eend(R'):\\2\ball_P\cap10\ball_{R'}\neq\varnothing}} \sum_{I\in \WW_P}\frac{\ell(I)^{n}}{\sigma(R_0)}\,
 m_{4I}g(\exe_Q^1,\cdot)\,\ell(Q)^{n-1}
\end{align*}
Next we take into account that
$$\ell(Q)^{n-1}\, m_{4I}g(\exe_Q^1,\cdot)\lesssim \omega^{x_I}(4Q),$$ 
where $x_I$ stands for the center of $I$ and $C>1$ is some absolute constant. 
This follows from Lemma \ref{lem1'} if $x_I$ is far from $Q$, and it can be deduced from Lemma \ref{l:bourgain} when $x_I$ is close to $Q$ (in this case, $\omega^{x_I}(4Q)\approx1$).
Then we derive
$$S_1
\lesssim \gamma^{1/2}
\sum_{Q\in \wt\fH_1(R')} \sum_{\substack{P\in \wt\eend(R'):\\2\ball_P\cap10\ball_{R'}\neq\varnothing}} \sum_{I\in \WW_P}
\omega^{x_I}(4Q)\,\frac{\ell(I)^n}{\sigma(R_0)}.$$
Since $\sum_{Q\in \wt\fH_1(R')}\omega^{x_I}(4Q)\lesssim 1$ for each $I$, we get
$$S_1
\lesssim \gamma^{1/2}
 \sum_{\substack{P\in \wt\eend(R'):\\2\ball_P\cap10\ball_{R'}\neq\varnothing}} \sum_{I\in \WW_P}
\frac{\ell(I)^n}{\sigma(R_0)}.$$
By Lemma \ref{lemgeom}(e)(ii), we have $\sum_{I\in \WW_P} \ell(I)^n\lesssim\ell(P)^n$, and so we deduce
$$S_1\lesssim \gamma^{1/2}
 \sum_{\substack{P\in \wt\eend(R'):\\2\ball_P\cap10\ball_{R'}\neq\varnothing}} 
\frac{\sigma(P)}{\sigma(R_0)}\lesssim \gamma^{1/2}\frac{\sigma(R')}{\sigma(R_0)}.$$

\vv

Next we turn our attention to the sum $S_2$ in \rf{eqs1s2}. Recall that
\begin{align*}
S_2& =\sum_{Q\in \wt\fH_1(R')} 
 \sum_{I\in  T_2}
\ell(I)^{n-1}\,m_{4I}g(p,\cdot) \,\, m_{4I}g(\exe_Q^1,\cdot)\,\ell(Q)^{n-1}.
\end{align*}
Let us remark that we assume the condition that $I\in\WW_P$ for some 
$2P\in \wt\eend(R')$ such that $2\ball_P\cap10\ball_{R'}\neq\varnothing$ to be part of the definition of $I\in T_2$.
Using the estimate
$m_{4I}g(p,\cdot) \lesssim \omega^p(B^I)\,\ell(I)^{1-n}$,
we derive
\begin{align*}
S_2& \lesssim\sum_{Q\in \wt\fH_1(R')} 
 \sum_{I\in T_2}
\omega^p(B^I) \, m_{4I}g(\exe_Q^1,\cdot)\,\ell(Q)^{n-1}\\
& = 
\sum_{Q\in \wt\fH_1(R')} 
 \sum_{I\in T_2:20I\cap 20 \ball_Q\neq\varnothing}
\ldots +
\sum_{Q\in \wt\fH_1(R')} 
 \sum_{I\in T_2:20I\cap 20 \ball_Q =\varnothing}
\ldots =: A+ B.
\end{align*}
To estimate the term $A$ we take into account that if $20I\cap 20\ball_Q\neq\varnothing$ and $I\in\WW_P$, then
$\ell(P)\lesssim \ell(Q)$ and thus $\ell(I)\lesssim\gamma^{1/2}\,\ell(Q)$ because $I\in T_2$.
As a consequence, $I\subset 21\ball_Q$ and also, by the H\"older 
continuity of $g(\exe_Q^1,\cdot)$, if we let $B$ be a ball concentric with $B^I$ with radius comparable to $\ell(Q)$ and such
that $\dist(\exe_Q^1,B)\approx \ell(Q)$, 
we obtain
$$m_{2B^I}g(\exe_Q^1,\cdot)\lesssim \biggl(\frac{r(B^I)}{r(B)}\biggr)^\alpha\,m_B g(\exe_Q^1,\cdot)\lesssim \gamma^{\alpha/2}\,\frac1{\ell(Q)^{n-1}},$$
where $\alpha>0$ is the exponent of H\"older continuity.
Hence,
$$A\lesssim \gamma^{\alpha/2} \sum_{Q\in \wt\fH_1(R')} \sum_{\substack{P\in \wt\eend(R'):\\2\ball_P\cap10\ball_{R'}\neq\varnothing\\
20\ball_P\cap20 \ball_Q\neq\varnothing}} 
\sum_{I\in \WW_P\cap T_2}
\omega^p(B^I).$$
By Lemma \ref{lemgeom}(e)(ii), we have $\sum_{I\in \WW_P}
\omega^p(B^I)\lesssim \omega^p(C\ball_P)$,
and using also that, for $P$ as above, $C\ball_P\subset C'\ball_Q$ for some absolute constant $C'$, we obtain
$$A\lesssim  \gamma^{\alpha/2} \sum_{Q\in \wt\fH_1(R')} 
\omega^p(C' \ball_Q)\lesssim \gamma^{\alpha/2} \sum_{Q\in \wt\fH_1(R')} \frac{\sigma(Q)}{\sigma(R_0)}\lesssim \gamma^{\alpha/2} \frac{\sigma(R')}{\sigma(R_0)}
.$$

Finally, we turn our attention to the term $B$. We have
\begin{align*}
B&= \sum_{Q\in \wt\fH_1(R')} 
 \sum_{I\in T_2:20I\cap 20 \ball_Q =\varnothing} \omega^p(B^I) \, m_{4I}g(\exe_Q^1,\cdot)\,\ell(Q)^{n-1}\\
& = \sum_{I\in T_2} \omega^p(B^I)\; \avint_{4I} \sum_{Q\in \wt\fH_1(R'):20I\cap 20\ball_Q = \varnothing}
g(\exe_Q^1,x)\,\ell(Q)^{n-1}\,dx\\
&\lesssim \sum_{I\in T_2} \omega^p(B^I)\; \avint_{4I} \sum_{Q\in \wt\fH_1(R'):20I\cap 20\ball_Q = \varnothing}\omega^x(8\ball_Q)\,dx.
\end{align*}
We claim now that, in the last sum, if one assumes that $20I\cap 20\ball_Q=\varnothing$, then $\dist(I,8\ball_Q)\geq c\,\gamma^{-1/2}\,\ell(I)$.
To check this, take $P\in\wt\eend(R')$ such that $I\in\WW_P$. Then note that
\begin{align*}
\ell(P)\leq \frac1{300}\,d_R(\zed_P)&\leq\frac1{300}\,\bigl(\dist(\zed_P,Q) + \ell(Q)\bigr)\\
&\leq \frac1{300}\,\bigl(\dist(\zed_P,I)+
\diam(I)+\dist(I,8\ball_Q) + C\ell(Q)\bigr).
\end{align*}
Using that $I\cap 2\ball_P\neq\varnothing$, $\diam(I)\leq C\gamma^{1/2}\ell(P)\ll\ell(P)$, and $\ell(Q)\leq \dist(I,8\ball_Q)$, we get
$$\ell(P) \leq\frac1{300}\,\bigl(\dist(I,8\ball_Q)+3r(\ball_P)+C\,\ell(Q)\bigr)
\leq C\,\dist(I,8\ball_Q) + \frac{12}{300}\,\ell(P),$$
which implies  that
$$\ell(I) \leq C\gamma^{1/2}\,\ell(P)\leq C\,\gamma^{1/2}\,\dist(I,8\ball_Q),$$
and yields our claim.

Taking into account that the balls $\{8\ball_Q\}_{Q\in  \wt\fH_1(R')}$ are disjoint and the H\"older continuity of $\omega^{(\cdot)}(\partial\Omega\setminus c\gamma^{-1/2} I)$, for all $x\in 4I$
we get
$$\sum_{Q\in \wt\fH_1(R'):20I\cap 20\ball_Q = \varnothing}\omega^x(8\ball_Q)\lesssim \omega^x(\partial\Omega\setminus c\gamma^{-1/2} I)
\lesssim\gamma^{\alpha/2}.$$
Thus,
$$B
\lesssim \gamma^{\alpha/2} \sum_{I\in T_2}\omega^p(B^I) \leq \gamma^{\alpha/2}\sum_{\substack{P\in \wt\eend(R'):\\2\ball_P\cap10\ball_{R'}\neq\varnothing}} 
\sum_{I\in \WW_P\cap T_2}\omega^p(B^I)
.$$
Recalling again that $\sum_{I\in \WW_P}
\omega^p(B^I)\lesssim \omega^p(C\ball_P)$, we deduce
$$B
\lesssim \gamma^{\alpha/2} \sum_{\substack{P\in \wt\eend(R'):\\2\ball_P\cap10\ball_{R'}\neq\varnothing}} \omega^p(C\ball_P)
\lesssim\gamma^{\alpha/2}\sum_{\substack{P\in \wt\eend(R'):\\2\ball_P\cap10\ball_{R'}\neq\varnothing}}\frac{\sigma(P)}{\sigma(R_0)}\lesssim \gamma^{\alpha/2}\,\frac{\sigma(R')}{\sigma(R_0)}
.$$
Remark that for the second inequality we took into account that $P$ is contained in a cube of the form $22P'$ with $P'\in\tree_\WSBC(R)$ and $\ell(P')\approx\ell(P)$, by Lemma \ref{lem74}. This implies that $\omega^p(C\ball_P)\leq
\omega^p(C'\ball_{P'})\lesssim \sigma(P')\,\sigma(R_0)^{-1}\lesssim  \sigma(P)\,\sigma(R_0)^{-1}$.

\vv

Gathering the estimates above and recalling \rf{eqi123}, we deduce
$$\lambda\tau\,\sigma(R')\lesssim \sup_{y\in 2\ball_{R'}\cap V_1:\delta_\Omega(y)\geq \delta\ell(R')} \frac{g(p,y)}{\delta_\Omega (y)} \,\sigma(R')\,\sigma(R_0)+
\delta^{\alpha/2}\,\sigma(R') + \gamma^{\alpha/2}\,\sigma(R').$$
So, if $\delta$ and $\gamma$ are small enough (depending on $\lambda,\tau$), we infer that
$$\lambda\,\tau\,\sigma(R')\lesssim \sup_{y\in 2\ball_{R'}\cap V_1:\delta_\Omega(y)\geq \delta\ell(R')} \frac{g(p,y)}{\delta_\Omega (y)} \,\sigma(R')\,\sigma(R_0).$$
That is, there exists some $y_0\in 2\ball_{R'}\cap V_1$ with $\delta_\Omega(y_0)\geq \delta\,\ell(R')$ such that
$$ \frac{g(p,y_0)}{\delta_\Omega (y)}\gtrsim \frac{\lambda\tau}{\sigma(R_0)},$$ 
with $\delta$ depending on $\lambda,\tau$. Since $\exe_{R'}^1$ and $y_0$ can be joined by a $C$-good 
Harnack chain (for some $C$ depending on $\delta$ and $\gamma$, and thus on $\lambda,\tau$), we deduce that
$$ \frac{g(p,\exe_{R'}^1)}{\ell(R')} \gtrsim \frac{c(\lambda,\tau)}{\sigma(R_0)},$$
as wished. 
\end{proof}

\vv
\begin{lemma}\label{keylemma2}
Let $\eta\in (0,1)$ and $\lambda>0$. Choose $\gamma=\gamma(\lambda,\tau)$ small enough as in Lemma \ref{keylemma} with $\tau=\eta/2$.
Assume that the family $\WSBC(\Gamma)$ is defined by choosing $\Gamma$ big enough depending on $\gamma$ (and thus on $\lambda$ and $\eta$) as in Lemma
\ref{lemgeom}.
Let $R\in\ttt^{(N)}_b$ and suppose that $\tree_\WSBC(R)\neq \varnothing$.
Then, there exists an exceptional family $\mathsf{Ex}_\WSBC(R)\subset \Stop_\WSBC(R)\cap\fG$ 
satisfying
$$\sum_{P\in\mathsf{Ex}_\WSBC(R)}\sigma(P)\leq \eta\,\sigma(R)$$
such that,
for every $Q\in \Stop_\WSBC(R)\cap\fG\setminus \mathsf{Ex}_\WSBC(R)$, 
 any $\lambda$-good corkscrew for $Q$ can be joined to some $\lambda'$-good corkscrew for $R$ by a $C(\lambda,\eta)$-good Harnack chain, with $\lambda'$ depending on $\lambda,\eta$.
\end{lemma}

\begin{proof} 
 For any $R'\in  \DD_{k_0}(R)\cap\tree_\WSBC(R)$, with $k_0=k_0(\gamma)$, we define $\fH_i(R')$ as in \rf{defhi}, so that 
$$\Stop_\WSBC(R)\cap\fG\cap \DD(R') = \fH_1(R')\cup \fH_2(R').$$
For each $R'$, we set  
$$\mathsf{Ex}_\WSBC(R') = \bigcup_{i=1}^2\,\,\Bigl\{Q\in\fH_i(R'): \mbox{$\sum_{P\in \fH_i(R')} \sigma(P)\leq \tau\,\sigma(R')$}\Bigr\}.$$
That is, for fixed $i=1$ or $2$, if $\sum_{P\in \fH_i(R')} \sigma(P)\leq \tau\,\sigma(R')$, then all the cubes from
$\fH_i(R')$
belong to $\mathsf{Ex}_\WSBC(R')$.
In this way, it is clear that
\begin{equation}\label{eq1pp}
\sum_{P\in\mathsf{Ex}_\WSBC(R')}\sigma(P)\leq 2\tau\,\sigma(R) = \eta\,\sigma(R').
\end{equation}

We claim that the $\lambda$-good corkscrews of cubes from $\Stop_\WSBC(R)\cap\fG\cap \DD(R')\setminus
\mathsf{Ex}_\WSBC(R')$ can be joined to some $\wt\lambda$-good corkscrew for $R'$ by a $\wt C$-good Harnack chain, with $\wt\lambda$ depending on $\lambda,\eta$, and $\wt C$ depending on $\Gamma$ and thus on  $\lambda,\eta$ too. 
Indeed, if $Q\in \fH_i(R')\setminus \mathsf{Ex}_\WSBC(R')$ and $\exe_Q^i$ is $\lambda$-good corkscrew 
belonging to $V_i$ (we use the notation of Lemma \ref{keylemma} and \ref{lemgeom}), then $\sum_{P\in \fH_i(R')} \sigma(P)> \tau\,\sigma(R')$ by the definition above
and thus Lemma \ref{keylemma} ensures that $g(p,\exe_{R'}^i)\geq c(\lambda,\tau)\,\frac{\ell(R')}{\sigma(R_0)}$. So $\exe_{R'}^i$
is a $\wt \lambda$-good corkscrew, which by Lemma \ref{lemgeom}(c) can be joined to $\exe_Q^i$ by a $\wt C$-good Harnack chain.
In turn, this $\wt\lambda$-good corkscrew for $R'$ can be joined to some $\lambda'$-good corkscrew for $R$ by a $C'$-good Harnack chain, by applying Lemma \ref{lembigcork}
$k_0$ times, with $C'$ depending on $k_0$ and thus on $\lambda$ and $\eta$.

On the other hand, the cubes $Q\in \Stop_\WSBC(R)\cap\fG$ which are not contained in any cube $R'\in\DD_{k_0}(R)\cap\tree_\WSBC(R)$
satisfy $\ell(Q)\geq 2^{-k_0}\ell(R)$, and then, arguing as above, their associated $\lambda$-good corkscrews can be joined
 to some $\lambda'$-good corkscrew for $R$ by a $C'$-good Harnack chain, by applying Lemma \ref{lembigcork}
at most $k_0$ times.
Hence, if we define
$$\mathsf{Ex}_\WSBC(R) = \bigcup_{R'\in\DD_{k_0}(R)}\mathsf{Ex}_\WSBC(R'),$$
taking into account \rf{eq1pp}, the lemma follows.
\end{proof}

\vv

\begin{proof}[\bf Proof of the Key Lemma \ref{keylemma3}]

We choose $\Gamma=\Gamma(\lambda,\eta)$ as in Lemma \ref{keylemma2} and we consider the associated family $\WSBC(\Gamma)$.
In case that $\tree_\WSBC(R)= \varnothing$,  we set $\mathsf{Ex}(R)=\varnothing$. Otherwise,
we consider the family  
$\mathsf{Ex}_\WSBC(R)$  from Lemma \ref{keylemma2}, and we define
$$\mathsf{Ex}(R) = 
\bigl(\mathsf{Ex}_\WSBC(R)\cap \Stop(R)\bigr)\cup\!\! 
\bigcup_{Q\in \mathsf{Ex}_\WSBC(R)\setminus \Stop(R)}\!\!\bigl(\sstop(Q)\cap\fG\bigr).$$
It may be useful for the reader to compare the definition above with the partition of $\Stop(R)$ in \rf{eqstop221}.
By Lemma \ref{keylemma2} we have
$$\sum_{P\in\mathsf{Ex}(R)}\sigma(P)\leq
\sum_{Q\in\mathsf{Ex}_\WSBC(R)}\sigma(P)\leq \eta\,\sigma(R).$$

Next we show that for every $P\in \Stop(R)\cap\mathsf G\setminus \mathsf{Ex}(R)$, 
 any $\lambda$-good corkscrew for $P$ can be joined to some $\lambda'$-good corkscrew for $R$ by a $C(\lambda,\eta)$-good Harnack chain. In fact, if $P\in \Stop_\WSBC(R)$, then $P\in \Stop_\WSBC(R)\cap\mathsf G\setminus
 \mathsf{Ex}_\WSBC(R)$ since such cube $P$ cannot belong to $\sstop(Q)$ for any $Q\in\Stop_\WSBC(R)\setminus \Stop(R)$ (recall the partition
\rf{eqstop221}),
  and thus the existence of such Harnack chain is ensured by Lemma \ref{keylemma2}.
On the other hand, if $P\not\in \Stop_\WSBC(R)$, then $P$ is contained in some cube $Q(P)\in\Stop_\WSBC(R)\setminus\WSBC(\Gamma)$. 
Consider the chain $P=S_1\subset S_2\subset\cdots\subset S_m=Q(P)$, so that
each $S_i$ is the parent of $S_{i-1}$. For  $1\leq i\leq m$, choose inductively a big corkscrew $x_i$ for
$S_i$ in such a way that
$x_1$ is at the same side of $L_P$ as the  good $\lambda$ corkscrew $\exe_P$ for $P$, and $x_{i+1}$ is at the same side of $L_{S_i}$ as $x_i$ for each $i$.
Using that $b\beta(S_i)\leq C\eps\ll1$ for all $i$, it easy to check that the line obtained by joining
the segments $[\exe_P,x_1]$, $[x_1,x_2]$,\ldots,$[x_{m-1},x_m]$ is a good carrot curve and so gives rise to a good
Harnack chain that joins $\exe_P$ to $x_m$. It may happen that $x_m$ is not a $\lambda$-good corkscrew. However, since $Q(P)\not\in \WSBC(\Gamma)$, it turns out that $x_m$ can be joined to some $c_3$-good corkscrew $\exe_{Q(P)}$
for $Q(P)$ by some $C(\Gamma)$-good Harnack chain, with $c_3$ given by \rf{eqgg1} (and thus independent of $\lambda$ and $\eta$), because $Q(P)\in\fG$. Note that since $\lambda\leq c_3$, $\exe_{Q(P)}$ is also a $\lambda$-good
corkscrew.
In turn, since $Q(P)\not\in\mathsf{Ex}_\WSBC(R)$, $\exe_{Q(P)}$ can be joined to some $\lambda'$-good corkscrew for $R$ by another $C'(\lambda,\eta)$-good Harnack chain.
Altogether, this shows that $\exe_P$ can  be connected to some $\lambda'$-good corkscrew for $R$ by a $C''(\lambda,\eta)$-good Harnack chain, which completes the proof of the lemma.
\end{proof}

\vv

Below we will write $\mathsf{Ex}(R,\lambda,\eta)$ instead of $\mathsf{Ex}(R)$ to keep track of the dependence of this family on the parameters $\lambda$ and $\eta$.

\vv

\section{Proof of the Main Lemma \ref{lemhc}}\label{sec8}

\subsection{Notation}

Recall that by the definition of $G_0^K$ in \rf{defg0k}, $\sum_{R\in\ttt}\chi_R(x) \leq K$ for all $x\in G_0^K$.
For such $x$, let $Q$ be the smallest cube from $\ttt$ that contains $x$, and denote 
$n_0(x) = \log_2\frac{\ell(R_0)}{\ell(Q)}$, so that
$Q\in\DD_{n_0(x)}(R_0)$.
Next let $N_0\in\ZZ$ be such that
$$\sigma\bigl(\bigl\{x\in G_0^K:n_0(x)\leq N_0-1\bigr\}\bigr) \geq \frac12\,\sigma(G_0^K),$$
and denote 
$$\wt G_0^K = \bigl\{x\in G_0^K:n_0(x)\leq N_0-1\bigr\}.$$

Fix $$N=N_0-1,$$ and set
$$\TT_a' = \DD_{N}(R_0)\cup \ttt^{(N)}_a,$$
and also
$$\TT_b' = \ttt^{(N)}_b\setminus\DD_{N}(R_0)$$
(recall that $\ttt^{(N)}_a$ and $\ttt^{(N)}_b$ were defined in Section \ref{secnnn}).
So if $R\in\TT_a'\setminus \DD_{N}(R_0)$, then $\Stop^{N}(R)$ coincides the family of sons of $R$, and it $R\in \TT_b'$ this will not be the case, in general. Next we denote by $\TT_a$ and $\TT_b$ the respective subfamilies of cubes from $\TT_a'$ and $\TT_b'$ which intersect $\wt G_0^K$. 

For $j\geq0$, we set
$$\TT_b^j = \Bigl\{R\in\TT_b:\sum_{Q\in\TT_b:Q\supset R} \chi_Q=j\text{ on $R$}\Bigr\}.$$
We also denote
$$\fS_b^j = \bigl\{Q\in\DD:Q\in\Stop^{N}(R)\text{ for some $R\in\TT_b^j$}\bigr\},\qquad \fS_b=\bigcup_j \fS_b^j,$$
and we let $\TT_a^j$ be the subfamily of cubes $R\in\TT_a$ such that there exists some $Q\in\fS_b^{j-1}$
such that $Q\supset R$ and $R$ is not contained in any cube from $\fS_b^{k}$ with $k\geq j$.

\vv

\subsection{Two auxiliary lemmas}

\vv

\begin{lemma}\label{lem7.1}
The following properties hold for the family $\TT_b^1$:
\begin{itemize}
\item[(a)] The cubes from $\TT_b^1$ are pairwise disjoint and cover  $\wt G_0^K$, assuming
$N_0$ big enough.
\item[(b)] If $R\in\TT_b^1$, then
$\ell(R)\approx_K\ell(R_0)$.
\item[(c)] Given $R\in\DD(R_0)$ with $\ell(R)\geq c\,\ell(R_0)$ (for example, $R\in\TT_b^1$) and
$\lambda>0$, if $\exe_R$ is a $\lambda$-good corkscrew point 
for $R$, then there is a $C(\lambda,c)$-good Harnack chain that joins $\exe_R$ to $p$.
\end{itemize}
\end{lemma}

\begin{proof} 
Concerning the statement (a), the cubes from $\TT_b^1$ are pairwise disjoint by construction. Suppose that
$x\in \wt G_0^K$ is not contained in any cube from $\TT_b^1$. By the definition of the family $\ttt^{N}$, this implies that all the cubes $Q\subset R_0$ with $2^{-N}\ell(R_0)\leq \ell(Q)\leq 2^{-10}\ell(R_0)$ containing $x$ belong to $\TT_a$. However, there are at most $K$ cubes $Q$ of this type, which is not possible
if $N$ is taken big enough. So the cubes from $\TT_b^1$ cover $\wt G_0^K$.

The proof of (b) is analogous. Given $R\in\TT_b^1$, all the cubes $Q$ which contain $R$ and
have side length smaller or equal that $2^{-10}\ell(R_0)$ belong to $\TT_a$.
Hence there at most $K-1$ cubes $Q$ of this type, because $\wt G_0^K \cap R\neq\varnothing$. Thus,
$\ell(R)\geq 2^{-K-10}\ell(R_0)$.

The statement (c) is an immediate consequence of (b) and Lemma \ref{lemclosejumps}.
\end{proof}

\vv
\begin{lemma}\label{lem7.2}
Let $Q\in\TT_a^j\cup\TT_b^j$ for some $j\geq 2$ and let $\exe_Q$ be a $\lambda$-good corkscrew for $Q$, with $\lambda>0$. There exists some constant $\gamma_0(\lambda,K)>0$ such if $\ell(Q)\leq \gamma_0(\lambda,K)\,\ell(R_0)$, then there exists some cube
$R\in\fS_b$ such that $R\supset Q$ with a $\lambda'$-good corkscrew $\exe_R$ for $R$ such that
 $\exe_R$ can be joined to $\exe_Q$ by a $C(\lambda,K)$-good Harnack chain, with $\lambda'$ depending on $\lambda$ and $K$.
\end{lemma}

\begin{proof}
We assume $\gamma_0(\lambda,K)>0$ small enough. Then we can apply Lemma \ref{lemshortjumps2} $K+1$ times
to get cubes $R_1,\ldots,R_{K+1}$ satisfying:
\begin{itemize}
\item $Q\subsetneq R_1\subsetneq R_2\subsetneq\ldots \subsetneq R_{K+1}$ and  $\ell(R_{K+1})\leq 2^{-10}\ell(R_0)$,
\item each $R_j$ has an associated $\lambda'$-good corkscrew $\exe_{R_i}$ (with $\lambda'$ depending on $\lambda,K$) and there exists a $C(\lambda,K)$-good Harnack chain joining $\exe_Q$ and $\exe_{R_1},
\ldots,\exe_{R_{K+1}}$.
\end{itemize} 
Since $Q\cap\wt G_0^K\neq\varnothing$, at least one of the cubes $R_1,\ldots,R_{K+1}$, say $R_j$, does not belong
to $\ttt$. This implies that $R_j\in\tree^{(N)}(\wt R)$ for some $\wt R\in\TT_b$. Let $R\in\Stop^{(N)}(\wt R)$
be the stopping cube that contains $Q$.  Then Lemma \ref{lemgeom} ensures that there is a good Harnack chain
that connects $\exe_{R_j}$ to some corkscrew $\exe_R$ for $R$. Notice that $\ell(R_j)\approx_{\lambda,K}\ell(Q)\approx_{\lambda,K}
\ell(R)$ because
$Q\subset R\subset R_j$. This implies that $g(p,\exe_R)\approx_{K,\lambda} g(p,\exe_{R_j})\approx_{K,\lambda} g(p,\exe_Q)$. Further, gathering the Harnack chain that joins $\exe_Q$ to $\exe_{\wt R}$ and the one
that joins $\exe_{R_j}$ to $\exe_R$, we obtain the good Harnack chain required by the lemma.
\end{proof}

\vv

\subsection{The algorithm to construct good Harnack chains}

We will construct good Harnack chains that join good corkscrews from ``most'' cubes from $\DD_{N}(R_0)$
that intersect $\wt G_0^K$
to good corkscrews from cubes belonging to $R\in\TT_b^1$, and then we will join these latter good corkscrews to
$p$ using the fact that $\ell(R)\approx\ell(R_0)$.
To this end we choose $\eta>0$ such that
$$\eta \leq \frac1{2K}\,\frac{\sigma(\wt G_0^K)}{\sigma(R_0)},$$
and we
denote
$$m=\max_{x\in \wt G_0^K}\sum_{R\in\TT_b}\chi_R(x)$$
(so that $m\leq K$) and we apply the following algorithm: we set $a_{m+1}=c_3$, so that \rf{eqgg1} ensures that 
for each $Q\in\TT_a\cup\TT_b$ there exists some good $a_{m+1}$-good corkscrew $\exe_Q$. 
For $j=m,m-1,\ldots,1$, we perform the following procedure:
\vv

\noindent\fbox{\begin{minipage}{12cm}
\vv
\begin{list}{(\theenumi)}{\usecounter{enumi}\leftmargin=.8cm
		\labelwidth=.8cm\itemsep=0.3cm\topsep=.15cm
		\renewcommand{\theenumi}{\arabic{enumi}}}

\item Join $a_{j+1}$-good corkscrews of cubes $Q$ from $\TT_a^{j+1}\cup \TT_b^{j+1}$ such that $\ell(Q)\leq c_j'\,\ell(R_0)$ to $a_j'$-good corkscrews of cubes $R(Q)$ 
from $\fS_b^1\cup\ldots\cup \fS_b^j$ by $C_j'$-good Harnack chains, with $a_j'\leq a_{j+1}$, so that $R(Q)$ is an ancestor of $Q$. This step can be performed because
of Lemma \ref{lem7.2}, with $c_j'=\gamma_0(a_{j+1},K)$ in the lemma. The constants $a_j'$, $c_j'$, and $C_j'$ depend on $a_{j+1}$
and $K$.

\item Set 
$$\NC_j = \bigcup_{R\in\TT_b^j} \mathsf{Ex}(R,a_j',\eta),$$
and join $a_j'$-good corkscrews for all cubes $Q\in\fS_b^j\setminus \NC_j$ to $a_j$-good corkscrews for cubes $R(Q)\in\TT_b^j$ by $C_j$-good Harnack chains, with $a_j\leq a_j'$, so that $R(Q)$ is an ancestor of $Q$.  To this end, one applies Lemma \ref{keylemma3}, which ensures the existence
of such Harnack chains connecting $a_j'$-good corkscrew points for cubes from $\fS_b^j \setminus \NC_j$ to
$a_j$-good corkscrew points for cubes from $\TT_b^j$. The constants $a_j$ and $C_j$ depend on $a_{j}'$ and $K$.
\\
\end{list}
\end{minipage}
\rule{.4cm}{0pt}
\vv}
\vspace{0.5cm}

After iterating the procedure above for $j=m,m-1\ldots,1$ and joining some Harnack chains arisen in the different iterations, we will have constructed $C$-good Harnack chains that join  
$a_{m+1}$-good corkscrew points for all cubes $Q\in\TT_a$ not contained in $\bigcup_{j=1}^m\bigcup_{P\in\NC_j}P$
to $a_{1}$-good corkscrews of some ancestors $R(Q)$ belonging either $\TT_b^1$ or, more generally, such that $\ell(R(Q))\gtrsim \ell(R_0)$. The constants $c_j'$, $a_j'$, $a_j$, $C_j$ worsen at
each step $j$. However, this is not harmful because the number of iterations of the procedure is at most $m$, and $m\leq K$.

Denote by $I_N$ the cubes from   $\DD_{N}(R_0)$ which intersect $\wt G_0^K$ and are not contained in any cube from 
$\{P\in\NC_j:j=1,\ldots m\}$.
By the algorithm above we have constructed good Harnack chains that join $a_{m+1}$-good corkscrew points for all cubes $Q\in I_N$  
 to some $a_1$-good corkscrew for cubes $R(Q)\in\DD(R_0)$ with $\ell(R(Q))\approx \ell(R_0)$. Also, by applying Lemma \ref{lem7.1} (c) we can connect the $a_1$-good corkscrew for $R(Q)$ to $p$ by a good Harnack chain.
 
Consider now an arbitrary point $x\in \wt G_0^K\cap Q$, with $Q\in I_N$.
By the definition of $\wt G_0^K$ and the choice $N=N_0$, all the cubes $P\in\DD$ containing $x$ with side length smaller
or equal than $\ell(Q)$ satisfy $b\beta(P)\leq \eps$. Then, by an easy geometric argument (see the proof of Lemma \ref{keylemma3} for a related one) it is easy to check that there is a good Harnack chain joining any good corkscrew for $Q$ to $x$. Hence, for all the points $x\in \bigcup_{Q\in I_N}Q\cap \wt G_0^K$
there is a good Harnack chain that joins $x$ to $p$.

Finally, observe that, for each $j$, by Lemma \ref{keylemma3},
$$\sum_{P\in\NC_j}\sigma(P) = \sum_{R\in\TT_b^j}\sum_{P\in\mathsf{Ex}(R,a_j',\eta)}\sigma(P) \leq
\eta \sum_{R\in\TT_b^j}\sigma(R) \leq \eta\,\sigma(R_0) \leq  \frac1{2K}\,\sigma(\wt G_0^K).$$
Therefore,
$$\sum_{j=1}^m\sum_{P\in\NC_j}\sigma(P)\leq \frac m{2K}\,\sigma(\wt G_0^K)\leq \frac12\,\sigma(\wt G_0^K),$$
and thus
$$\sum_{Q\in I_N} \sigma(Q)\geq \sigma(\wt G_0^K) - \sum_{j=1}^m\sum_{P\in\NC_j}\sigma(P) \geq \frac12\,\sigma(\wt G_0^K) \approx \sigma(R_0).$$
This finishes the proof of the Main Lemma \ref{lemhc}.
\qed

\vv


\appendix




\section{Some counter-examples} \label{appa}

We shall discuss some counter-examples which show that
our background hypotheses  in Theorem \ref{tmain} (namely,
$n$-ADR and interior corkscrew condition)
are natural, and in some sense in the nature of best possible.  In the first two examples, 
$\Omega$ is a domain satisfying an interior corkscrew condition, such that $\pom$
satisfies exactly one
(but not both) of the upper or the lower $n$-ADR bounds, and for which harmonic measure $\hm$ 
fails to be weak-$A_\infty$ with respect to surface measure $\sigma$ on $\pom$.
In this setting, in which full $n$-ADR
fails, there is no established notion of uniform rectifiability, but in each case, the domain
will enjoy some substitute property which would imply uniform rectifiability of the boundary in the 
presence of full $n$-ADR.  Moreover, these examples may be constructed in such a way that the 
 failure of the condition (either upper or lower $n$-ADR) can be expressed quantitatively, with a bound that may be 
 taken arbitrarily close to a true $n$-ADR bound; see \eqref{upperadrfail} and \eqref{loweradrfail}
 below.

In the last example, we construct an open set $\Omega$ with $n$-ADR boundary, and for which
$\hm\in$ weak-$A_\infty$ with respect to surface measure, but for which the interior
corkscrew condition fails, and $\pom$ is not $n$-UR.

\smallskip

\noindent{\bf Example 1}. {\it Failure of the upper $n$-ADR bound.}
In \cite{AMT1}, the authors construct an example of a Reifenberg flat domain $\Omega \subset \ree$ 
for which surface measure $\sigma= H^n\lfloor_{\,\pom}$  is locally finite on $\pom$, but for 
which the
upper $n$-ADR bound
\begin{equation}\label{a1}
\sigma(\Delta(x,r)) \leq C r^n
\end{equation}
fails, and for which harmonic measure $\hm$ is not absolutely continuous with respect to $\sigma$.  Note that
the hypothesis of Reifenberg flatness implies in particular that $\Omega$ and 
$\Omega_{ext}:= \ree\setminus \overline{\Omega}$ are both NTA domains, hence both enjoy the
corkscrew condition, so by the relative isoperimetric inequality, the lower $n$-ADR bound
\begin{equation}\label{a2}
\sigma(\Delta(x,r)) \geq c r^n
\end{equation} holds.  
Thus, it is the failure of \eqref{a1} which causes the failure of absolute continuity: in the presence of
\eqref{a1}, the results of \cite{DJe} apply, and one has that $\hm \in A_\infty(\sigma)$, and that $\pom$ satisfies
a ``big pieces of Lipschitz graphs" condition (see \cite{DJe} for a precise statement), and hence is
$n$-UR.  We note that by a result of Badger \cite{Ba}, a version of the Lipschitz approximation
result of \cite{DJe} still holds for NTA domains with locally finite surface measure, even in the absence of the upper
$n$-ADR condition.

In addition, given any $\eps>0$, the construction in \cite{AMT1} can be made in such a way that
\eqref{a1} fails ``within $\eps$", i.e., so that
\begin{equation}\label{upperadrfail}
\sigma(\Delta(x,r) \leq C r^{n-\eps}\,, \quad \forall \, x\in\pom,\, r<1.
\end{equation}
Let us sketch an argument to explain why this is so; we refer the interested reader to 
\cite{AMT1} for more details.

The domain $\Omega$ in \cite{AMT1} is obtained by enlarging a Wolff snowflake, 
that we will denote here by $D$.  Both $\Omega$ and $D$ are $\delta$-Reifenberg flat, with $\delta$ 
as small as wished in the construction (recall that Wolff snowflakes can be taken 
$\delta$-Reifenberg flat, with $\delta$ as small as wished).

It is shown in \cite[Theorem 3.1]{AMT1} that for all $x\in \partial\Omega$ and $r<1$,
\begin{equation}\label{**}
H^n(B(x,r)\cap \partial\Omega) \lesssim \max( r^n , r^\alpha \mu(B(x,Cr)) ) \leq  \max( r^n ,  \mu(B(x,Cr)) )
\end{equation}
where  $\mu$ is some measure supported on $\partial D$ satisfying $\mu(B(x,r)) \gtrsim r^{n-\alpha}$ for all 
$x$ in some compact set $E \subset \partial\Omega \cap \partial D$, and some $\alpha>0.$
In the construction in \cite{AMT1} , the authors take $\mu=\hm_D$,  the harmonic measure for $D$. 
Further, from results of Kenig and Toro it follows that harmonic measure in a  $\delta$-Reifenberg flat domain $D$ satisfies
\[ \omega_D(B(x,r)) \lesssim r^{n-\varepsilon} \omega_D(B(x,1)),   \quad \forall\, x\in \partial D,
\, r<1,
\]
with $\varepsilon \to 0$ as $\delta\to 0$
  (see \cite[Theorem 4.1]{Kenig-Toro-duke}).
As a consequence, the measure $\mu$ satisfies 
\[ \mu(B(x,r)) \lesssim  r^{n-\varepsilon}, \quad \forall\, x\in \RR^{n+1},\, r<1,\]
with $\varepsilon$ as small as wished depending on $\delta$. From \eqref{**}, it follows that
\[ H^n(B(x,r)\cap \partial\Omega) \lesssim \max( r^n , r^{n-\varepsilon} ) \leq r^{n-\varepsilon}\,, \quad 
\forall\, x\in \partial\Omega \,, r<1.
\]

\smallskip

\noindent{\bf Example 2}. {\it Failure of the lower $n$-ADR bound.} In \cite[Example 5.5]{ABoHM}, the authors give an 
example of a domain satisfying
the interior corkscrew condition, whose boundary is rectifiable  (indeed, 
it is contained in a countable
union of hyperplanes), and 
satisfies the upper $n$-ADR condition \eqref{a1}, but 
not the lower $n$-ADR condition \eqref{a2},  but for which
surface measure $\sigma$ fails to be absolutely continuous with respect to harmonic measure, and in fact,
for which the non-degeneracy condition
\begin{equation}\label{a3}
A\subset \Delta_x:= B(x,10\delta_\Omega(x)) \cap\pom,\quad \sigma(A)\geq (1-\eta) \sigma(\Delta_x)  \, \, 
\implies \,\,  \hm^x(A) \geq c\,,
\end{equation}
fails to hold uniformly for $x\in\Omega$, for any fixed positive $\eta$ and $c$, and therefore $\hm$ cannot be weak-$A_\infty$ with respect to
$\sigma$.  We note that in the presence of the full $n$-ADR condition,
if $\pom$ were contained in a countable union of hyperplanes (as it is in the example), 
then in particular it would satisfy
the ``BAUP" condition of \cite{DS2}, and thus would be $n$-UR \cite[Theorem I.2.18, p. 36]{DS2}.

Moreover, given any $\eps>0$, the parameters in the example of \cite{ABoHM} can be chosen in such a way that
the lower ADR bound fails ``within $\eps$", i.e., so that 
\begin{equation} \label{loweradrfail}
H^n(\Delta(x,r))\gtrsim \min(r^{n+\eps},r^n)\,, \quad \forall x\in \pom.
\end{equation}
To see this, we proceed as follows.  
We follow closely the construction in 
 \cite[Example 5.5]{ABoHM}, with some modification of the parameters.
Fix $\eps>0$, and set
 \[ c_k:= 2^{-k(n+\eps)}\,.\]
 For $k\geq 1$, and $n\geq 2$, set
$$\Sigma_k:= \{(x,t)\in \reu:\, t=2^{-k},\, x\in \overline{\Delta(0,2^{-\eps k}c_k)}+ c_k \ZZ^n\} \,,$$
where for $x\in \rn$, $\Delta(x,r):= \{y\in\rn:\, |x-y|<r\}$ is the 
usual $n$-disk of radius $r$ centered at $x$.   Define
$$\Omega:= \reu\setminus \left(\cup_{k=1}^\infty\Sigma_k \right),\qquad \Omega_k:= 
\reu\setminus \Sigma_k  \,,$$
each of which is clearly open and connected. Notice that $\Omega$ satisfies the interior
Corkscrew condition (since the sets $\Sigma_k$ are located at heights which are sufficently separated).
Moreover, it is easy to see that $\pom$ satisfies the upper ADR condition  and that $\re^n\times\{0\}\subset\pom$. 
 
On the other hand, the lower ADR bound fails. To see this, let $X=(x,0)\in \pom$, 
and choose $\vec{m}_{k,x}\in\ZZ^n$ and $X_k=(c_k\,\vec{m}_{k,x}, 2^{-k})\in \Sigma_k\subset \pom$ such that  
 $X_k\to X$.  Set $B_k=B(X_k, 2^{-k-2})$, and observe that 
 $H^n(B_k\cap \pom)/(2^{-kn})\approx 2^{-kn\eps}\to 0$ as $k\to \infty$, 
 or equivalently 
 \[ H^n(B_k\cap \pom)\approx r_k^{n+\eps'}\,,\]
where $B_k$ has radius $r_k \approx 2^{-k}$, and $\eps'=n\eps$.
We shall show that this behavior is in fact typical, and that \eqref{loweradrfail} holds, with $\eps'$ in place of $\eps$.

Let $\hm^{(\cdot)}:= \hm^{(\cdot)}_\om$ and $\hm_k^{(\cdot)}:= \hm^{(\cdot)}_{\om_k}$ denote
harmonic measure for the domains $\om$ and $\om_k$ respectively.

\smallskip

\noindent{\bf Claim}.  $\hm^{(\cdot)}(F) = 0$, with $F:=\re^n\times\{0\}$.  Thus, in particular
\eqref{a3} fails.

It remains to verify \eqref{loweradrfail}, and the claim.  As regards the former,
note that for $X=(x,0)\in F$, we have the trivial standard lower $n$-ADR bound $H^n(\Delta(X,r))\gtrsim r^n$, whereas for
$X=(x,2^{-k}) \in \Sigma_k$, we have
\begin{equation}\label{eq.Hnbound}
H^n\lfloor_{\pom}\big(B(X,r)\big)\geq H^n\lfloor_{\pom_k}\big(B(X,r)\big)\gtrsim\left\{
\begin{array}{ll}
r^n\,, & r<2^{-\eps k} c_k \,,
\\[4pt]
2^{-kn\eps}c_k^{n}\,, & 2^{-\eps k} c_k \leq r\leq  c_k
\\[4pt]
2^{-kn\eps}r^n \,, &  c_k< r\leq 2^{-k+1}\\[4pt]
r^n \,,& r> 2^{-k+1}
\,.
\end{array}
\right.
\end{equation}
The first and fourth of these estimates are of course the standard lower $n$-ADR bound.  
For $r\leq c_k$, the second estimate is bounded below by $2^{-kn\eps}r^n$, and in turn,
with $r\lesssim 2^{-k}$, the second and third estimates are therefore bounded below by
\[2^{-kn\eps}r^n \gtrsim r^{n+n\eps}= r^{n+\eps'}\,,\]
which yields  \eqref{loweradrfail} with $\eps'=n\eps$ in place of $\eps$.

Let us now prove the claim.
We first recall some definitions. Given an open set 
$O\subset \ree$, and a compact set $K\subset O$, 
we define the capacity of  $K$ relative to $O$ as
$$
\mbox{cap}(K,O)=\inf\left\{\iint_O|\nabla\phi|^2\, dY:\ \phi\in C_0^\infty(O),
\ 
\phi\geq 1\mbox{ in }K\right\}.
$$
Also, the inhomogeneous capacity of $K$ is defined as 
$$
\mbox{Cap} (K)=\inf\left\{\iint_{\ree} \big(|\phi|^2+|\nabla\phi|^2\big)\, dY:\ \phi\in C_0^\infty(\re),
\ 
\phi\geq 1\mbox{ in }K\right\}.
$$
Combining \cite[Theorem 2.38]{HKM}, \cite[Theorem
 2.2.7]{AH} and \cite[Theorem 4.5.2] {AH} we have that if $K$ is a compact subset of $\overline{B}$, where $B$ is a ball with radius smaller than $1$, then
\begin{equation}\label{cap-dual}
\mbox{cap}(K,2B)
\gtrsim
\mbox{Cap}(K)
\gtrsim
\sup_{\mu} \mu(K) 
\end{equation}
where the implicit constants depend only on $n$,
the sup runs over all Radon positive measures $\mu$ supported on $K$, for which
$$
W(\mu)(X):=\int_0^1 \frac{\mu(B(X,t))}{t^{n-1}}\,\frac{dt}{t} \leq 1,
\qquad
\forall\, X\in\supp \mu.
$$

Fix $k\geq 2$, and set 
\[\beta=\beta_k:= 2^{k(n-1)}c_k= 2^{k(n-1)} 2^{-k(n+\eps)} = 2^{-k(1+\eps)},\]
by definition of $c_k$.   
Our next goal is to show that 
\begin{equation}\label{eq4.capbound}
\cp\big(\overline{B(X_0,s)}\cap\Sigma_k, B(X_0,2s)\big) \gtrsim s^{n-1}\,,\quad X_0:=(x_0,2^{-k})\in\Sigma_k,\ 
\beta\le s<1\,.
\end{equation}
For a fixed $X_0$ and $s$, write $K=\overline{B(X_0,s)}\cap\Sigma_k$, set $\mu = 2^{kn\eps} s^{-1} \, H^n\lfloor_K\,,$
and note that for $X\in K$, similarly to \eqref{eq.Hnbound}, we have
\begin{equation}\label{eq4.mubound}
\mu\big(B(X,r)\big)\approx 2^{kn\eps} s^{-1} \left\{
\begin{array}{ll}
r^n\,, & r<2^{-\eps k} c_k \,,
\\[4pt]
2^{-kn\eps}c_k^{n} \,,&2^{-\eps k} c_k \leq r\leq  c_k
\\[4pt]
2^{-kn\eps}r^n \,,&  c_k< r\leq s\\[4pt]
2^{-kn\eps }s^n \,, & r> s
\,.
\end{array}
\right.
\end{equation}
To compute $W(\mu)(X)$ for $X\in K$ write
\[
W(\mu)(X) = \int_0^1 \frac{\mu(B(X,t))}{t^{n-1}}\,\frac{dt}{t}= \int_0^{2^{-\eps k}c_k}  +\, \int_{2^{-\eps k}c_k}^{c_k}
+ \, \int_{c_k}^s +\, \int_s^1\, =: \,I + II+III+IV\,.
\]
Then, since $s\geq\beta = 2^{k(n-1)} c_k = 2^{-k(1+\eps)}$,
$$I + II \lesssim 2^{kn\eps} s^{-1} \left(2^{-\eps k}c_k \, + 2^{-kn\eps }c_k^{n}\int_{2^{-\eps k}c_k}^\infty\frac{dt}{t^n} \right)
\lesssim\, 2^{\eps  k(n-1)} c_k s^{-1} \lesssim 1\,.$$
Furthermore, the last two estimates in \eqref{eq4.mubound} easily imply that $III+IV\lesssim 1$ and hence $W(\mu)(X)\lesssim 1$ for every $X\in K$. This, \eqref{cap-dual}, and \eqref{eq4.mubound} imply
as desired \eqref{eq4.capbound}:
\[
\cp\big(\overline{B(X_0,s)}\cap\Sigma_k, B(X_0,2s)\big) \gtrsim \mu(K)
\gtrsim
s^{n-1}.
\]

Set 
\[
P_k:= \left\{\left(x,2^{-k}-\beta\right)\in\reu:\, x\in\rn\right\}\,,
\]
and observe that for $X\in P_k$,
$$
\beta\le \delta_k(X):= \dist(X,\pom_k) = \dist(X,\Sigma_k)\le 2\beta\,.$$
Recall that $F=\re^n\times\{0\}$, and define
$$
u(X):= \hm_k^X(F) \,,\qquad X\in \om_k\,.
$$
Observe that $u\in W^{1,2}(\Omega_k)\cap C(\overline{\Omega_k})$ since $\pom_k$ is ADR (constants depend on $k$ but we just use this qualitatively) and $\chi_F$ is a Lipschitz function on $\pom_k$. Fix $Z_0\in P_k$ and let $Z_0'\in \Sigma_k$ be such that  $|Z_0-Z_0'|=\dist(Z_0, \pom_k)\le 2\beta$.  Let $\Omega_{Z_0}=
\Omega_k\cap B(Z_0', \frac34 2^{-k})$, which is an open connected bounded set. 
We can now apply the usual capacitary estimates
(see, e.g., \cite[Theorem 6.18]{HKM}) to find a constant  $\alpha=\alpha(n)>0$  such that 
$$
u(Z_0) 
\lesssim 
\exp\left(-\alpha \int_{3\,\beta}^{2^{-k-2}}  \frac{ds}{s} \right)  
\approx \big(2^k \beta\big)^\alpha =2^{-\alpha \eps k}.
$$
where we have used \eqref{eq4.capbound}, the definition of $\beta$,
and the fact that $u\equiv 0$ on $\pom_k \cap B(Z_0',2^{-k-1})$.
Note that the last estimate holds for any $Z_0\in P_k$ and therefore, by the maximum principle,
$$u(x,t) \lesssim 2^{-\alpha \eps k}\,, \qquad (x,t) \in \om_k\,, \,\, t > 2^{-k}-\beta\,.$$
In particular, if we set $X_0:= (0,\dots,0,1) \in \reu$, then by another application
of the maximum principle,
$$\hm^{X_0}(F)\, \leq \, \hm_k^{X_0}(F)=
u(X_0) \, \lesssim 2^{-\alpha \eps k}\, \to 0\,,$$
as $k \to \infty$, and the claim is established.

\smallskip

\noindent{\bf Example 3}. {\it Failure of the interior corkscrew condition}. The example is based on the construction of Garnett's 4-corners
Cantor set $\mathcal{C}\subset\re^2$
(see, e.g., \cite[Chapter 1]{DS2}).   Let $I_0$ be a unit square
positioned with lower left corner
at the origin in the plane, and in general for each $k = 0, 1,2,\dots$, we let $I_k$ be the unit square positioned with lower left corner at the point $(2k,0)$ on the $x$-axis.
Set $\Omega_0:= I_0$.
Let $\Omega_1$ be the first stage of the 4-corners construction, i.e., a union of four squares
of side length 1/4, positioned in the corners of the unit square $I_1$, and similarly, for each $k$, let
$\Omega_k$ be the $k$-th stage of the 4-corners construction, positioned inside $I_k$.  Note that
$\dist(\Omega_k, \Omega_{k+1}) = 1$ for every $k$.  Set $\Omega := \cup_k \Omega_k$. It is easy to check that 
$\pom$ is $n$-ADR, and that the non-degeneracy condition
\eqref{a3} holds in $\Omega$ for some uniform positive $\eta$ and $c$, and thus by the criterion of \cite{BL},
$\hm\in$ weak-$A_\infty(\sigma)$.  On the other hand, the interior corkscrew condition clearly fails to hold
in $\Omega$ (it holds only for decreasingly small scales as $k$ increases), and certainly $\pom$ cannot
be $n$-UR:  indeed, if it were, then $\pom_k$ would be $n$-UR, with uniform constants, for each $k$,
and this would imply that $\mathcal{C}$ itself was $n$-UR, whereas in fact, as is well known, it is totally non-rectifiable. One can produce a similar set in 3 dimensions by simply taking the cylinder $\Omega'=\Omega\times[0,1]$. Details are left to the interested reader.


\begin{thebibliography}{AHMTT1}
\parskip=0.1cm
 
 \bibitem[AHe]{AHe}
D. Adams and L. Hedberg.
\newblock {\em Function spaces and potential theory}, volume 314 of {\em
  Grundlehren der Mathematischen Wissenschaften [Fundamental Principles of
  Mathematical Sciences]} (corrected second printing).
\newblock Springer-Verlag, Berlin, 1999.

\bibitem[AH]{AH} 
H. Aikawa and K. Hirata. {\em Doubling conditions for harmonic measure in John domains.} Ann. Inst. Fourier (Grenoble) 58 (2008), no. 2, 429--445. 

\bibitem[ABHM]{ABoHM}  M. Akman, S. Bortz, S. Hofmann, and J. M. Martell. {\em Rectifiability, interior approximation and Harmonic Measure.} 
Ark. Mat. {57} (2019), no. 1, 1--22. 


\bibitem[ACF]{ACF}
H.W. Alt, L.A. Caffarelli, and A.~Friedman. \emph{Variational problems with two
  phases and their free boundaries}. Trans. Amer. Math. Soc. 282 (1984), no.~2, 431--461. 





\bibitem[AHLT]{AHLT} P. Auscher, S. Hofmann, J.L. Lewis, and P. Tchamitchian. {\em  Extrapolation of Carleson measures and the analyticity of Kato's square-root operators.}  {Acta Math.}     {187}  (2001),  no. 2, 161--190.




\bibitem[AHMTT]{AHMTT} P. Auscher, S. Hofmann, C. Muscalu, T. Tao, and C. Thiele. {\em 
Carleson measures, trees, extrapolation, and $T(b)$ theorems.}  {Publ. Mat.} { 46}
(2002),  no. 2, 257--325.




\bibitem[Azz]{Az} J. Azzam. {\em  Semi-uniform domains and a characterization of the $A_\infty$ property for harmonic measure.}
Preprint, arXiv:1711.03088. 

\bibitem[AGMT]{AGMT} J. Azzam, J. Garnett, M. Mourgoglou, and X. Tolsa. {\em  Uniform rectifiability, elliptic measure, square functions, and $\varepsilon$-approximability via an ACF monotonicity formula.} Preprint arXiv:1612.02650 (2016).



\bibitem[AMT1]{AMT1} J.Azzam, M. Mourgoglou, and X. Tolsa. {\em  Singular sets for harmonic measure
on locally flat domains with locally finite surface measure.}
{  Int. Math. Res. Not. IMRN} { 2017} (2017), no. 12, 3751--3773.

\bibitem[AMT2]{AMT} J. Azzam, M. Mourgoglou, and
X. Tolsa. {\em 
Harmonic measure and quantitative connectivity: geometric characterization of the $L^p$-solvability of the Dirichlet problem. Part II.} Preprint
{\it arXiv:1803.07975}.

\bibitem[Bad]{Ba} M. Badger. {\em  Null sets of harmonic measure on NTA domains:
Lipschitz approximation revisited.} {Math. Z.} { 270} (2012), no. 1-2, 241--262.


 \bibitem[BL]{BL} B.~Bennewitz and J.~L.~Lewis. {\em On weak reverse H\"older inequalities for nondoubling harmonic measures.} Complex Var. Theory Appl. 49 (2004), no.7--9, 571--582.



\bibitem[BJ]{BiJo} C. Bishop and P. Jones. {\em  Harmonic measure and arclength.}
{Ann. of Math. (2)} { 132} (1990), 511--547.



\bibitem[Bou]{B} J. Bourgain. {\em  On the Hausdorff dimension of harmonic measure in higher dimensions.}
{Invent. Math.} { 87} (1987), 477--483.

\bibitem[CS]{CS} L. Caffarelli and S. Salsa. \emph{A geometric approach to free boundary problems.} Graduate
Texts in Math. 64. Amer. Math. Soc. (2005).


\bibitem[Car]{Car} L. Carleson. {\em  Interpolation by bounded analytic functions and the corona problem.}
{Ann. of Math. (2)} { 76} (1962), 547--559.

\bibitem[CG]{CG} L. Carleson and J. Garnett. {\em  Interpolating sequences and separation properties.}
{J. Analyse Math.} { 28} (1975), 273--299.

\bibitem[Chr]{Ch} M. Christ. {\em   A $T(b)$ theorem with remarks on analytic
capacity and the Cauchy integral.} { Colloq. Math.}, { LX/LXI} (1990), 601--628.


\bibitem[Dah]{Dah} B. Dahlberg. {\em  
On estimates for harmonic measure.}
{Arch. Rat. Mech. Analysis} { 65} (1977), 272--288

\bibitem[Da1]{David88}  G.\,David. {\em  Morceaux de graphes lipschitziens et int\'egrales singuli\`eres sur une surface.} (French) [Pieces of Lipschitz graphs and singular integrals on a surface] { Rev. Mat. Iberoamericana} {4} (1988), no. 1, 73--114.

\bibitem[Da2]{David91} G.\,David. {\em  
Wavelets and singular integrals on curves and surfaces. }
{ Lecture Notes in Mathematics,} {1465}. Springer-Verlag, Berlin, 1991.

\bibitem[DJ]{DJe} G. David and D. Jerison. {\em  Lipschitz approximation
to hypersurfaces, harmonic measure,
and singular integrals.} {Indiana Univ. Math. J.} { 39} (1990),
no. 3, 831--845.

\bibitem[DS1]{DS1} G. David and S. Semmes. {\em 
Singular integrals and rectifiable sets in $\re^n$: Beyond Lipschitz graphs.} {Asterisque} { 193} (1991).

\bibitem[DS2]{DS2} G. David and S. Semmes. {\it Analysis of and on
Uniformly
Rectifiable Sets}. Mathematical Monographs and Surveys { 38}, AMS
1993.

\bibitem[EG]{EG} L.C. Evans and R.F. Gariepy. {\it Measure Theory and Fine
Properties of Functions}. Studies in Advanced Mathematics, CRC Press,
Boca Raton, FL, 1992.


\bibitem[HKM]{HKM} J. Heinonen, T. Kilpel{\"a}inen, and O. Martio. {\it Nonlinear potential theory of degenerate 
 elliptic equations}. Dover (Rev. ed.), 2006.

\bibitem[H]{H}  S. Hofmann. {\em  Quantitative absolute continuity of harmonic measure 
	and the Dirichlet problem:  a survey of recent progress.} To appear in
	special volume in honor of the 65th birthday of Carlos Kenig.

\bibitem[HLe]{HLe}  S. Hofmann and P. Le. {\it 
BMO solvability and absolute continuity of harmonic measure} {J. Geom. Anal.} 28 (2018), no. 4, 3278--3299.


\bibitem[HLMN]{HLMN}  S. Hofmann, P. Le, J. M. Martell, and K. Nystr\"om. {\em 
The weak-$A_\infty$ property of harmonic and $p$-harmonic measures
implies uniform rectifiability.} {Anal. PDE.}  
{ 10} (2017), no. 3, 513--558. 

\bibitem[HLw]{HL}  S. Hofmann and J.L. Lewis. {\em  The Dirichlet problem for parabolic operators
with singular drift terms.} {Mem. Amer. Math. Soc.} {151} (2001), no. 719.

\bibitem[HM1]{HM-TAMS} S. Hofmann and J.M. Martell. {\it $A_\infty$ estimates via extrapolation of Carleson measures
and applications to divergence form elliptic operators.} {Trans. Amer. Math. Soc.} {364} (2012), no. 1, 65--101

\bibitem[HM2]{HM-I} 
S. Hofmann and J.M. Martell. {\em  Uniform rectifiability and harmonic measure I: Uniform rectifiability implies Poisson kernels in $L^p$.}
Ann. Sci. \'Ecole Norm. Sup. { 47} (2014), no. 3, 577--654.


\bibitem[HM3]{HM-4} S. Hofmann and J.M. Martell. {\em  Uniform Rectifiability and harmonic measure IV: Ahlfors regularity plus Poisson kernels in $L^p$ implies uniform rectifiability.} Preprint,  {\it  arXiv:1505.06499}.

\bibitem[HM4]{HM4} S. Hofmann and J.M. Martell. {\em  
A sufficient geometric criterion for quantitative absolute continuity of harmonic 
measure.} Preprint {\it arXiv:1712.03696v1}.

\bibitem[HM5]{HM3} S. Hofmann and J.M. Martell. {\em Harmonic measure and 
quantitative connectivity: geometric characterization of the $L^p$ solvability of the 
Dirichlet problem. Part I.}  Preprint arXiv:1712.03696v3 (2018).

\bibitem[HMM]{HMM} 
 S. Hofmann, J.M. Martell, and S. Mayboroda. {\em 
Uniform rectifiability, Carleson measure estimates,
and approximation of harmonic functions.} { Duke Math. J.} 165 (2016), no. 12, 2331--2389.

\bibitem[HMMM]{HMMM} 
S. Hofmann, D. Mitrea, M. Mitrea, and A.J. Morris.
{\em  $L^p$-square function estimates on spaces of homogeneous type and on
  uniformly rectifiable sets.} {Mem. Amer. Math. Soc.} {245} (2017), no. 1159.


\bibitem[HMT]{HMT} S. Hofmann, M. Mitrea, and M. Taylor. {\em 
Singular integrals and elliptic boundary problems on regular
Semmes-Kenig-Toro domains.}  {Int. Math. Res. Not. IMRN}
{ 2010} (2010), 2567-2865.


\bibitem[HK]{HK} T. Hyt\"onen and A. Kairema. {\em Systems of dyadic cubes in a doubling metric space}.
Colloquium Mathematicum 126 (2012), no. 1, 1--33.


\bibitem[JK]{JK} D. Jerison and C. Kenig. {\em  Boundary behavior of
harmonic functions in nontangentially accessible domains.} {Adv. in Math.}
{46} (1982), no. 1, 80--147.

 \bibitem[KT]{Kenig-Toro-duke} C.E. Kenig and T.~Toro. \emph{Harmonic measure 
on locally flat domains}. Duke Math. J. {87} (1997), no.~3, 509Ð551.


\bibitem[Lav]{Lav} M. Lavrentiev. {\em Boundary problems in the theory of univalent functions,} (Russian).
{Math Sb.} { 43} (1936), 815-846;  { AMS Transl. Series 2} { 32} (1963), 1--35.

\bibitem[LM]{LM} J. Lewis and M. Murray. {\em  The method of layer potentials for the heat equation in time-varying domains.} {Mem. Amer. Math. Soc.} {114} (1995), no. 545.

\bibitem[MMV]{MMV} P. Mattila, M. Melnikov, and J. Verdera. {\em  The Cauchy integral, analytic capacity, and uniform rectifiability.}
{Ann. of Math. (2)} {144} (1996), no. 1, 127--136.

\bibitem [MT]{MT} M. Mourgoglou and X. Tolsa. {\em 
Harmonic measure and Riesz transform in uniform and general domains.}  {J. Reine Angew. Math.} {758} (2020), 183--221.

\bibitem[NTV]{NToV} F. Nazarov, X. Tolsa, and A. Volberg. {\em  On the uniform rectifiability of ad-regular measures with bounded Riesz
transform operator: The case of codimension 1.} {Acta Math.} 
{ 213} (2014), no. 2, 237--321.


\bibitem[PSU]{PSU} A. Petrosyan, H. Shahgholian and N. Uraltseva. {\em Regularity of free boundaries in obstacle-type problems}.  
Regularity of free boundaries in obstacle-type problems.
Graduate Studies in Mathematics, 136. American Mathematical Society, Providence, RI, 2012. 




\bibitem[RR]{RR} F. and M. Riesz. {\em  \"Uber die randwerte einer analtischen funktion.}
{Compte Rendues du Quatri\`eme Congr\`es des Math\'ematiciens Scandinaves}, Stockholm 1916,
Almqvists and Wilksels, Uppsala, 1920.

\bibitem[Sem]{Se} S. Semmes. {\em  A criterion for the boundedness of singular integrals on
on hypersurfaces.} {Trans. Amer. Math. Soc.} { 311} (1989), 501--513.

\bibitem[Ste]{St} E.M. Stein. {Singular Integrals and Differentiability Properties
of Functions}. Princteon University Press, Princeton, NJ, 1970.

\bibitem[To]{Tolsa-llibre}
X.~Tolsa.
{\em Analytic capacity, the {C}auchy transform, and non-homogeneous
  {C}alder\'on-{Z}ygmund theory}, volume 307 of { Progress in Mathematics}.
 Birkh\"auser Verlag, Basel, 2014.

\end{thebibliography}
\end{document}